\documentclass[11pt,a4paper]{article}
\usepackage{bbm}
\usepackage{psfrag,epic,eepic,epsfig}
\usepackage{subfig}
\usepackage{color}
\usepackage[applemac]{inputenc} 
\usepackage{amsmath,amsfonts,amssymb,amsthm,mathrsfs, accents}
\usepackage{graphicx,graphics}
\usepackage{latexsym}
\usepackage{ae,aecompl}
 \usepackage[colorlinks=true]{hyperref}
\usepackage{enumerate}
\usepackage{comment}
\usepackage[normalem]{ulem}
\usepackage{xcolor}
\usepackage{empheq}
\usepackage[framemethod=tikz]{mdframed}
\usepackage{lipsum}
\usepackage{pifont}
\usepackage{fancybox} 
\usepackage{textcomp}
\usepackage{tikz}
\usepackage{graphicx}
\usepackage[margin=1cm,font=small]{caption}
\usepackage{array}
\usepackage{marvosym}
\usepackage{enumitem}
\usepackage{titlesec}

\definecolor{gris25}{gray}{0.75}

\definecolor{mycolor}{rgb}{0, 0, 0.1}

\newmdenv[innerlinewidth=0.5pt, roundcorner=4pt,linecolor=mycolor,innerleftmargin=6pt,
innerrightmargin=6pt,innertopmargin=6pt,innerbottommargin=6pt]{mybox}

\newcommand{\E}{\mathbb{E}}

\renewcommand{\subset}{\subseteq}

\newcommand{\R}{{\mathbb{R}}}
\newcommand{\T}{\mathcal{T}}

\newcommand{\s}{\mathcal{S}}
\newcommand{\p}{\mathcal{P}}
\newcommand{\N}{\mathbb{N}}

\newcommand{\Z}{\mathbb{Z}}

\newcommand{\veps}{\varepsilon}
\newcommand{\Var}{\mathrm{Var}}
\newcommand{\pr}{\mathbb{P}}

\newcommand{\cross}{\text{\Cross}}

\titleformat{\paragraph}
{\normalfont\normalsize\bfseries}{\theparagraph}{1em}{}
\titlespacing*{\paragraph}
{0pt}{3.25ex plus 1ex minus .2ex}{1.5ex plus .2ex}

\def\llbracket{[\hspace{-.10em} [ }
\def\rrbracket{ ] \hspace{-.10em}]}

\newcommand{\bnu}{\ensuremath{\boldsymbol{\bar{\nu}}}}

\renewcommand{\theequation}{\arabic{equation}}

\newtheorem{thm}{Theorem}[section]
\newtheorem{defn}[thm]{Definition}
\newtheorem{lem}[thm]{Lemma}
\newtheorem{prop}[thm]{Proposition}
\newtheorem{cor}[thm]{Corollary}

\newtheorem{rem}[thm]{Remark}
\date{}

\title{\bf{\textsc{Scaling limits of multi-type Markov Branching trees}}}

\author{B\'en\'edicte Haas\thanks{ Universit\'e Paris 13, LAGA, CNRS UMR 7539, 93430 Villetaneuse, haas@math.univ-paris13.fr}  \quad \& \hspace{0.2cm} Robin Stephenson\thanks{Department of Statistics, University of Oxford, 24-29 St Giles', Oxford OX1 3LB, UK, robin.stephenson@normalesup.org  \newline \text{} \quad  This work is partially supported by the ANR GRAAL ANR--14--CE25--0014 and the EPSRC Fellowship EP/N004833/1.}}

\begin{document}

\maketitle

\begin{abstract} We introduce multi-type Markov Branching trees, which are simple random population tree models where individuals are characterized by their size and their type and give rise to (size,type)-children in a Galton-Watson fashion, with the rule that the size of any individual is a least the sum of the sizes of its children. 
Assuming that the macroscopic size-splittings are rare, we describe the scaling limits of multi-type Markov Branching trees in terms of multi-type self-similar fragmentation trees. We observe three different regimes according to whether the probability of type change of a size-biased child is proportional to the probability of macroscopic splitting (the critical regime, in which we get in the limit multi-type fragmentation trees with indeed several types), smaller than the probability of macroscopic splitting (the solo regime, in which the limit trees are monotype as we never see a type change), or larger than the probability of macroscopic splitting (the mixing regime,  in which case the types mix in the limit and we get monotype fragmentation trees).

This framework allows us to unify models which may \emph{a priori} seem quite different, a strength which we illustrate with two notable applications. The first one concerns the description of the scaling limits of growing models of random trees built by gluing at each step on the current structure a finite tree picked randomly in a finite alphabet of trees, extending R\'emy's well-known algorithm for the generation of uniform binary trees to a fairly broad framework. We are then either in the critical regime with multi-type fragmentation trees in the scaling limit, or in the solo regime. The second application concerns the scaling limits of large multi-type critical Galton-Watson trees when the offspring distributions all have finite second moments. This topic has already been studied but our approach gives a different proof and we improve on previous results by relaxing some hypotheses. We are then in the mixing regime: the scaling limits are always multiple of the Brownian CRT, a pure monotype fragmentation tree in our framework.
\end{abstract}

\tableofcontents

\bigskip

\bigskip

\setlength\parindent{0pt}
\setlength\parskip{8pt}
\linespread{1.095}

\section{Introduction}
\label{Intro}

We consider population models where individuals are completely characterised by two parameters: their \emph{size}, which is a positive integer, and their \emph{type}, which is an integer in the finite set $[\kappa]:=\{1,\ldots,\kappa\}$, where $\kappa$ is a positive integer fixed throughout the paper. We say that this model is \emph{multi-type Markov Branching} (multi-type MB) if it is built recursively, generation by generation, with the rule that an individual of size $n$ and type $i$ gives birth, independently of other individuals of its generation and according to a distribution that only depends on $n$ and $i$, to a group of (size,type)-children whose sum of sizes is less than or equal to $n$.
The monotype setting ($\kappa=1$) has been investigated in several papers \cite{Ald96,BDMcS08,Ford05,HM12,HMPW08, Riz12}. The first examples of monotype MB trees are the well-studied monotype Galton-Watson trees conditioned to have a given number of vertices or leaves, or  important models in phylogenetics such as the Yule or comb models. However, this framework is in fact much broader, see the previous references and the survey \cite{HSurvey16} for other monotype examples. In \cite{HM12,HMPW08}, the scaling limits of monotype MB trees have been studied under a natural condition satisfied in many examples, namely  that an individual of size $n$ asymptotically gives rise to strictly more than one individual of macroscopic sizes (i.e.\ proportional to $n$) with a probability of order $n^{-\gamma}$ for some $\gamma>0$. Then, if $T_n$ denotes the tree of descendants of an individual of size $n$, where here size can mean either the total number of descendants or the total number of leaves, the rescaled tree $n^{-\gamma} \cdot T_n$ converges in distribution for the Gromov-Hausdorff-Prokhorov topology to a random compact real tree. Rizzolo \cite{Riz12} extends this result to trees with more general notions of sizes. The trees obtained in the limit belong to the family of \emph{fragmentation trees} introduced in \cite{HM04, Steph13}, which describe the genealogy of self-similar fragmentation processes as introduced by Bertoin \cite{BertoinSSF,BertoinBook}. This family includes in particular the Brownian CRT of Aldous and more generally the stable L\'evy trees of Duquesne-Le Gall and Le Jan -- in fact the scaling limit results mentioned above for MB trees allow us to recover well-known results by  Aldous \cite{Ald93} and Duquesne \cite{Duq03} on the convergence of rescaled Galton-Watson trees conditioned on their total progeny to a stable L\'evy tree. Other applications were developed in \cite{BerFire, HMPW08, HM12, Riz12}. We complete this picture by mentioning that recently Dadoun \cite{Dadoun17+} studied the scaling limits of MB tree models that incorporate growth, with connections with the theory of random maps, and Pagnard \cite{Pagnard17} studies the local limits of MB trees and their volume growth (see the references therein for other papers partly interested in  local limits of MB trees).  

The class of multi-type MB trees contains as first examples multi-type Galton-Watson trees conditioned to have a given number of vertices, or a given number of leaves. The scaling limits of multi-type Galton-Watson trees conditioned to have a given number of vertices of a fixed type have been first studied by Miermont \cite{M08} when the covariance matrix of the offspring distributions is finite, assuming furthermore some finite exponential moments.  A first extension has been made by Berzunza \cite{Berzunza18} who described the scaling limits of \emph{forests} of multi-type Galton-Watson trees with offspring distributions in the domain of attraction of stable laws.  A second extension has been made by de Raph\'elis \cite{deRaph17} who considers \emph{infinite} sets of types, a very delicate case. 

The purpose of this paper is to study the scaling limits of multi-type MB trees in a general and unifying setting, and develop applications illustrating the different facets of our framework. Naturally, one may expect that the scaling limits of multi-type MB trees are multi-type fragmentation trees. This family of trees has been introduced in \cite{Steph17} to describe the genealogy of multi-type self-similar fragmentations, also introduced in this paper. They are self-similar models which generalize the notion of homogeneous multi-type fragmentations earlier constructed by Bertoin \cite{Ber08}. Note that we are restricted here to finitely many types, but some different fragmentation models with infinitely many have also been studied recently by Duchamps~\cite{Duchamps19}.

We will observe different types of behavior in the scaling limit, depending on the dynamics of the model.
We will work under the following assumptions (here rough versions, see Section \ref{sec:main} for precise ones):
\vspace{-0.38cm}
\begin{itemize}
\item[(i)] the macroscopic size-splittings of an individual of size $n$ are rare and occur with a probability of order $n^{-\gamma}$ for all types, for some $\gamma>0$ (in fact, for the condition (ii) c) below this probability will more generally be allowed to be of order $O(n^{-\gamma})$ for all types with at least one type of order $n^{-\gamma}$)
\end{itemize}
\vspace{-0.52cm}
and
\vspace{-0.4cm}
\begin{itemize}
\item[(ii)] the probability of type change of a size-biased child of an individual of size $n$ is:
\begin{itemize}
\vspace{-0.1cm}
\item[a)] either of order $n^{-\gamma}$ for all types,
\vspace{-0.05cm}
\item[b)] or of order $o(n^{-\gamma})$ for all types,
\vspace{-0.05cm}
\item[c)] or of order $n^{-\beta}$ for some $\beta \in [0,\gamma)$, for all types. 
\end{itemize} 
\end{itemize} 
We could certainly relax some of those hypotheses to be in a more general framework, but at a significant cost of technicalities in which we do not want to embark as the interest in terms of applications is not clear. In all of our three cases, fixing a type $i$, we will have to rescale the tree of descendants of an individual of size $n$ and type $i$, by $n^{\gamma}$ to observe a non-trivial limit. 
We will then observe three different regimes, which can roughly be summed up as follows:
\begin{itemize}[leftmargin=2.7cm]
\item[under (ii) a),]the limiting tree is a multi-type fragmentation tree,
\item[under (ii) b),]the limiting tree is a monotype fragmentation tree whose dynamics are governed by the type $i$ alone,
\item[under (ii) c),]the limiting tree is a monotype fragmentation tree whose dynamics are governed by a mixture of contributions from all the types, via their stationary distribution appearing in the scaling limit.
\end{itemize}
The case (ii) b) is certainly the least interesting since in the $n^{\gamma}$ scale no type change is  observed asymptotically, and the study reduces to a purely monotype case (this case, though, is a slight extension of the results of \cite{HM12} since we include here a rather general notion of size). In general, for all cases, our proofs consist of exploring the tree starting from the root and evaluating the scaling limits of typical paths, starting from the path from the root to a typical vertex of the tree.  A key point is that this path, including the types of the vertices that compose it, is a bivariate Markov chain on $\mathbb Z_+ \times \{1,\ldots,\kappa\}$. In \cite{HS18} we have studied the scaling limits of such processes and we will use these results to show that here our typical path converges to a time-changed Markov additive process, which, roughly, is the typical path in a multi-type fragmentation tree. While this way of exploring the tree is  inspired by the monotype study of \cite{HM12}, we insist that the multi-type framework brings its own difficulties, notably with the need to deal with types at different scales.

We will then develop two notable applications illustrating our results in their different regimes. The first one concerns growing sequences of random trees that are built recursively by gluing at each step on a random edge of the current structure a random tree chosen in a finite alphabet of finite trees, thus generalizing R\'emy's algorithm \cite{Rem85} for the generation of random binary trees. This model may be seen as a MB model fitting in our cases (ii) a) or (ii) b), with all possible values of $\gamma>0$, depending on the average number of edges of the alphabet trees.
We will see that the scaling limit is a multi-type fragmentation tree, which has strictly more than one type, except when the alphabet is uniquely composed of the tree with a unique edge (it is well-known that then the limit is the Brownian CRT, a monotype fragmentation tree in our setting) or uniquely composed of star trees. The second application deals with critical multi-type Galton-Watson trees having offsprings with finite second moments. This is an illustration of our case (ii) c), with $\gamma=1/2$ and $\beta=0$. 
We will recover the results of Miermont \cite{M08}, under less restrictive assumptions, since we do not have to assume exponential moments. We emphasize that our proof, based on an exploration of the trees via typical genealogical paths is different from Miermont's one, based on the study of contour functions.  Mostly, we believe that this method could also be used to describe the scaling limits of multi-type Galton-Watson trees conditioned to have a given number of vertices (ideally also a given number of leaves, but this will require more work) with offspring distributions in the domain of attraction of stable laws. This has not been proved yet for a single tree, and would complete the work of Berzunza on multi-type Galton-Watson forests. This will be considered in future work. 

\textbf{Organization of the paper.} In Section \ref{sec:main}, after having introduced our discrete (Section \ref{sec:the model}) and continuous (Section \ref{sec:multi-type frag}) multi-type trees, we will expose our main theorems on the scaling limits of MB trees (Section \ref{sec:main results}). The proofs are postponed to Section \ref{sec:proofs1} and Section \ref{sec:proofs2}.
The applications to growing models of random trees and multi-type Galton-Watson trees with a finite second moments are developed in Section \ref{sec:App1} and Section \ref{sec:App2} respectively.

\section{Multi-type Markov Branching trees and their scaling limits}
\label{sec:main}

We emphasize that throughout the paper, all the discrete trees that we consider are rooted, unordered, unlabelled and that they may be interpreted as metric spaces, equipped with the graph distance.

\subsection{The model: multi-type MB trees}
\label{sec:the model}

\textbf{Discrete typed partitions.} For $n\in\N$, we call \emph{$\kappa$-type partition of $n$} any finite $(\N\times[\kappa])$-valued sequence of the form $$\bar{\lambda}=(\lambda,\mathbf{i})=\big((\lambda_1,i_1),\ldots,(\lambda_{p(\bar{\lambda})},i_{p(\bar{\lambda})})\big)$$ such that $p(\bar{\lambda})$ is a nonnegative integer -- the length of the partition -- and:
\begin{itemize}
\item[(i)] $\sum_{m=1}^{p(\bar{\lambda})} \lambda_m\leq n$;
\item[(ii)] the sequence is lexicographically nonincreasing: for all $m\leq p(\bar{\lambda})-1$, either $\lambda_{m+1}<\lambda_m$ or $\lambda_{m+1}=\lambda_m$ and $i_{m+1}\leq i_m$.
\end{itemize}
We then let $\overline{\p}_n$ be the set of $\kappa$-type partitions of $n.$ By convention, the empty sequence $\emptyset$ is an element of $\overline{\p}_n$, with length $p(\emptyset)=0$, corresponding to the situation when an individual has no children. 

\textbf{Splitting distributions and associated MB trees.} In a multi-type MB model the children of any individual with size $n$ can be sorted into an element of $\overline{\p}_{n},$ and thus the offspring distributions are probability measures on $\overline{\p}_{n}.$ We call the offspring distributions of MB models \emph{splitting distributions}, to emphasize the fact that the size of a parent is spread out in its children. Let $q_n^{(i)}$ be the splitting distribution of an individual with type $i\in[\kappa]$ and size $n\in\N.$ 
Then for all $i\in[\kappa]$ and $n\in\N$, we denote by  $T_n^{(i)}$ the family (rooted, unordered, unlabelled) tree  of the population started at an individual with size $n$ and type $i$, and call it a \emph{$\kappa$-type MB tree}. Formally, it is a multi-type Galton-Watson tree with type set $\N\times [\kappa],$ where the offspring of an individual with characteristics $(m,j)\in\N\times [\kappa]$ has distribution $(q_m^{(j)}, m \leq n, j \in [\kappa])$. To guarantee that $T_n^{(i)}$ is finite a.s., we will always make the following assumption: 
\begin{eqnarray}
\label{hyp:principal}
&&\hspace{-0.9cm}\text{For all } n\in\N \text{ and }i\in[\kappa], \\
&&\hspace{-0.3cm} \bullet \text{ either }q_n^{(i)}\left(\{\emptyset\}\cup\{\bar{\lambda}\in\overline{\p}_n:\;\lambda_1<n\}\right)>0 \nonumber \\ 
&&\hspace{-0.3cm} \bullet \text{ or there exists a type }j\neq i \text{ satisfying the previous point and a path }i_1=i,i_2,\ldots,i_p=j \nonumber \\
&& \hspace{-0.1cm} \text{ such that } 
q_n^{(i_l)}(\{(n,i_{l+1})\})>0 \text{ for } l=1,\ldots,p-1. \nonumber
\end{eqnarray}
Note that this implies that $q_n^{(i)}(\{\emptyset\})>0$ for at least one integer $n$ and one type $i$.

\emph{A probability measure on $T_n^{(i)}$.} The tree $T_n^{(i)}$ comes with a natural probability measure which we will call $\mu_n^{(i)},$ defined thus: for every vertex $u$ in the tree with size $k\leq n$, let $\bar{\lambda}$ be the list of sizes and types of its children. Then, if $\sum_{m=1}^{p(\bar{\lambda})} \lambda_m <k$, put at $u$ an atom with mass 
$$\mu_n^{(i)}(\{u\}):=\frac{k-\sum_{m=1}^{p(\bar{\lambda})} \lambda_m}{n}.$$ 
By definition, we see that the subtree rooted at $u$ then has mass $k/n$; in particular $\mu_n^{(i)}$ is a probability measure.

\emph{Conservative cases.} For $i\in [\kappa]$ and $n\geq 2$ we say that the probability $q_n^{(i)}$ is \emph{conservative} if
$$q_n^{(i)}\Bigg(\sum_{m=1}^{p(\bar{\lambda})}\lambda_m=n\Bigg)=1.$$ 
When the measures $q_n^{(i)}$ are conservative for all $n\geq 2$ and $i\in [\kappa]$, the tree $T_n^{(i)}$ has $n$ leaves and the measure $\mu_n^{(i)}$ is uniformly supported on its set of leaves, $\forall n\geq 1$. 

\begin{rem} In some applications, we will need to allow individuals to have size $0$. This difference will be treated on a case-by-case basis.
\end{rem}

\subsection{Multi-type fragmentation trees}\label{sec:deftree}
\label{sec:multi-type frag}

We now present some background on multi-type self-similar fragmentation trees as constructed in \cite{Steph17}.

\textbf{Continuous typed partitions.} We let $\overline{\s}^{\downarrow}$ be the set of sequences of the form $$\bar{\mathbf{s}}=(\mathbf{s},\mathbf{i})=(s_n,i_n)_{n\in\N}\in ([0,1]\times \{0,1,\ldots,\kappa\})^{\N}$$ such that:
\begin{itemize}
\item[(i)] $\sum_{n\in\N} s_n\leq 1$; 
\item[(ii)] for all $n\in\N$, $i_n=0$ if and only if $s_n=0;$
\item[(ii)] the sequence is lexicographically nonincreasing: for all $n\in\N$, either $s_{n+1}<s_n$ or $s_{n+1}=s_n$ and $i_{n+1}\leq i_n.$
\end{itemize}
An element of $\overline{\s}^{\downarrow}$ should be seen as a finite or countable set of particles with masses $(s_n,n\in\N)$ and types $(i_n,n\in\N).$ We do not allow for particles with mass $0$: $s_n=0$ for some $n$ means that there is no $n$-th particle at all, and so it is matched with the placeholder type $0$. 
We also define $s_0:=1-\sum_{n \in \mathbb N} s_n$. Following \cite{Ber08}, $\overline{\s}^{\downarrow}$ is compactly metrised by letting, for two partitions $\bar{\mathbf{s}}$ and $\bar{\mathbf{s}}',$ $d(\bar{\mathbf{s}},\bar{\mathbf{s}}')$ be the Prokhorov distance between the two measures 
\begin{equation}
\label{def:metric}
s_0\delta_{\mathbf 0}+\sum_{n=1}^{\infty}s_{n}\delta_{s_n\mathbf{e}_{i_n}} \quad \text{ and } \quad s'_0\delta_{\mathbf 0}+\sum_{n=1}^{\infty}s'_{n}\delta_{s'_n\mathbf{e}_{i'_n}}
\end{equation} 
on the $\kappa$-dimensional unit cube (where $(\mathbf{e}_i,i\in[\kappa])$ is the canonical basis of $\R^{\kappa}$). We emphasize that the functions $\bar{\mathbf s} \mapsto \sum_{n\geq n_0} s_n^q$ and $\bar{\mathbf s} \mapsto \sum_{n\geq n_0} s_n^q \mathbf 1_{\{i_n=i\}}$ are continuous on $\overline{\s}^{\downarrow}$ when $q>1$, for each integer $n_0$ and each type $i \in [\kappa]$. 
This is not the case of the functions $\bar{\mathbf s} \mapsto \sum_{n\geq n_0} s_n$ and $\bar{\mathbf s} \mapsto \sum_{n\geq n_0} s_n \mathbf 1_{\{i_n=i\}}$, which are however continuous at each point $\bar{\mathbf s}$ such that $\sum_{i=1}^{\infty} s_i=1$ (by Fatou's lemma). These continuity properties will be used regularly throughout the paper.

\textbf{Dislocation measures and multi-type fragmentation trees.} Let $\gamma>0$ and $\boldsymbol{\bar{\nu}}=(\bar{\nu}^{(i)},i\in[\kappa])$ be a vector of $\sigma$-finite measures on $\overline{\s}^{\downarrow}.$ We call them \emph{dislocation measures} if they also satisfy these four conditions for all $i$:
\begin{itemize}
\item[(i)] $\bar{\nu}^{(i)}\left(\sum_{n\in \mathbb N} s_n <1 \right)=0$;
\item[(ii)] $\bar{\nu}^{(i)}\big((1,i),(0,0),\ldots\big)=0;$
\item[(iii)] $\bar{\nu}^{(i)}(s_1<1)>0$; 
\item[(iv)] $\int_{\overline{\s}^{\downarrow}} (1-s_1\mathbf{1}_{\{i_1=i\}})\bar{\nu}^{(i)}(\mathrm d\bar{\mathbf{s}}) < \infty.$
\end{itemize}
In this case, for all $\gamma>0$, one can construct a continuous-time population with individuals characterised by a mass $x\in(0,1]$ and a type $j\in[\kappa]$ such that we start with a single individual $(1,i)$ and an individual with characteristics $(x,j)\in (0,1]\times [\kappa]$ splits into individuals with characteristics $\left((xs_m,i_m),m\in\N\right)$ at rate $$x^{-\gamma}\bar{\nu}^{(i)}(\mathrm d \bar{\mathbf{s}}).$$
Note that Condition (i) corresponds to a \emph{conservation property} -- no mass is lost when an individual splits -- that we adopt here for the sake of simplicity. Condition (iv) is necessary for our process not to vanish immediately, and ensures that both the rate of type change of the largest fragment from a split and the rate of macroscopic splittings are finite.
It is then possible to build the family tree of this genealogy, which is a compact real tree denoted by $\T_{\gamma,\bnu}^{(i)},$ called the \emph{multi-type fragmentation tree} with index of self-similarity $\gamma$ and dislocation measures  $(\bar{\nu}^{(i)},i\in[\kappa])$. We refer to Section \ref{sec:construcmultMB} for a precise definition of such a tree and more generally to \cite[Section 3]{Steph17} for details. It is moreover naturally equipped with a probability measure $\mu_{\gamma,\bnu}^{(i)},$ which is supported by its set of leaves, a set with Hausdorff dimension equal to $1/\gamma$ a.s under some mild additional assumptions.

When $\kappa=1$, the tree $(\T_{\gamma,\bnu}^{(1)},\mu_{\gamma,\bnu}^{(1)})$ is a monotype fragmentation tree, as introduced in \cite{HM04} (see also \cite{Steph13} for the non-conservative models).  In fact, it is easier in this setting to use simpler notations that do not refer to types, and we let $\mathcal S^{\downarrow}$ denote the set of nonincreasing sequences  $(s_n,n\geq 1)$ such that $s_n\geq 0$ for all $n$ and $\sum_{n \in \mathbb N}s_n\leq 1$, equipped with distance $\sup_{n}|s_n-s'_n|$ for $\mathbf s,\mathbf s' \in \mathcal S^{\downarrow}$. In this context, a dislocation measure is a measure $\nu$ on $\mathcal S^{\downarrow}$ such that $\nu(\sum_{n \in \mathbb N}s_n<1)=\nu(s_1=1)=0$ and $\int_{\mathcal S^{\downarrow}} (1-s_1)\nu(\mathrm d{\mathbf{s}})<\infty$, and we will thus denote this tree by $(\T_{\gamma,{\nu}},\mu_{\gamma,{\nu}})$.
A key example that will appear several times in our results is the Brownian CRT of Aldous, denoted here by $\T_{\mathrm{Br}}$. We recall from Bertoin \cite{BertoinSSF} that equipped with its ``uniform" measure $\mu_{\mathrm{Br}}$ it is a  fragmentation tree with index of self-similarity $\gamma=1/2$ and dislocation measure $\nu_{\mathrm{Br}}$, which is binary (i.e.\ it only charges sequences such that $s_3=0$),  conservative and such that
\begin{equation}
\label{Brdisloc}
\nu_{\mathrm{Br}}(s_1\in \mathrm{d}x)=\sqrt{\frac{2}{\pi x^3(1-x)^3}}\mathbf{1}_{\{1/2\leq x < 1\}}\mathrm{d}x.
\end{equation}

\subsection{Main results: scaling limit theorems}
\label{sec:main results}

We consider a family of offspring distributions $(q_n^{(i)})_{(n, i) \in \mathbb N \times [\kappa]}$ satisfying (\ref{hyp:principal}) and an associated sequence of multi-type MB measured trees  $$(T_n^{(i)}, \mu_n^{(i)})_{(n, i) \in \mathbb N \times [\kappa]}.$$ We underline that there is no notion of joint distribution for these MB measured trees, that may be constructed on different probability spaces. In order to describe their scaling limits (in distribution) we make some further assumptions on the sequence $(q_n^{(i)})$ that involve two parameters: a real number $\gamma$ and a vector of dislocation measures. 

\medskip

\noindent\textbf{Possible values of} $\gamma$. If, for all $i\in[\kappa]$ and all $n$ large enough, the probability $q_n^{(i)}$ is conservative, then consider any $\gamma>0$. If the above is not true, we restrict ourselves to $0<\gamma<1.$ This holds for both theorems below.

As is now standard, we endow the trees $T_n^{(i)},n\geq 1,i \in [\kappa]$ with the graph distance, and consider the Gromov-Hausdorff-Prokhorov (GHP) topology on the set of equivalent classes of measured compact metric spaces. See e.g. \cite{ADH, ABBGM17} for background on this topology. 

We can now formalize the results announced in the introduction, by considering three situations where we compare the probabilities of macroscopic splittings (we informally say that the splitting of an individual of size $n$ is macroscopic if its largest child has size less than $n(1-\varepsilon)$ for some $\varepsilon>0$) with the probability of type change of a size-biased child of an individual of size $n$. In the first situation, called the \emph{critical case}, macroscopic splittings of individuals of size $n$ and any type happen at rate $n^{-\gamma},$ which is also the order of magnitude of the probability that a size-biased child of such an individual has a different type than its parent. We then observe in the limit a multi-type fragmentation tree with index $\gamma$ and whose dislocation measures are the scaling limits of the offspring distributions, in the sense of (\ref{hypocvcritique}) below. In the second situation, called \emph{solo case}, macroscopic splits still occur with a probability of order $n^{-\gamma}$ whereas the probability of type change of a size-biased fragment is a $o(n^{-\gamma}),$ and then the limit is a monotype fragmentation tree, whose type is technically that of its root. These two situations will be often considered together,
as opposed to the third situation, called the \emph{mixing regime}. There, macroscopic splits happen with a probability which is a $O(n^{-\gamma})$ for all types, with at least one type realising the bound, and the probability of type change of a size-biased fragment is larger, specifically of order $n^{-\beta}$ for some $\beta<\gamma$, and we observe in the limit a monotype fragmentation tree with index $\gamma$ and a dislocation measure which is a mixture of monotype dislocation measures appearing as scaling limits of the offspring distributions in the sense of (\ref{hypocvmixing}). The scaling factors in this mixture are given by the stationary distribution of the Markov chain describing the asymptotic evolution of the types. All of these results are formally stated as follows.

\medskip

\noindent\textbf{Critical and solo regimes.} For $n\in\N$, $i\in[\kappa]$ we let $\bar{\nu}_n^{(i)},$ a probability distribution on $\overline{\s}^{\downarrow},$ be the distribution of $$\left(\left(\frac{\Lambda_1}{n},i_1\right),\ldots,\left(\frac{\Lambda_{p(\bar{\lambda})}}{n},i_{p(\bar{\Lambda})}\right),(0,0),\ldots\right)$$ if $\bar{\Lambda}$ has distribution $q_n^{(i)}.$
\begin{thm}\label{theo:critique}
Assume that, for all $i\in[\kappa]$, we have the following weak convergence of measures
\begin{equation}
\label{hypocvcritique}
n^{\gamma}\left(1-s_1\mathbf{1}_{\{i_1=i\}}\right)\bar{\nu}_n^{(i)}(\mathrm d \bar{\mathbf{s}}) \underset{n\to\infty}\longrightarrow \left(1-s_1\mathbf{1}_{\{i_1=i\}}\right)\bar{\nu}^{(i)}(\mathrm d \bar{\mathbf{s}})\end{equation}
where $\bnu=(\bar{\nu}^{(i)},i\in[\kappa])$ is a vector of dislocation measures on $\overline{\s}^{\downarrow}$ which satisfies one of the following: 
\begin{equation}\label{hypocvcritique2} \forall i\in[\kappa], \quad \bar{\nu}^{(i)}\left(\exists k\in \mathbb N, s_k>0\text{ and } i_k\neq i\right)>0.\end{equation}
or
\begin{equation}\label{hypocvcritique3} \forall i\in[\kappa], \quad \bar{\nu}^{(i)}\left(\exists k\in \mathbb N, s_k>0\text{ and } i_k\neq i\right)=0.\end{equation}
We then have the following convergence in distribution of measured metric spaces for the GHP-topology, for all $i\in[\kappa]$:
\[\big({n^{-\gamma}} \cdot T_n^{(i)}, \mu_n^{(i)}\big) \overset{(d)}{\underset{n\rightarrow \infty}{\longrightarrow}} \big(\T^{(i)}_{\gamma,\bnu}, \mu^{(i)}_{\gamma,\bnu}\big).
\]
\end{thm}

Hypothesis (\ref{hypocvcritique}) means in particular that, for Lebesgue-almost-every $\varepsilon>0$, the probability that an individual with size $n$ and type $i$ gives a sequence of children with largest size smaller than $n(1-\varepsilon)$ is asymptotically proportional to
$n^{-\gamma} \bar{\nu}^{(i)}(s_1 \leq 1-\varepsilon)$, where $\bar{\nu}^{(i)}(s_1 \leq 1-\varepsilon)$ is finite by definition. If we combine it with assumption (\ref{hypocvcritique2}) then we are in the critical regime mentioned earlier. Indeed, the probability that a size-biased fragment changes its type behaves asymptotically as 
$$\sum_{\lambda\in\overline{\p}_n}q_n^{(i)}(\bar{\lambda})\sum_{m=1}^{p(\bar{\lambda})} \frac{\lambda_m}{\sum_{k=1}^{p(\bar{\lambda})}\lambda_k} \mathbf{1}_{\{i_m \neq i\}} \underset{n \rightarrow \infty}\sim n^{-\gamma}\int_{\bar \s^{\downarrow}}\sum_{m=1}^{\infty}s_m \mathbf{1}_{\{i_m \neq i\}}\mathrm d\bar{\nu}^{(i)}(\bar{\mathbf{s}}),$$ 
where the integral is finite by definition of a dislocation measure. Meanwhile, if we combine hypotheses (\ref{hypocvcritique}) and (\ref{hypocvcritique3}) then the probability that a size-biased fragment changes its type is $o(n^{-\gamma}),$ placing us in the solo case.
We believe that Theorem \ref{theo:critique} would still hold without either of (\ref{hypocvcritique2}) or (\ref{hypocvcritique3}), putting us in an intermediate regime where some, but not all of the types would act as dead ends. In order to lighten this already long article, we do not develop this case which requires a certain number of technicalities and is not the most interesting in terms of applications.

\emph{The monotype case.} We emphasize that Theorem \ref{theo:critique} includes the monotype cases (assuming (\ref{hypocvcritique}) for the unique type involved, then  (\ref{hypocvcritique3}) is automatically satisfied) and that it then generalizes slightly the results obtained in \cite{HM12} and \cite{Riz12} by allowing more general notions of sizes.

\bigskip

\noindent\textbf{Mixing regime.} Let $0\leq \beta < \gamma$. For $n\in\N$, let $P_n$ be the matrix defined by
$$P_n(i,j)=q_n^{(i)} (i_1=j) \quad \text{for all types } i,j.$$ 
Let also, for $n\in\N$ and $i\in[\kappa]$, $\nu_n^{(i)}$ be the probability measure on $\s^{\downarrow}$ which is the distribution of 
$$\left(\frac{\Lambda_1}{n},\ldots,\frac{\Lambda_{p(\bar{\Lambda})}}{n},0,\ldots\right),$$ without the types, if $\bar{\Lambda}$ has distribution $q_n^{(i)}.$ In our theorem below, we use the notion of $\mathsf Q-$matrix on $[\kappa]$: we recall that it is a $\kappa\times \kappa$ matrix such that the diagonal coefficients are nonpositive, the coefficients outside the diagonal are nonnegative and the sum of each line is 0.

 \begin{thm}\label{theo:melange}
Assume that, for all $i\in[\kappa],$ we have the following weak convergence of measures on $\s^{\downarrow}$:
\begin{equation}
\label{hypocvmixing}
n^{\gamma}(1-s_1)\nu_n^{(i)}(\mathrm d \mathbf{s}) \underset{n\to\infty}\longrightarrow (1-s_1)\nu^{(i)}(\mathrm d \mathbf{s})
\end{equation}
where the $\nu^{(i)},i\in[\kappa]$ are either dislocation measures on ${\s}^{\downarrow}$ (which by definition cannot be the null measure) or the null measure on ${\s}^{\downarrow}$, such that at least one of these measures is not null. Assume also the following convergence of matrices:
\begin{equation}
\label{hypocvmixing2}
n^{\beta}(P_n-I) \ {\underset{n \rightarrow \infty}\longrightarrow} \ Q
\end{equation}
where $Q$ is an irreducible $\mathsf Q$-matrix on $[\kappa]$, and $\chi=(\chi_1,\ldots,\chi_\kappa)$ is the corresponding unique invariant probability measure. We then have the following convergence in distribution of measured metric spaces for the GHP-topology, for all $i\in[\kappa]$:
\[\big(n^{-\gamma} \cdot T_n^{(i)},\mu_n^{(i)}\big) \overset{(d)}{\underset{n \rightarrow \infty}\longrightarrow} \big(\T_{\gamma,\nu},\mu_{\gamma,\nu}\big)
\]
where $\big(\T_{\gamma,\nu},\mu_{\gamma,\nu}\big)$ is a monotype fragmentation tree, with dislocation measure $\nu=\sum_{i=1}^{\kappa}  \chi_i\nu^{(i)}$.
\end{thm}

Note that assumption (\ref{hypocvmixing}) guarantees that macroscopic splittings happen with a probability of order $n^{-\gamma}$ (or lower, for $i$ such that $\nu^{(i)}$ is null) and hypothesis (\ref{hypocvmixing2}) places us in the mixing regime as informally defined earlier. Indeed, while it states that the largest fragment from a split changes type with probability of order $n^{-\beta},$ we can check that this then also holds for a size-biased fragment, see Lemma \ref{lem:hypok2}.

The proofs of these two theorems will be developed in Section \ref{sec:proofs1} and Section \ref{sec:proofs2}. We emphasize that the proofs are more involved technically in the mixing regime when at least one measures $\nu^{(i)}$ is null, since some macroscopic splittings are then significantly slower than the others. This case, however, fully deserves our interest: it is in particular encountered for finite variance multi-type Galton-Watson trees if some types may only give birth to one child, or multi-type Galton-Watson trees whose offspring distributions are in the domain of attraction of stable laws, with varying indices depending on the type of the parent.

\section{Application 1: Growing models of random trees}
\label{sec:App1}

Throughout this section we consider an alphabet of two finite rooted trees to keep things simple, but we emphasize that all the results extend readily to a \emph{finite} alphabet of finite rooted trees. Our two-tree case also include the single tree case when $\tau_A=\tau_B$, with the notation below.

We will consider a growing sequence of random trees obtained by drawing recursively a tree in the set $\{\tau_A,\tau_B\}$, where  $\tau_A,\tau_B$ are two rooted trees with respectively $n_A \geq 0$ and $n_B \geq n_A$ 
edges. When $n_A=0$, \ $\tau_A$ is the tree with a single vertex.
The probability to choose $\tau_A$ and $\tau_B$ are respectively denoted by $q_A \in [0,1]$ and $q_B:=1-q_A$, and 
throughout this section
\begin{equation}
\label{genericrv}
N \text{ denotes  a random variable with distribution } q_A \delta_{n_A}+q_B\delta_{n_B},
\end{equation}
that corresponds to the number of edges of the randomly chosen tree. 
Our sole assumption on the parameters is that the expectation of this number of edges  is nonzero:
$$\mathbb E[N]:=q_An_A+q_Bn_B>0.$$ 
Let then $T_0$ be a rooted planted tree (i.e.\ the root has degree 1) with $n_0\geq 1$ edges, which will be our starting point. The sequence $(T_n,n \geq 0)$ is built recursively starting from $T_0$ by: 
\begin{itemize}
\item[(i)] choosing at step $n$ an edge uniformly in $T_n$ 
\item[(ii)] gluing on this edge a random tree which is equal to $\tau_A$ with probability  $q_A$ and to  $\tau_B$ with probability $q_B$, independently of everything else: the gluing is done by inserting the root of the random tree ``in the middle" of the selected edge of $T_n$. This gives $T_{n+1}$.
\end{itemize}
We will call the successive random trees used for this construction the \emph{brick trees}. 

Our aim is to describe the scaling limit of $(T_n)$ in terms of multi-type fragmentation trees. This will be achieved by the using the underlying MB structure of these trees. In some particular cases, the scaling limit is already known. When $\tau_A=\tau_B=T_0$ is the tree with a single edge, this procedure is known as \emph{R\'emy's algorithm} \cite{Rem85} and generates a sequence of trees distributed as planted binary Galton-Watson trees conditioned to have $n+1$ leaves, $n\geq0$, whose planar order has been forgotten. 
In this case, it is well known that $n^{-1/2} \cdot T_n$ converges almost surely in the GHP-topology to a multiple of the Brownian CRT (see \cite{Ald93} for the convergence in distribution and e.g. \cite{CH13} for a proof of the almost sure convergence). This result was extended in \cite{HS14} to  the case when $\tau_A=\tau_B$ is a star tree with $k$ edges (and $k+1$ vertices) and $T_0$ is the tree with a single edge: then, $n^{-\frac{1}{k+1}} \cdot T_n$ converges to a fragmentation tree with an infinite dislocation measure constructed from a Dirichlet distribution. This convergence has only been proved in \emph{probability} in \cite{HS14}, but the theorem below shows that it is in fact almost sure.

We introduce some notation. For each $v \in \tau_A$, each $v \in \tau_B$ and each $v \in T_0$, let  $\tau_v$ be a planted version of the subtree of the descendants of $v$, including $v$ (by \emph{planted version} we mean that an edge is attached to $v$ and that the other extremity of this edge is the new root). Then let
$$
\mathcal B=\left\{\tau_{v},v \in \tau_A\backslash{\{\rho_A\}}\cup \tau_B\backslash{\{\rho_B\}} \cup T_0\backslash\{{\rho_0}\}\right\}\
$$
where $\rho_A,\rho_B,\rho_0$ are the respective roots of $\tau_A,\tau_B,T_0$. Note that the tree with a unique edge always belongs to $\mathcal B$ (it is generated by the leaves of $\tau_A, \tau_B, T_0$) and that $T_0 \in \mathcal B$ since it is planted.

\begin{thm}
\label{theo:recursif}
Let $\mu_n$ be the uniform probability on the vertices of $T_n$. We have the following almost sure convergence in the GHP-topology:
\begin{equation}
\label{cvGT}
\big(n^{-\frac{1}{\mathbb E[N]+1}} \cdot T_n,\ \mu_n \big)\  \overset{\mathrm{a.s.}}{\underset{n \rightarrow \infty}{\longrightarrow}}\ \left(\mathcal T, \mu \right)
\end{equation}
where the limit is a multi-type fragmentation tree with \ $\# \mathcal B$ \ types, with index of self-similarity $(\mathbb E[N]+1)^{-1}$ and with dislocation measures denoted $\bar \nu_{\mathrm{growth}}^{(i)}$, $i \in [\# \mathcal B]$, which are constructed from limits of urn models depending on all the parameters of the model, $\tau_A,\tau_B,p_A,p_B,T_0$, see Definition \ref{def:randdisloc2}. When $n_A=n_B$, these measures are mixtures of biased-Dirichlet-distributions. With the notation below, the type of the root of $\mathcal T$ is 1.
\end{thm}

Note that the fragmentation tree $(\mathcal T,\mu)$ is monotype if and only if $\# \mathcal B=1$, i.e.\ if and only if $T_0, \tau_A,\tau_B$ are star trees (since $T_0$ is planted, it is thus necessarily the tree with a single edge).
The proof of this theorem relies mainly on our Theorem \ref{theo:critique}, which gives a convergence in distribution and allows us to identify the limit as a multi-type fragmentation tree. However, clearly, the recursive construction on a common probability space induces a stronger convergence. The most subtle point to get this stronger convergence is to establish the almost sure compactness. This is done by S\'enizergues \cite{S18} who gives a sufficient condition for recursive constructions of graphs to converge almost surely in the scaling limit. We emphasize that his result includes our setting here, hence proving the a.s.\ scaling limit of the trees $T_n$, but that it does not identify the limit as a multi-type fragmentation tree, but rather as a gluing of random metric spaces as studied in \cite{S17}. His sufficient condition can be stated as follows in our setting: for $N_i,i\geq 1$ a sequence of i.i.d. random variables distributed as $N$ (that represents the successive numbers of edges of the brick trees), we have almost surely the existence of some $\varepsilon>0$ and $c<(\mathbb E[N]+1)^{-1}$ such that
$$
\sum_{i=1}^n N_i=\mathbb E[N]n (1+O(n^{-\varepsilon})) \quad \text{and} \quad N_i \leq i^{c+o(1)}.$$
This is obviously true for $\varepsilon<1/2$ and $c=0$ since the random variable $N$ is deterministically bounded.

We need some more vocabulary and notation. First, noting that most of the trees we work with are planted, we decide to call \emph{ancestor} of a planted tree the unique child of its root. Second, we rewrite $\mathcal{B}$ as
$$
\mathcal B=\left\{\tau_j, j \in [\# \mathcal B]\right\},
$$
for any ordering such that $\tau_1:=T_0$. For each $j$, we let $n_j$ denote the number of edges of $\tau_j$ and we give type $j$ to the ancestor of $\tau_j.$ We let $p_j$ be the out-degree of this ancestor. If $p_j\geq 1$, the $p_j$ planted subtrees descending from this ancestor are themselves in $\mathcal B$ (these subtrees are planted, i.e.\ we include for each of them the edge adjacent to the ancestor). We let $\mathbf n_j:=(n_{j,1},\ldots,n_{j,p_j})$ denote their sequence of number of edges and $\mathbf i_j:=(i_{j,1},\ldots,i_{j,p_i})$ the sequence of types of their respective ancestors, these sequences being indexed so that $(\mathbf n_j,\mathbf i_j)$ is in $\overline{\p}_n$. If $p_j=0$, $(\mathbf n_j,\mathbf i_j):=\emptyset$. Similarly, in $\tau_A$, we let $p_A$ denote the out-degree of the root and, if $p_A\geq 1$, $\mathbf n_A:=(n_{A,1},\ldots,n_{A,p_A})$ the sequence of the number of edges of the $p_A$ subtrees descending from the root and $\mathbf i_A:=(i_{A,1},\ldots,i_{A,p_A})$ the corresponding sequence of types of the  ancestors of the subtrees, so that $(\mathbf n_A,\mathbf i_A)$ is in $\overline{\p}_n$. If $p_A=0$, $(\mathbf n_A,\mathbf i_A):=\emptyset$. We use similar notation for $\tau_B$. 

\bigskip

\textbf{Organization of the rest of the section.} Below we first describe the family of dislocation measures involved in the limiting tree and then proceed to the proof. It is split into two parts: the verification of the MB property, with an adequate notion of size, and then the verification of the criterion (\ref{hypocvcritique}) for the corresponding splitting distributions (the fact that the limiting dislocation measures satisfy (\ref{hypocvcritique2}) when there are more than two types ($\# \mathcal B \geq 2$) and (\ref{hypocvcritique3}) otherwise follows easily from their definition). In a last part we will discuss the case where $T_0$ has a root degree larger than 2, to which Theorem \ref{theo:recursif} can easily be adapted.

\subsection{The dislocation measures} 
\label{sec:disloc}

 The dislocation measures $\bar{\nu}^{(i)}_{\mathrm{growth}}, i \in [\# \mathcal B]$, are built from distributions appearing as scaling limits of urn models, which we first review.

\subsubsection{Background on asymptotics of urn models} 

\noindent \textbf{Classical P\'olya urns.} Consider an urn model with $k\geq 2$ colors and initial weights $a_1,\ldots,a_k>0$ respectively. At each step draw a color with a probability proportional to its weight and add a weight $\beta>0$ to this color. Let $W_{n,1},\ldots,W_{n,k}$ denote the weights of the $k$ colors after $n$ steps. Then
$$
(\beta n)^{-1} \cdot \big(W_{n,1},\ldots,W_{n,k}\big)  \underset{n \rightarrow \infty}{\overset{\mathrm{a.s.}}\longrightarrow} \big(W_1,\ldots,W_k\big)
$$ 
where $(W_1,\ldots,W_k)$ follows a Dirichlet $\mathrm{Dir}(a_1/\beta, \ldots, a_k/\beta)$ distribution. 

\bigskip

\noindent \textbf{P\'olya urns with random increments and random initial weights.} We still start with $k$ colors, but now the case $k=1$ is included, and initial weights $a_1,\ldots,a_k>0$, that can possibly be random. We moreover assume that the increments are random and deterministically bounded: at step $i$ draw a color with a probability proportional to its weight and add a weight $\beta_i>0$ to this color, where the $\beta_i,i\geq 1$ are i.i.d. deterministically bounded and independent of $(a_1,\ldots,a_k)$. Then if we let $W_{n,1},\ldots,W_{n,k}$ denote the weights of the $k$ colors after $n$ steps 
$$
\big(\textstyle \sum_{i=1}^kW_{n,i} \big)^{-1} \cdot\big(W_{n,1},\ldots,W_{n,k}\big)  \underset{n \rightarrow \infty}{\overset{\mathrm{a.s.}}\longrightarrow} \big(W_1,\ldots,W_k \big)
$$ 
where the limit is a random variable on the $k-1$ dimensional simplex. The existence of the limit is easy to see since for each $1 \leq i \leq k$, letting $B_n$ denote the total weight after $n$ steps, $(W_{n,i}/B_n)_n$ is clearly a bounded martingale, but its distribution is not explicit as in the balanced case. Note that the assumption that the  $\beta_i,i\geq 1$ are deterministically bounded is not needed to get this convergence, but it is important for the following property, due to Pemantle \cite{Pem90}: almost surely,
\begin{equation}
\label{diff}
W_i>0, \text{ for all } 1 \leq i \leq k, \quad \text{and} \quad W_i\neq W_j, \text{ for all }1 \leq i \neq j \leq k.
\end{equation}
In this paper, the random increments will always be distributed as $N+1$, with $N$ defined in (\ref{genericrv}). We will then denote the distribution of the limit $(W_1,\ldots,W_k)$ by
\begin{equation}
\label{def:urn}
\mathrm{Urn}_{N+1}(\mathbf a)
\end{equation}
where $\mathbf a=(a_1,\ldots,a_k)$ is the initial sequence of weights.

In fact, we will need the following strengthening of the above convergence. If $N_{n,i}$ denotes the number of times the color $i$ has been drawn until step $n$, then
\begin{equation}
\label{cvurnR}
n^{-1} \cdot\big(N_{n,1},\ldots,N_{n,k}\big)  \underset{n \rightarrow \infty}{\overset{\mathrm{a.s.}}\longrightarrow} \big(W_1,\ldots,W_k \big).
\end{equation}
This is an easy consequence of the above convergence and the following lemma, due to Dubins and Freedman \cite{DF65}.

\begin{lem}
\label{lem:DF}
Let $(\mathcal F_n)_{n\geq 1}$ be a filtration and $(X_n)_{n \geq 1}$ a sequence of Bernoulli random variables adapted to this filtration. Set  $p_n:=\mathbb P(X_n=1|\mathcal F_{n-1})$. Then,
$$
\frac{\sum_{j=1}^n X_j}{\sum_{j=1}^n p_j}\overset{a.s.}{\underset{n\rightarrow \infty}\longrightarrow} 1 \quad \text{ on the set }\big\{\textstyle \sum_{j=1}^{\infty}p_j=\infty \big\}.
$$
\end{lem}

Indeed, for each $i, 1 \leq i \leq k$ and for each $1\leq j \leq n$ set $X^{(i)}_{j}:=1$ if the color $i$ is drawn at step $n$ and $X^{(i)}_{j}:=0$ otherwise. The filtration $\mathcal F^{(i)}$ is the filtration generated by this sequence of Bernoulli random variables, and $p^{(i)}_{j}:=\mathbb P(X_{j}=1|\mathcal F_{j-1})=W_{j-1,i}/\sum_{\ell=1}^{k}W_{j-1,\ell}$. By Lemma \ref{lem:DF} and since $W_i>0$ for all $1\leq i \leq k$, we get the expected behaviour (using Ces\`aro's lemma).

\bigskip

\subsubsection{The dislocation measures $\bar \nu^{(i)}_{\mathrm{growth}}$} 

Let ${\mathcal S}$ denote the set of nonnegative summable sequences with sum smaller than 1 (with no constraint of monotonicity). 
It will be more convenient here to define first a dislocation-like measure on the set $$\overline{\mathcal S} \subset {\mathcal S}\times \{0,1,\ldots,\kappa\}^{\mathbb N} \text{ of sequences } (s_n,i_n)_{n\in \mathbb N} \text{ satisfying } s_n=0 \Leftrightarrow i_n=0, \text{ }\forall n \in \mathbb N.$$ As for $\overline{\mathcal S}^{\downarrow}$, we endow $\overline{\mathcal S}$ with the metric that assigns to two elements $\bar{\mathbf s}, \bar{\mathbf s}'$ the Prokhorov distance between the two measures defined from $\bar{\mathbf s}, \bar{\mathbf s}'$ by (\ref{def:metric}). 

If $\bar{\omega}$ denotes a measure on $\overline{\mathcal S}$ such that $\int_{\overline{\mathcal S}}(1-s_1) \bar{\omega} (\mathrm d\overline{\mathbf s})<\infty$, we will then let 
$\bar{\omega}^{\downarrow}$ denote the push-forward of this measure obtained by the map 
\begin{equation}
\label{def:rank}
\mathsf{rank} : \overline {\mathbf s} \mapsto (s_{\sigma(j)},i_{\sigma(j)})_{j \in \mathbb N} \in \overline{\mathcal S}^{\downarrow}
\end{equation} 
where $\sigma$ is a permutation on $\mathbb N$ such that $j <k$ if and only if $s_{\sigma(j)} > s_{\sigma(k)}$ or $s_{\sigma(j)} = s_{\sigma(k)}$ and $i_{\sigma(j)} \geq i_{\sigma(k)}$. Note that $\int_{\overline{\mathcal S}}(1-s_1) \bar{\omega}^{\downarrow} (\mathrm d\overline{\mathbf s})$ is then finite.

\bigskip

\begin{defn}
\label{def:randdisloc}
Let $(S_k)_{k\geq 0}$ denote a random walk starting from $S_0=0$ with increments having the distribution of $N+1$.
We define the measure $\bar{\omega}_{\mathrm{growth}}^{(i)}$ as a mixture of probability measures on  $\overline{\mathcal S}$:
\begin{eqnarray*}
\bar{\omega}_{\mathrm{growth}}^{(i)}&:=&\ell_0 \cdot \big(\mathrm{Urn}_{N+1}(\mathbf n_i), \mathbf i_i\big) \mathbf 1_{\{p_i \geq 1\}} \\
&+& \sum_{k=1}^{\infty} \ell_k \cdot \Big[q_A \cdot \big(\mathrm{Urn}_{N+1}(S_{k-1}+n_{i},\mathbf n_A), (i,\mathbf i_A)\big) +q_B \cdot \big(\mathrm{Urn}_{N+1}(S_{k-1}+n_{i},\mathbf n_B), (i,\mathbf i_B)\big)\Big],
\end{eqnarray*}
where $\ell_k, k\geq 0$ are the abstract nonzero quantities appearing in Lemma \ref{lemcv} (i) below.
When $n_B=n_A$, this expression is more explicit:
\begin{eqnarray*}
\bar{\omega}_{\mathrm{growth}}^{(i)}&=& \frac{\Gamma\left(\frac{n_i}{n_A+1} \right)}{\Gamma \left( \frac{n_i-1}{n_A+1}\right)} \cdot \left(\mathrm{Dir}\left(\frac{n_{i,1}}{n_A+1},\ldots,\frac{n_{i,p_i}}{n_A+1}\right), \mathbf i_i\right) \mathbf 1_{\{p_i \geq 1\}} \\
&+&  \frac{\Gamma\left(\frac{n_i}{n_A+1} \right)}{\Gamma \left(\frac{n_A+n_i}{n_A+1} \right) \cdot (n_A+1)}\cdot \left[q_A \cdot \frac{1}{1-s_1}\left(\mathrm{Dir}\left(\frac{n_{i}}{n_A+1},\frac{n_{A,1}}{n_A+1},\ldots,\frac{n_{A,p_A}}{n_A+1}\right), (i,\mathbf i_A)\right) \right. \\
&& \hspace{3.55cm}\ \left.+q_B \cdot \frac{1}{1-s_1}\left(\mathrm{Dir}\left(\frac{n_{i}}{n_A+1},\frac{n_{B,1}}{n_A+1},\ldots,\frac{n_{B,p_B}}{n_A+1}\right), (i,\mathbf i_B)\right)\right].
\end{eqnarray*}
\end{defn}

It is straightforward to check that this explicit expression of $\bar{\omega}_{\mathrm{growth}}^{(i)}$ when $n_A=n_B$ indeed corresponds to the abstract one, using (\ref{explicitell}) and the fact that $S_k=k(n_A+1), \forall k\geq 0$ in this case. Note that $n_A$ is necessarily nonzero in such a case (equivalently $p_A \geq 1)$, so that $(\mathbf n_A,\mathbf i_A) \neq \emptyset$ and the Dirichlet distributions are well-defined. 
In all cases the integral $\int_{\overline{\mathcal S}}(1-s_1 \mathbf 1_{\{i_1=i\}}) \bar{\omega}_{\mathrm{growth}}^{(i)} (\mathrm d\overline{\mathbf s})$ is finite according to Lemma \ref{lemsum} below. We can thus define:

\begin{defn}
\label{def:randdisloc2}
The dislocation measure $\bar{\nu}^{(i)}_{\mathrm{growth}}$ is then defined as $\bar{\nu}^{(i)}_{\mathrm{growth}}:=\bar{\omega}_{\mathrm{growth}}^{(i)\downarrow}$. 
\end{defn}

It it clear that this is indeed a dislocation measure that satisfies (\ref{hypocvcritique2}) when $\# \mathcal B \geq 2$ and (\ref{hypocvcritique3}) when $\# \mathcal B=1$. Note the particular case where $p_1=0$ and $n_A=n_B=1$:  $\bar{\nu}^{(1)}_{\mathrm{growth}}$ is then a monotype dislocation measure that corresponds to the Brownian dislocation measure (\ref{Brdisloc}) multiplied by $(2\sqrt 2)^{-1}$. More generally, if $T_0$ is the tree with a single edge and $\tau_A$ and $\tau_B$ are star trees with both $k$ edges, the dislocation measure  $\bar{\nu}^{(1)}_{\mathrm{growth}}$ is a monotype measure that was already identified in \cite{HS14}.

\subsection{The MB property}
\label{subsec:MB}

For each type $i \in [\# \mathcal B]$, we let $(T_n^{(i)})_{n\geq 0}$ denote a sequence of trees built from $T_0^{(i)}:=\tau_i$. In general,  these trees have a total number of vertices and total number of leaves that are random. There is however a quantity which is deterministic: the number of branchpoints created by gluing the roots of the successive brick trees onto the structure (this number is equal to $n$ in $T_n^{(i)}$). We will use this quantity to exhibit a MB property for a sequence of reduced trees, which will be sufficient to prove Theorem \ref{theo:recursif}.

\bigskip

\noindent \textbf{Assignment of types to the vertices of $T_n^{(i)}$.} We give types to all the non-root vertices of $T_0,$ $\tau_A$ and $\tau_B$ by declaring each ancestor of a copy of $\tau_j$ to have type $j$, for $1 \leq j \leq \# \mathcal B.$ Then, starting our recursive construction of $(T_n^{(i)})_{n\geq 0}$ from $T_0^{(i)}:=\tau_i$, the types are recursively attributed as follows: given $T_{n-1}^{(i)}$ and given that the root of the brick tree added at step $n$ is branched on an edge of $T_{n-1}^{(i)}$ which is the ``parent" of a vertex of type $j$, the new created branchpoint is of type $j$; the other new vertices are vertices of the added brick tree but its root and arrive with a type assigned according to the rule above; the vertices already present in $T_{n-1}^{(i)}$ keep their types. The root of $T_{n}^{(i)}$ does not have a type. Note that with this rule, the ancestor of $T_{n}^{(i)}$ is of type $i$ for all $n\geq 0$.

\bigskip

\noindent \textbf{Assignment of sizes to the vertices of $T_n^{(i)}$.} In the construction process of $(T_n^{(i)})_n$ we decide to color in red the branchpoints created when gluing the roots of the successive brick trees on the structure, so that $T_n^{(i)}$ possesses $n$ red vertices. For us, the size of a vertex of $T_n^{(i)}$ is the number of red vertices amongst its descendants, including itself. 

\bigskip

\noindent \textbf{The MB property of a family of reduced trees.} For each $i \in [\# \mathcal B]$ and each $n\geq 1$, let $T_n^{(i),\dag}$ be the tree obtained from $T_n^{(i)}$ by removing the root and its adjacent edge as well as all vertices of size 0 and their descending edges. The root of $T_n^{(i),\dag}$ is thus the ancestor of $T_n^{(i)}$, and it has size $n$ and type $i$. The MB property of the family $(T_n^{(i),\dag},n\geq 1,i\in [\# \mathcal B])$ then follows readily since the selected edge on which is glued the next brick tree in the construction process is chosen uniformly and the type of the root always corresponds to the initial tree: the MB property is therefore simply due to the fact that the restriction of the uniform distribution to a subset is still uniform in that subset. We let $q_n^{(i)}$ denote the distribution on $\overline{\mathcal P}_n$ of the couples of sizes and types of the vertices above the root of $T_n^{(i),\dag}$, ranked according to the usual rule - note that they may be non-conservative. We then let $\bar \nu^{(i)}_n$ be the push-forward of $q_n^{(i)}$ by the map which divides the size-parts by $n$.

\begin{lem} For all $i \in [\# \mathcal B]$,
\label{lem:cvnuGT}
\begin{equation*}
\label{cv:nuGT}
n^{\frac{1}{\mathbb E[N]+1}}\left(1-s_1\mathbf{1}_{\{i_1=i\}}\right)\bar{\nu}_n^{(i)}(\mathrm d \bar{\mathbf{s}}) \underset{n\to\infty}{\overset{weakly}{\longrightarrow}} \left(1-s_1\mathbf{1}_{\{i_1=i\}}\right)\bar{\nu}_{\mathrm{growth}}^{(i)}(\mathrm d \bar{\mathbf{s}}).
\end{equation*} 
\end{lem}

This lemma will be proved in the next section. We finish this section by noticing that it implies Theorem \ref{theo:recursif}. 

\bigskip 

\textbf{From Lemma \ref{lem:cvnuGT}  to Theorem \ref{theo:recursif}.} A direct consequence of Lemma \ref{lem:cvnuGT} and Theorem \ref{theo:critique} is that  if $\mu_n^{(i),\dag}$ denotes the probability measure that assigns the weight $n^{-1}$ to each red vertex of $T_n^{(i),\dag}$ and the weight 0 to each other vertex, we have that
\[\big(n^{-\frac{1}{\mathbb E[N]+1)}} \cdot T_n^{(i),\dag},\mu_n^{(i),\dag}\big) \overset{(d)}{\underset{n\rightarrow \infty}{\longrightarrow}} \big(\T^{(i)}_{\gamma,\bnu_{\mathrm{growth}}},\mu^{(i)}_{\gamma,\bnu_{\mathrm{growth}}}\big),
\]
for the GHP-topology, where $\bnu_{\mathrm{growth}}=\big(\bar \nu_{\mathrm{growth}}^{(j)},j \in [\# \mathcal B]\big)$. This in turn implies the convergence in distribution of $\big(n^{-\frac{1}{\mathbb E[N]+1}} \cdot T_n^{(i)}, \mu_n^{(i)}\big)$  to $\big(\T^{(i)}_{\gamma,\bnu_{\mathrm{growth}}},\mu^{(i)}_{\gamma,\bnu_{\mathrm{growth}}}\big)$, where $ \mu_n^{(i)}$ denotes the uniform probability on the vertices of  $T_n^{(i)}$. Indeed, note first that the subtrees removed above the vertices of size 0 to get $T_n^{(i),\dag}$ from $T_n^{(i)}$, $n\geq 1$,  all have a number of edges smaller than $\max(n_A,n_B,n_i)$, which means that the Hausdorff distance between $T_n^{(i),\dag}$ and $T_n^{(i)}$ is bounded and therefore that $n^{-\frac{1}{\mathbb E[N]+1}} \cdot T_n^{(i)}$ converges in distribution to $\T^{(i)}_{\gamma,\bnu_{\mathrm{growth}}}$ for the Gromov-Hausdorff topology.
Second, to incorporate the measures, one could note that the strong law of large numbers implies the convergences $\mu_n^{(i)}(T^{(i)}_{n,r_j})/\mu_n^{(i), \dag}(T^{(i)}_{n,r_j})\rightarrow 1$ a.s.\ for all $j\geq 1$, where $r_1,r_2,\ldots$ denotes the red vertices by order of appearance and $T^{(i)}_{n,r_j}$ the subtree in $T_n^{(i)}$ of descendants of $r_j$ including itself (this subtree exists for $n\geq j$). One could then conclude by using the Skorokhod representation theorem and a tightness argument, but this would be a bit long technically, so we prefer to use the following shortcut: Proposition 1 of  S\'enizergues \cite{S18}  implies the a.s.\ convergence of \linebreak $\big(n^{-\frac{1}{\mathbb E[N]+1}} \cdot T_n^{(i)}, \mu_n^{(i)}\big)$ and having a look at the way things proceed in Section 5.1 and Section 5.2 within, we see that one jointly gets the a.s.\ convergence of $\big(n^{-\frac{1}{\mathbb E[N]+1}} \cdot T_n^{(i)}, \mu_n^{(i),\dag}\big)$ to the same limit, which is then necessarily distributed as $\big(\T^{(i)}_{\gamma,\bnu_{\mathrm{growth}}},\mu^{(i)}_{\gamma,\bnu_{\mathrm{growth}}}\big)$. This leads to Theorem \ref{theo:recursif} since  $T_0=T_0^{(1)}$ and the sequences $(T_n)_n$ and $(T_n^{(1)})_n$ have the same distribution. 

\subsection{Convergence of the splitting distributions: proof of Lemma \ref{lem:cvnuGT}}
\label{sec:cvSD}

We fix a type $i \in [\# \mathcal B]$. Our goal is to prove Lemma \ref{lem:cvnuGT}. The global strategy is to first prove the convergence of non-ordered versions of the measures  $\bar{\nu}_n^{(i)}$ (the $\bar \omega_n^{(i)}$ introduced below), and, in fact, we will first focus on conditioned versions of these measures $\bar \omega_n^{(i)}$ given the index of the last brick tree glued next to the root in the construction of $T_n^{(i)}$.  We start by introducing the notations and then turn to the different steps of the proof. 

\subsubsection{Non-monotonic sequences and introduction of the main notation} 
\label{sec:notation}

We fix a type $i \in [\# \mathcal B]$ and an integer $n\geq 0$. 

\bigskip

\noindent \textbf{Index of the last brick glued next to the root.}
Let $J_n\in \{0,1,\ldots,n\}$ denote the random variable that corresponds to the rank of the last step in the construction of $T^{(i)}_n$ at which the new brick is glued next to the root. Specifically, $J_n$ is equal to $k \geq 1$ if at step $k$ the new brick is glued on the edge adjacent to the root of $T^{(i)}_{k-1}$ and if for $k+1\leq j \leq n$ the new brick is \emph{not} glued on the edge adjacent to the root of $T^{(i)}_{j-1}$, and we let $J_n=0$ in the case where none of the brick trees are glued next to the root up to step $n$. Let $(S_n)$ be a random walk starting from $S_0=0$ with increments distributed as $N+1$. Then, clearly,  $\mathbb P(J_n=k)=0$ if $k>n$, $\mathbb P(J_n=0)=0$ if $p_i=0$ and
$$
\mathbb P(J_n=0)=\mathbb E\left[ \prod_{j=0}^{n-1} \frac{n_i+S_{j}-1}{n_i+S_j} \right] \quad \text{if } p_i \geq 1,
$$
and for $1 \leq k\leq n$
$$
\mathbb P(J_n=k)=\mathbb E\left[\frac{1}{n_i+S_{k-1}} \prod_{j=k}^{n-1} \frac{n_i+S_{j}-1}{n_i+S_j} \right]. 
$$

\bigskip

\noindent \textbf{Notation for the sizes of the subtrees descending from the ancestor of $T_n^{(i)}$.} When $p_i \geq 1$ and $J_n=0$, the ancestor of  $T_n^{(i)}$ is the ancestor of $T_0^{(i)}=\tau_i$, which, with the notation introduced at the beginning of Section \ref{sec:App1}, splits in $p_i$ subtrees with sizes and types sequence $(\mathbf{n}_{i},\mathbf{i}_i)$. For the subtree with index $j$, $1\leq j\leq p_i$, we let $W^{(i)}_{n,j}$ denote its number of edges and $R^{(i)}_{n,j}$ its number of red vertices. Note that if $R^{(i)}_{n,j}\geq 1$, $i_{i,j}$ is the type of its ancestor. We then let $\mathbf W_n^{(i)}:=(W^{(i)}_{n,j})_{1 \leq j \leq p_i}$ and $\mathbf R_n^{(i)}:=(R^{(i)}_{n,j})_{1 \leq j \leq p_i}$. When $J_n=k \geq 1$ and the brick tree glued on $T^{(i)}_{k-1}$ at step $k$ is $\tau_A$, the ancestor of $T_n^{(i)}$ splits in $p_A+1$ subtrees. We decide to give the index $j+1$ to the subtree ``built on the subtree above the root of $\tau_A$ with initial size $n_{A,j}$ and ancestor type $i_{A,j}$", $1 \leq j \leq p_A$, and to give the index 1 to the remaining tree (built on a subtree above the ancestor of $T_k^{(i)}$ which is identical to $T_{k-1}^{(i)}$). For the subtree with label $j$, $1\leq j\leq p_A+1$, we let $W^{(i,A)}_{n,j}(k)$ denote its number of edges and $R^{(i,A)}_{n,j}(k)$ its number of red vertices. The notations $\mathbf W_n^{(i,A)}(k)$ and $\mathbf R_n^{(i,A)}(k)$ denote the corresponding sequences of length $p_A+1$. 
We proceed similarly when the  brick tree glued on $T^{(i)}_{k-1}$ at step $k$ is $\tau_B$, replacing in all notations the letter $A$ by $B$. 

\bigskip

\noindent \textbf{The splitting measure $\bar \omega^{(i)}_n$.} We first define conditional versions of this probability measure on $\overline{\mathcal S}$. With the notation introduced above, we let $\bar \omega^{(i)}_{n  |J_n=0}$ denote the distribution of $(n^{-1}\mathbf R^{(i)}_{n},\mathbf i_{i})$ when $p_i\geq 1$ and the null measure when $p_i=0$, and for $1\leq k\leq n,$
$$\bar \omega^{(i)}_{n  | J_n=k}:=q_A\delta_{(n^{-1}\mathbf R^{(i,A)}_{n}(k),(i,\mathbf i_A))}+q_B\delta_{(n^{-1}\mathbf R^{(i,B)}_{n}(k),(i,\mathbf i_B))}.$$
We then define
$$
\bar \omega^{(i)}_n:=\sum_{k=0}^{n}  \mathbb P(J_n=k) \bar \omega^{(i)}_{n |J_n=k}.
$$
Note that the measure $\bar \nu^{(i)}_n$ introduced in Section \ref{subsec:MB} is the push-forward of $\bar \omega^{(i)}_n$ by the map $\mathsf{rank}$ introduced in Section \ref{sec:disloc}.

\bigskip

In order prove Lemma \ref{lem:cvnuGT}, we will proceed in four steps. We will first determine the asymptotic behaviour of the probabilities $\mathbb P(J_n=k), k \geq 0$, as $n \rightarrow \infty$. Second we will determine the asymptotic behaviour of the conditional probabilities $\bar \omega^{(i)}_{n |J_n=k}$ for each $k\geq 0$, and then the asymptotic behaviour of $\bar \omega^{(i)}_n$ by summing over $k$. Last, ranking the sequences in decreasing order, we will deduce the expected behaviour of $(\bar \nu^{(i)}_n)$.

\subsubsection{Asymptotic behaviour of $\mathbb P(J_n=k)$ when $n \rightarrow \infty$} 

In this section we prove the following lemma, and as a corollary build the $\ell_k, k\geq 0$ involved in the Definition \ref{def:randdisloc} of the dislocation-like measures $\bar \omega_{{\mathrm{growth}}}^{(i)}$. 

\begin{lem}
\label{lemcv}
\emph{(i)} For $k \geq 1$, and  $k=0$ when $p_i\geq 1$, there exists $\ell_k \in (0,\infty)$ such that 
$$n^{\frac{1}{\mathbb E[N]+1}} \cdot \mathbb P(J_n=k) \underset{n \rightarrow \infty}{\longrightarrow} \ell_k.$$ 
\noindent \emph{(ii)} Moreover, there exists a constant $c \in (0,\infty)$ such that
\begin{equation*}
\label{eq:domination}
\sup_{n \geq 1} n^{\frac{1}{\mathbb E[N]+1}} \cdot \mathbb P(J_n=k) \leq \frac{c}{k^{1-\frac{1}{\mathbb E[N]+1}}}, \quad \forall k \geq 1.
\end{equation*}
\end{lem}

\bigskip

When $n_B=n_A$ this lemma is in fact easy to prove and the limits are explicit, since
$$
\mathbb P(J_n=k)=\left\{\begin{array}{l}
 \vspace{0.3cm} 0  \quad \text{if } k=0 \quad \text{and} \quad p_i=0
 \\
 \vspace{0.3cm}
\frac{\Gamma \left( \frac{n_i-1}{n_A+1}+n\right)}{\Gamma \left( \frac{n_i}{n_A+1}+n\right)} \cdot \frac{\Gamma \left( \frac{n_i}{n_A+1}\right)}{\Gamma \left(\frac{n_i-1}{n_A+1} \right)} \quad \text{if } k=0  \quad \text{and} \quad p_i \geq 1 \\ 
\frac{\Gamma \left(  \frac{n_i-1}{n_A+1}+n\right)}{\Gamma \left(\frac{n_i}{n_A+1}+n \right)} \cdot \frac{\Gamma \left(  \frac{n_i}{n_A+1}+k\right)}{\Gamma \left( \frac{n_i-1}{n_A+1}+k \right)} \cdot \frac{1}{n_i+(n_A+1)(k-1)}\quad \text{if } 1\leq k \leq n. \end{array}\right.
$$
By Stirling's formula, we then have
\begin{equation*}
\mathbb P(J_n=0) \underset{n\rightarrow \infty}\sim \frac{\Gamma \left( \frac{n_i}{n_A+1}\right)}{\Gamma \left(\frac{n_i-1}{n_A+1} \right)} \cdot n^{-\frac{1}{n_A+1}}\mathbf 1_{\{p_i\geq 1\}} \quad \text{and} \quad 
\mathbb P(J_n=k) \underset{n\rightarrow \infty}\sim \frac{1}{n_A+1} \cdot \frac{\Gamma \left( \frac{n_i}{n_A+1}+k-1\right)}{\Gamma \left(\frac{n_i-1}{n_A+1}+k \right)} \cdot n^{-\frac{1}{n_A+1}}
\end{equation*}
for $k\geq 1$, and then
\begin{equation}
\label{explicitell}
\ell_0= \frac{\Gamma \left( \frac{n_i}{n_A+1}\right)}{\Gamma \left(\frac{n_i-1}{n_A+1} \right)}\mathbf 1_{\{p_i\geq 1\}} \quad \text{and} \quad \ell_k= \frac{1}{n_A+1} \cdot \frac{\Gamma \left( \frac{n_i}{n_A+1}+k-1\right)}{\Gamma \left(\frac{n_i-1}{n_A+1}+k \right)}, \quad k\geq 1.
\end{equation}

\bigskip

To prove Lemma \ref{lemcv} in the general setting, the key point is the following consequence of Hoeffding's inequality. The notation $(S_{j})_{j\geq 0}$ still refers to a random walk starting from 0 and with increments distributed as $N+1$.
\begin{lem}
\label{lemHoeffding}
Let $\varepsilon \in (0,1/2)$ and $\lambda \geq 0$. Then,
$$
\mathbb E\left[\exp\left(\lambda \sup_{j\geq 1}j^{\varepsilon} \left|\frac{S_j}{j}-(\mathbb E[N]+1) \right| \right)\right]<+\infty.
$$
\end{lem}
\begin{proof}
The increments of the random walk being deterministically bounded, we know, according to Hoeffding's inequality, that there exists  $c>0$ such that 
$$\mathbb P\left( \left|\frac{S_n}{n}-(\mathbb E[N]+1) \right| \geq un^{-\varepsilon}\right) \leq 2\exp(-cu^2n^{1-2\varepsilon}), \quad \forall u\in [0,\infty) \text{ and } \forall n\geq 1.$$
Consequently, for $u \in [1,\infty)$,
\begin{eqnarray*}
\mathbb P\left( \sup_{j\geq 1}j^{\varepsilon}\left|\frac{S_j}{j}-(\mathbb E[N]+1) \right| \geq u\right) &\leq& \sum_{j=1}^{\infty } 2\exp(-cu^2j^{1-2\varepsilon}) \\
&\underset{\text{since }u\geq 1}\leq & 2\exp(-(c/2)u^2)  \sum_{j=1}^{\infty } \exp(-(c/2)j^{1-2\varepsilon}) \\
&\leq & d \exp(-(c/2)u^2), 
\end{eqnarray*}
with $d$ finite, independent of $u\geq 1$.
We then use that for any positive random variable $X$
$$
\mathbb E[\exp(\lambda X)]=1+{\lambda}\int_0^{\infty} \exp(\lambda u) \mathbb P(X \geq u) \mathrm du
$$
(found easily by e.g. integrating by parts). Hence
\begin{eqnarray*}
\mathbb E\left[\exp\left(\lambda \sup_{j\geq 1}j^{\varepsilon} \left|\frac{S_j}{j}-(\mathbb E[N]+1) \right| \right)\right]-1&=&\lambda \int_0^{\infty} \exp(\lambda u) \mathbb P\left(\sup_{j\geq 1}j^{\varepsilon} \left|\frac{S_j}{j}-(\mathbb E[N]+1) \right| \geq u \right) \mathrm du \\
&\leq & \lambda \int_0^{1} \exp(\lambda u) \mathrm du+\lambda \int_1^{\infty} \exp(\lambda u)d \exp(-(c/2)u^2) \mathrm du,
\end{eqnarray*}
which is finite, as expected.
\end{proof}

\noindent \textit{Proof of Lemma \ref{lemcv}}. Fix $k\geq 1$, let for $n\geq k$
$$
a_{n,k}:=\exp^{\frac{1}{\mathbb E[N]+1}  \sum_{j=k}^{n} \frac{1}{j}} \cdot \mathbb E\left[\frac{1}{n_i+S_{k-1}}\prod_{j=k}^n \frac{n_i+S_j-1}{n_i+S_j} \right]
$$
and note that (i) will be proved (for $k\geq 1$) if we show that $a_{n,k}$ has a finite nonzero limit when $n \rightarrow \infty$, and that (ii) will be proved if we show that $\sup_{n\geq k\geq 1}ka_{n,k}$ is finite. We will do both simultaneously. We start by rewriting
\begin{eqnarray*}
a_{n,k}&=&\mathbb E\left[\frac{1}{n_i+S_{k-1}}\exp\left(\sum_{j=k}^n \left(\frac{1}{j(\mathbb E[N]+1)}+\ln \left(1-\frac{1}{n_i+S_j}\right)\right) \right) \right] \\
&=& \mathbb E\left[\frac{1}{n_i+S_{k-1}}\exp\left(\frac{1}{\mathbb E[N]+1}\sum_{j=k}^n \frac{1}{(n_i+S_j)}\left(\frac{S_j}{j}-(\mathbb E[N]+1)\right)+ \frac{n_i}{\mathbb E[N]+1} \sum_{j=k}^n\frac{1}{j(n_i+S_j)} \right. \right.\\
&& \hspace{3cm} + \left. \left.\sum_{j=k}^n \left(\frac{1}{n_i+S_j}+\ln \left(1-\frac{1}{n_i+S_j}\right)\right) \right) \right].
\end{eqnarray*}
Since $S_j \geq j$ for all $j\geq 1$, and $\ln(1+x)=x+O(x^2)$ $(x \rightarrow 0)$, the sum $\sum_{j=k}^n \big( (n_i+S_j)^{-1}+\ln \big(1-(n_i+S_j)^{-1}\big)\big) $ converges almost surely as $n \rightarrow \infty$ and is deterministically bounded in $n\geq k \geq 1$. This also implies that the sum $\sum_{j=k}^n(j(n_i+S_j))^{-1}$ converges almost surely as $n \rightarrow \infty$ and is deterministically bounded in $n\geq k \geq 1$. Next, by Lemma \ref{lemHoeffding}, the sum $\sum_{j=k}^n (n_i+S_j)^{-1} (S_j j^{-1}-(\mathbb E[N]+1))$ is also convergent and is bounded from above, in absolute value, by a random variable which has exponential moments of all orders. Last, the term $(n_i+S_{k-1})^{-1}$ is smaller than $(n_i+k-1)^{-1}$, which is deterministic. We can therefore conclude by dominated convergence that $a_{n,k}$ converges to a finite limit when $n \rightarrow \infty$. This limit is the expectation of the exponential of a finite random variable, hence it is nonzero. Hence (i). Moreover these few lines also show that $\sup_{n\geq k \geq 1}ka_{n,k}$ is smaller than the expectation of the exponential of a random variable which has exponential moments of all orders. Hence (ii).

The proof of (i) holds similarly when $k=0$ and $p_i \geq 1$. This is left to the reader.
$\hfill \square$

\subsubsection{Embedded urn models and the asymptotic behaviour of $\bar \omega^{(i)}_{n |J_n=k}$} 

The notation $(S_{k})_{k\geq 0}$ still refers to a random walk starting from 0, with increments distributed as $N+1$.

\begin{lem} 
\label{cvurns}
\emph{(i)}
For $k\geq 1$,
\begin{eqnarray*}
\bar \omega^{(i)}_{n | J_n=k} &\underset{n\rightarrow \infty}{\overset{\mathrm{weakly}}\longrightarrow}& 
q_A \cdot \big(\mathrm{Urn}_{N+1}(S_{k-1}+n_{i},\mathbf n_A); (i,\mathbf i_A)\big) + q_B \cdot \big(\mathrm{Urn}_{N+1}(S_{k-1}+n_{i},\mathbf n_B); (i,\mathbf i_B)\big).
\end{eqnarray*}
For $k=0$ and $p_i \geq 1$, 
$$\bar \omega^{(i)}_{n | J_n=0} \underset{n\rightarrow \infty}{\overset{\mathrm{weakly}}\longrightarrow} \big(\mathrm{Urn}_{N+1}(\mathbf n_i); \mathbf i_i\big).$$ 

\noindent \emph{(ii)} Moreover,
$$
\sup_{n\geq 1}\int_{\overline S}(1-s_1\mathbf 1_{\{i_1=i\}}) \bar \omega^{(i)}_{n | J_n=k}(\mathrm d \overline{\mathsf s}) \leq \frac{\mathbb E[N]+2}{k}, \quad \forall k\geq 1.
$$
\end{lem}

\bigskip

The proof of this result relies on urn models involved in our recursive scheme. 

\textit{Proof of Lemma \ref{cvurns}.}
(i) Let $n\geq k \geq 1$. We work conditionally on $J_n=k$ and use the notation for the sizes of the subtrees descending from the ancestor of $T_n^{(i)}$ introduced in Section \ref{sec:notation}. For the moment, we also work conditionally on the fact that the brick tree at step $k$ is $\tau_A$. It is then clear that  $\mathbf W^{(i,A)}_{n}(k)$ is distributed as a P\'olya urn model  with random increments distributed as $N+1$ and initial (random) weights $(S_{k-1}+n_i,\mathbf n_A)$, after $n-k$ steps, and that $\mathbf R^{(i,A)}_{n}(k)-(k-1,0,\ldots,0)$ is the corresponding sequence of numbers of times each ``color" has been drawn. Hence, by (\ref{cvurnR}), 
$$n^{-1}\cdot \mathbf R^{(i,A)}_{n}{(k)} \overset{\mathrm{law}}{\underset{n \rightarrow \infty}\longrightarrow} \mathrm{Urn}_{N+1}(S_{k{\color{blue}-1}}+n_{i},\mathbf n_A).$$
Moreover, still given that the grafted tree at step $k$ is $\tau_A$, the sequence of types of the ancestors of the subtrees above the ancestor of $T_n^{(i)}$ is equal to $(i,\mathbf i_{A})$.
The arguments are similar when replacing $A$ by $B$.  Since the grafted brick tree at step $k$ is $\tau_A$ with probability $q_A$ and $\tau_B$ with probability $q_B$, this gives the statement for $k\geq 1$.

When $k=0$ and $n\geq 0$ we proceed similarly with the random sequence $(\mathbf W_{n}^{(i)},\mathbf R_{n}^{(i)})$. 

\bigskip

\noindent (ii) Fix $k\geq 1$. We use a coupling argument.  Let $(X_j,j\geq 1)$ be a sequence of i.i.d. random variables distributed as $N+1$, and $T_n:=\sum_{j=1}^n X_j, n \geq 1$, $T_0:=0$. Then let $(R_n,W_n)_{n\geq k}$ be a random sequence evolving as follows, conditionally on $(X_j,j\geq 1)$: $R_k=k-1,W_k=T_{k-1}$ and given $(R_i,W_i)_{k \leq i \leq n}$, 
$$
\begin{array}{llll}
R_{n+1}= R_n+ 1& \text{and} & W_{n+1} = W_{n}+X_{n+1} & \text{with probability } \frac{W_n}{T_n} \\
R_{n+1}= R_n & \text{and} & W_{n+1} = W_n &  \text{with probability } 1-\frac{W_n}{T_n}. 
\end{array}
$$
The sequence $\big(\frac{W_n}{T_n}\big)_{n \geq k}$ being a martingale, we have that $\mathbb E[R_{n+1}-R_n]=\mathbb E\big[\frac{W_n}{T_n}\big]=\mathbb E\big[\frac{W_k}{T_k}\big]=\mathbb E \big[\frac{T_{k-1}}{T_k}\big]$ for all $n\geq k$. Hence, $\mathbb E[R_n]=\mathbb E[R_k]+(n-k)\mathbb E \big[\frac{T_{k-1}}{T_k}\big]$ and
\begin{eqnarray*}
\mathbb E\left[1-\frac{R_n}{n}\right]&=&1-\frac{k-1}{n}-\mathbb E\left[\frac{T_{k-1}}{T_k}\right]+\frac{k}{n}\mathbb E\left [\frac{T_{k-1}}{T_k}\right] \\
&\underset{T_{k-1}\leq T_k}\leq & \frac{1}{n} +\mathbb E\left[\frac{T_{k}-T_{k-1}}{T_k}\right] \\
&\underset{T_k\geq k}\leq& \frac{1}{n}+\frac{\mathbb E[X_k]}{k} \\
& \underset{n\geq k}\leq& \frac{1}{k}+\frac{\mathbb E[N]+1}{k}. 
\end{eqnarray*}
To conclude we notice that for all $n\geq k$, $R_n/n$ is distributed as the push-forward of $\bar \omega^{(i)}_{n | J_n=k}$ by the map $\bar{\mathbf s} \in \overline{\mathcal S} \mapsto  s_1$, which gives us
$$
\sup_{n\geq 1}\int_{\overline S}(1-s_1) \bar \omega^{(i)}_{n | J_n=k}(\mathrm d \overline{\mathsf s})\leq \frac{\mathbb E[N]+2}{k}, \quad \forall k \geq 1
$$
as expected (note that $i_1=i$ \ $ \bar \omega^{(i)}_{n | J_n=k}$-a.e. when $k \geq 1$).
$\hfill \square$

\subsubsection{Summing over $k$ and the asymptotic behaviour of $\bar{\omega}^{(i)}_n$}

We can now deduce from the previous sections that $n^{\frac{1}{\mathbb E[N]+1}} \cdot \bar{\omega}^{(i)}_n$ approximates the dislocation-like measure $\bar{\omega}^{(i)}_{\mathrm{growth}}$ of Definition \ref{def:randdisloc} in the following manner.

\begin{lem}
\label{lemsum}
Let $f:\overline{\mathcal S} \rightarrow \mathbb R$ be a bounded continuous function. Then, 
$$
n^{\frac{1}{\mathbb E[N]+1}} \int_{\overline{\mathcal S}}(1-s_1\mathbf 1_{\{i_1=i\}})f(\overline{\mathbf s}) \bar \omega^{(i)}_{n}(\mathrm d \overline{\mathbf s})
\underset{n\rightarrow \infty}\longrightarrow \int_{\overline{\mathcal S}}(1-s_1\mathbf 1_{\{i_1=i\}})f(\overline{\mathsf s}) \bar{\omega}_{\mathrm{growth}}^{(i)}(\mathrm d \overline{\mathsf s}),
$$ 
the integral in the limit being well-defined and finite.
\end{lem}

\begin{proof} With no loss of generality we can assume that $f$ is positive and bounded from above by 1. Recall that 
$$
\bar \omega^{(i)}_n=\sum_{k=0}^{n}  \mathbb P(J_n=k) \bar \omega^{(i)}_{n |J_n=k}.
$$
From Lemma \ref{lemcv} (i) and Lemma \ref{cvurns} (i), we know that for all $k\geq 0$
$$
n^{\frac{1}{\mathbb E[N]+1}} \int_{\overline{\mathcal S}}(1-s_1\mathbf 1_{\{i_1=i\}})f(\overline{\mathbf s}) \mathbb P(J_n=k)\bar \omega^{(i)}_{n | J_n=k}(\mathrm d \overline{\mathbf s}) \underset{n\rightarrow \infty}\longrightarrow \int_{\overline{\mathcal S}}(1-s_1\mathbf 1_{\{i_1=i\}})f(\overline{\mathbf s}) \ell_k \cdot \bar \omega^{(i)}_{{\mathrm{growth}},k}(\mathrm d \overline{\mathbf s}),$$ 
where $\bar \omega^{(i)}_{{\mathrm{growth}},k}$ is the measure 
\begin{eqnarray*}
 q_A \cdot \big(\mathrm{Urn}_{N+1}(S_{k-1}+n_{i},\mathbf{n}_A), (i,\mathbf i_A)\big) + q_B \cdot \big(\mathrm{Urn}_{N+1}(S_{k-1}+n_{i},\mathbf n_B), (i,\mathbf i_B)\big),
\end{eqnarray*} 
when $k\geq 1$, the measure $(\mathrm{Urn}_{N+1}(\mathbf n_i), \mathbf i_i)$ when $k=0$ and $p_i \geq 1$ and the null measure when $k=0$ and $p_i =0$.
If we can sum the above convergences over $k \in \mathbb Z_+$, we will have the expected result. 
By Lemma \ref{lemcv} (ii) and Lemma \ref{cvurns} (ii) we have the existence of some $c\in (0,\infty)$ such that  for all $k\geq 1$
$$\sup_{n\geq 1} n^{\frac{1}{\mathbb E[N]+1}} \int_{\overline{\mathcal S}}(1-s_1\mathbf 1_{\{i_1=i\}})f(\overline{\mathbf s}) \mathbb P(J_n=k)\bar \omega^{(i)}_{n | J_n=k}(\mathrm d \overline{\mathbf s}) \leq \frac{c}{k^{2-\frac{1}{\mathbb E[N]+1}}}.$$ Since $\mathbb E[N]>0$, $2-1/(\mathbb E[N]+1)>1$ and we can apply the dominated convergence theorem to conclude.
\end{proof}

\subsubsection{Back to $\bar \nu^{(i)}_n$}

For all $n$ in $\N,$ the measure $\bar \nu^{(i)}_n$ is the push-forward of $\bar \omega^{(i)}_n$ by the function $\mathsf{rank}$ defined in (\ref{def:rank}), and similarly $\bar \nu^{(i)}_{\mathrm{growth}}$ is the push-forward of $\bar \omega^{(i)}_{\mathrm{growth}}$ by $\mathsf{rank}$. In order to get Lemma \ref{lem:cvnuGT} from Lemma \ref{lemsum}, we then note that 
$$
\int_{\overline{\mathcal S}^{\downarrow}}\left(1-s_1 \mathbf 1_{\{i_1=i\}} \right)f(\overline{\mathbf s}) \overline{\nu}^{(i)}_n(\mathrm d \overline{\mathbf s})=\int_{\overline{\mathcal S}}\left(1-s_1 \mathbf 1_{\{i_1=i\}} \right)g^{(i)}(\overline{\mathbf s})(f\circ \mathsf{rank})(\overline{\mathsf s}) \overline \omega^{(i)}_n(\mathrm d \overline{\mathbf s}), \quad \forall n \geq 1
$$
for any bounded continuous function $f:\overline{\mathcal S} \rightarrow \mathbb R$, where for $\overline{\mathbf s}=(s_n,n\in \mathbb N)\in \overline{\mathcal S}$
\begin{eqnarray*}
g^{(i)}(\overline{\mathbf s}):=\left\{\begin{array}{ll}\frac{1-\max_{j\geq 1} s_j \mathbf 1_{\{i_{\max}(\overline{\mathbf s})=i\}}}{1-s_1 \mathbf 1_{\{i_1=i\}}} & \text{if }s_1\neq 1 \text{ or }i_1\neq i \\1 & \text{if }s_1=1 \text{ and }i_1=i \end{array}\right.
\end{eqnarray*}
and $i_{\max}(\overline{\mathbf s}):=\max\{i_j : s_j=\max_{k \in \mathbb N}s_k,  j \in \mathbb N \}$ is the largest index of the largest size-term of $\overline{\mathbf s}$. A similar identity holds for the limiting measures $\overline{\nu}^{(i)}_{\mathrm{growth}}$ and $\overline{\omega}^{(i)}_{\mathrm{growth}}$. 

The function $i_{\max}$ is not continuous on $\overline{\mathcal S}$ but is continuous on $\overline{\mathcal S}_{\neq}:=\{\overline{\mathbf s} \in \overline{\mathcal S} :s_i \neq s_j, \forall i\neq j\}$, and consequently the function $g^{(i)}$ is also continuous on $\overline{\mathcal S}_{\neq}$. The function $f\circ \mathsf{rank}$ is also  continuous on that set. Since $ \overline \omega^{(i)}_{\mathrm{growth}}\left(\overline{\mathcal S}\backslash \overline{\mathcal S}_{\neq} \right)=0$ by the property (\ref{diff}) applied to the urn models involved in the definition of $\overline \omega^{(i)}_{\mathrm{growth}}$, we indeed get the convergence of Lemma \ref{lem:cvnuGT} as a consequence of Lemma \ref{lemsum}.

\subsection{Starting with an initial tree with root degree larger than 1}

Keeping the notations of the introduction of Section \ref{sec:App1}, we have assumed until now that the sequence $(T_n)_{n \in \mathbb N}$ was built on an initial tree $T_0$ with root degree 1. In fact, Theorem \ref{theo:recursif} extends easily to the case where the root of $T_0$ has a degree $d \geq 2$.  This root gives rise to $d$ \emph{planted} trees that we denote $(T_{0,1}, \ldots, T_{0,d})$, in arbitrary order. We let $(a_1,\ldots,a_d)$ denote their respective numbers of edges. Clearly, the respective numbers of trees glued on these $d$ subtrees in our construction process correspond to the respective numbers of drawn colors in an urn model with initial weight $(a_1,\ldots,a_d)$ and random increments distributed as $N+1$. By (\ref{diff}), we know that their proportions converge almost surely to a random variable with distribution $\mathrm{Urn}_{N+1}(a_1,\ldots,a_d)$. It is also clear that given these proportions, each of the $d$ subtrees evolve according to our construction scheme, independently of the others.  Consequently, in this case,
\begin{equation*}
\big(n^{-\frac{1}{\mathbb E[N]+1}} \cdot T_n,\ \mu_n \big)\  \overset{\mathrm{a.s.}}{\underset{n \rightarrow \infty}{\longrightarrow}}\ \left(\mathcal T, \mu \right)
\end{equation*}
where the limiting tree is obtained by identifying at their roots $d$ multi-type rescaled fragmentions trees $((W_j)^{1/(\mathbb E[N]+1)} \cdot \mathcal T_{j}, \ W_j \cdot \mu_{j}), 1 \leq j \leq d$, where: 
\begin{enumerate}
\item[$\bullet$] $(W_1,\ldots,W_d) \sim \mathrm{Urn}_{N+1}(a_1,\ldots,a_d)$, is independent of $(\mathcal T_{j},\mu_{j},1 \leq j \leq d)$
\item[$\bullet$] the $(\mathcal T_{j},\mu_{j}), 1 \leq j \leq d$ \ are independent multi-type fragmentation trees, all of index $(\mathbb E[N]+1)^{-1}$, such that $(\mathcal T_{j},\mu_{j})$  has a set of types given by
$$
\mathcal B_{j}=\left\{\tau_{v},v \in \tau_A\backslash{\{\rho_A\}}\cup \tau_B\backslash{\{\rho_B\}} \cup T_{0,j} \backslash\{{\rho_{0,j}}\}\right\}\
$$
-- where, as before, $\tau_v$ is a planted version of the subtree descending from $v$ -- and associated $\# B_{j}$ dislocation measures defined as in Definition \ref{def:randdisloc2}. The type of the root of $\mathcal T_{j}$ is the one corresponding to $T_{0,j}$.
\end{enumerate}

\section{Application 2: Multi-type Galton-Watson trees}
\label{sec:App2}

Let $\zeta=(\zeta^{(i)},i\in[\kappa])$ be a set of offspring distributions for a $\kappa$-type Galton-Watson tree: for each $i\in[\kappa]$, $\zeta^{(i)}$ is a probability distribution on $(\Z_+)^\kappa$, and we consider a branching population such that, for $i\in[\kappa]$ and $\mathbf{z}=(z_1,\ldots,z_{\kappa})\in (\Z_+)^{\kappa}$, the probability for an individual with type $i$ to have $z_j$ individuals of type $j$ for all $j\in[{\kappa}]$ is $\zeta^{(i)}(\mathbf{z}).$ We call $T^{(i)}$ the Galton-Watson tree with set of offspring distributions $\zeta$, started at an individual with type $i$, and recall that we consider unordered versions of all trees.

Let $M=(m_{i,j})_{i,j\in[{\kappa}]}$ be the mean matrix of $\zeta$, defined by
\[m_{i,j}=\sum_{\mathbf{z}\in(\Z_+)^{\kappa}} \zeta^{(i)}(\mathbf{z})z_j.\]
We make the standard assumptions that $\zeta$ is \emph{non-singular} (for at least one type, there is a possibility of having two or more children), that $M$ is \emph{finite} and \emph{irreducible} in the Perron-Frobenius sense, and that $\zeta$ is \emph{critical}, meaning that the largest eigenvalue of $M$ is $1$. We let $\mathbf{a}$ and $\mathbf{b}$ be the positive corresponding left and right positive eigenvectors of $M$, normalised such that $\sum_{i=1}^{\kappa} a_i =\sum_{i=1}^{\kappa} a_ib_i=1$.

We also assume that $\zeta^{(i)}$ has finite second moments for all $i\in[{\kappa}]$, and define the following quantities:
\begin{align*}
&Q^{(i)}_{j,k}:=\sum_{\mathbf{z}\in(\Z_+)^{\kappa}} \zeta^{(i)}(\mathbf{z}) z_j\big(z_k-\mathbf 1_{\{k=j\}}\big),\quad i,j,k\in[{\kappa}], \\
&\sigma^2:=\sum_{1 \leq i,j,k \leq \kappa} a_ib_jb_kQ^{(i)}_{j,k}.
\end{align*}

Under these assumptions the trees $T^{(i)}$ are all a.s.\ finite and our aim is to prove a scaling limit theorem for $T^{(i)}$ conditioned to be large, in the sense that we condition it on having $n$ vertices of type $1$ and let $n$ tend to infinity. For this we will use the underlying multi-type MB structure of these trees. Below, we use the notation $\#_1 T^{(i)}$ to refer to the number of vertices of type $1$ in $T^{(i)}$, and also make the simplifying aperiodicity assumption that $\pr(\#_1 T^{(i)}=n)>0$ for all $i\in[{\kappa}]$ and $n$ large enough, otherwise we would have to restrict $n$ to a sublattice of $\Z$.

Our theorem is the following:
\begin{thm}\label{thGW} For $i\in [{\kappa}]$, let $T_n^{(i)}$ be a version of $T^{(i)}$ conditioned on having $n$ vertices of type $1$, equipped with the uniform measure $\mu_n^{(i)}$ on these vertices. We then have the following convergence in distribution for the GHP-topology:
\[\bigg(\frac{T_n^{(i)}}{\sqrt{n}},\mu_n^{(i)}\bigg)\overset{(d)}{\underset{n \rightarrow \infty} \longrightarrow} \bigg( \frac{2}{\sigma\sqrt{a_1}} \cdot \T_{\mathrm{Br}},\mu_{\mathrm{Br}} \bigg),\]
where $(\T_{\mathrm{Br}},\mu_{\mathrm{Br}})$ is the Brownian CRT.
\end{thm}

The one-type version of Theorem \ref{thGW} is the classical theorem of Aldous (\cite{Ald93}).  A slightly weaker multi-type version is obtained by Miermont in \cite{M08}, where it is also assumed that the offspring distributions have some finite exponential moments. Both Aldous and Miermont's proofs consisted of studying the contour functions of the trees.  We use here a different strategy, based on the MB property. We will see that we are in the mixing regime, and Theorem \ref{thGW} will be proved by using Theorem \ref{theo:melange}. In \cite{HM12}, Aldous' theorem was recovered similarly by using the monotype MB property of monotype Galton-Watson trees.

\textbf{Organization of the rest of the section.} After recalling some technical details, our proof will be split into four main parts. First we identify the MB structure of the Galton-Watson trees and give their splitting distributions, in a slightly more general setting than Section \ref{sec:the model} since we will allow for vertices with  size 0. Then we verify assumption (\ref{hypocvmixing2}) (showing that the types mix, with $\beta=0$), and then assumption (\ref{hypocvmixing}) (showing that we have a convergence to $\nu_{\mathrm{Br}}$ with $\gamma=1/2.$) Finally, since we use an alternate MB structure, we need one additional section to show that the vertices with size 0 do not alter the metric structure on the $n^{1/2}$ scale.

\begin{rem}
Several steps of our proof are easily, if not immediately, adapted to other frameworks (conditionings other than by the number of vertices of a given type, or the case where the second moments may be infinite). However, not all do, and in particular proving that assumption (\ref{hypocvmixing}) holds in a more general framework is delicate, and will be the subject of future work. In particular, this should allow us to describe the scaling limits of multi-type Galton-Watson trees with offspring distributions in the domain of attraction of stable laws, hence completing the work of Berzunza \cite{Berzunza18} on scaling limits of forests of such Galton-Watson trees. 
\end{rem}

\subsection{Elementary preliminaries}

We recall that $(\mathbf{e}_i,i\in[{\kappa}])$ denotes the canonical basis of $\R^{\kappa}.$

\noindent{\textbf{Forest notation.}} To simplify notation, we will write, for $\mathbf{z}\in(\Z_+)^{\kappa}$, $\pr^{(\mathbf{z})}$ for a probability distribution under which the variable $F$ is a forest made of $$|\textbf z|:=\sum_{i=1}^{\kappa} z_i$$ independent Galton-Watson trees, with $z_i$ trees having root of type $i$ for all $i\in[{\kappa}]$. In the case where $\mathbf z=\mathbf e_i$ for some  $i\in[{\kappa}]$, we will keep the notation $T^{(i)}$ instead of $F$.

We now recall a few useful tools from \cite{M08} and \cite{Steph15} concerning the multi-type Galton-Watson structure.

\noindent{\textbf{Reduced forests and trees.}} For any ${\kappa}$-type forest $F$, we let $\Pi^{(1)}(F)$ be the monotype forest whose vertices are exactly the vertices of type $1$ in $F$, and which has the same ancestral relations as $F$. It was shown in \cite{M08} that, in the case where $F=T^{(1)},$ $\Pi^{(1)}(T^{(1)})$ is also a critical Galton-Watson tree, and the variance $\sigma_1^2$ of its offspring distribution -- denoted by $\bar{\zeta}_{1,1}$ -- satisfies
\[\sigma_1^2=\frac{\sigma^2}{a_1b_1^2}.\]

\noindent{\textbf{First generation of type 1.}} One operation which will be useful is, in a tree or a forest, to look at all the vertices of type $1$ which have no ancestors of that type. We call this the first generation of type $1$ in the tree or forest. If $F$ has distribution $\pr^{(\mathbf{z})}$ we call 
\begin{center}
$\zeta_{\mathbf{z},1}$: the distribution of the number of vertices in the first generation of type $1$ of $F$. 
\end{center}
Note that $\zeta_{\mathbf e^{(1)},1}=\delta_1$: if the type of the root of a tree is $1$, then this root is the first generation of type $1.$ By \cite[Proposition 2.1, (i)]{Steph15} the mean of $\zeta_{\mathbf{z},1}$ is $b_{\mathbf{z}}/b_1,$ where $b_{\mathbf{z}}$ is defined as
\[b_{\mathbf{z}}:=\mathbf{b}\cdot \mathbf{z}=\sum_{j=1}^{{\kappa}} b_jz_j.\] It is immediate that, in the finite variance case, the variance of $\zeta_{\mathbf{z},1}$ is smaller than $C|\mathbf{z}|$ for a certain $C>0$.

\noindent\textbf{General estimate for the distribution of number of vertices of type $1$.} It was proved in \cite[Section 4.1]{Steph15} that for all $\mathbf{z}\in(\Z_+)^{\kappa}$ and $q\in\Z$, 
\begin{equation}\label{byword}
\pr^{(\mathbf{z})}(\#_1 F=n+q) \underset{n\to\infty}{\sim} \frac{b_{\mathbf{z}}}{b_1} \pr \big(\#_1 T^{(1)}=n\big).
\end{equation}
A significant part of the proof of Theorem \ref{thGW} will consist in obtaining and using a refined version of this, see Proposition \ref{prop:betterbyword}  below.

\subsection{The MB property}

For any ${\kappa}$-type tree $T$ (or forest $F$), we let $\#_1 T$ (or $\#_1 F)$ be its number of vertices of type $1$. For a vertex $u\in T$, we let $T_u$ be the subtree of $T$ rooted at $u$. In the case where $u$ is a vertex of $T_n^{(i)},$ we give $u$ a size equal to $\#_1(T_n^{(i)})_u.$ Note that this might be $0.$ Hence we need to use here an extended notion of MB tree where we allow individuals to have size $0$, and also to reproduce into children of size $0$. To fit with this,  we also use the different set of partitions $\overline{\p}^0_n,$ which is the same set as $\overline{\p}_n$ except for the fact that partitions can have parts of the form $(0,j)$ for $j\in[{\kappa}]$. In Section \ref{sec:GWzero}, we will show that pruning away the zero-size vertices does not change the GHP limit, and will actually apply Theorem \ref{theo:melange} to the pruned tree. Note that, unlike in Section \ref{sec:App1}, the fragments with size $0$ have a role in the combinatorial structure and so we only do this pruning after obtaining the wanted limit properties of the splitting distributions.

For $\bar{\lambda}\in\overline{\p}^0_n$ for any $n$, let $\mathbf{z}(\bar{\lambda})=(z_j(\bar{\lambda}),j\in[{\kappa}])$ where for all $j\in[{\kappa}]$, $$z_j(\bar{\lambda})=\#\{m\in\{1,\ldots,p(\bar{\lambda})\}:i_m=j\}$$ is the number of parts of $\bar{\lambda}$ with type $j$. We also let $m_{(\ell,j)}(\bar{\lambda)}$ be the multiplicity of the term $(\ell,j)\in{\Z_+\times [{\kappa}]}$ in $\bar{\lambda}.$

\begin{prop}\label{prop:GWisMB} The sequence $(T_n^{(i)})$ is a multi-type \emph{MB}-sequence with splitting distributions $(q_n^{(i)})$ given for each $i \in [{\kappa}],n \in \mathbb N$ and each $\bar{\lambda}=(\lambda_k,i_k)_{1\leq k \leq p(\bar{\lambda})}\in\overline{\p}^0_n$ such that $\sum_{k=1}^{p(\bar{\lambda})} \lambda_k=n-\mathbf 1_{\{i=1\}}$ by:
\[ q_n^{(i)}(\bar{\lambda}):=\zeta^{(i)}(\mathbf{z}(\bar{\lambda}))\frac{\prod_{j\in[{\kappa}]} z_j(\bar{\lambda})!}{\displaystyle\prod_{{(l,j)}\in\Z_+\times[{\kappa}]}m_{(\ell,j)}(\bar{\lambda})!}\frac{\prod_{k=1}^{p(\bar{\lambda})}\pr\left(\#_1 T^{(i_k)}=\lambda_k\right)}{\pr\left(\#_1 T^{(i)}=n\right)},\]
and 0 otherwise.
\end{prop}

The proof is straightforward, and we omit it here, referring to \cite{HM12} for a proof in the monotype case. 
This result is more practically interpreted thus: if $\bar{\Lambda}$ has distribution $q_n^{(i)},$ then the distribution of $\mathbf{z}(\bar{\Lambda})$ is given by
\[q_n^{(i)}(\mathbf{z}(\bar{\Lambda})=\mathbf{z})=\zeta^{(i)}(\mathbf{z})\frac{\pr^{(\mathbf{z})}(\#_1 F=n-\mathbf{1}_{\{i=1\}})}{\pr^{(i)}(\#_1 T=n)}\]
for $\mathbf{z}\in(\Z_+)^{\kappa}.$ Conditionally on $\mathbf{z}(\bar{\Lambda})=\mathbf{z},$ let $F=(T_1,\ldots,T_{|\mathbf{z}|})$ have distribution $\pr^{(\mathbf{z})}$ but be conditioned on satisfying $\#_1 F=n-\mathbf{1}_{\{i=1\}},$ and let $i_k$ be the type of the root of $T_k$ for $1\leq k \leq |\mathbf{z}|.$ Then $\bar{\Lambda}$ has the same distribution as the lexicographically decreasing rearrangement of $\big((\#_1 T_k,i_k),1\leq k\leq |\mathbf{z}|\big)$.

\subsection{Change of type}

We show here assumption (\ref{hypocvmixing2}), with $\beta=0$. This stays in fact true without the second moment condition. We also identify the invariant distribution of the resulting $\mathsf Q$-matrix.

\begin{prop}\label{prop:cvtypesGW} For all types $i$ and $j$, we have
\[
\underset{n\to\infty}\lim q_n^{(i)}(i_1=j)=\frac{1}{b_i} \sum_{\mathbf{z}\in (\Z_+)^{\kappa}}\zeta^{(i)}(\mathbf{z})z_j b_j.
\]
\end{prop}
This is not surprising: the limit is none other than the probability that the second element of the spine in the size-biased Kesten tree $\widehat{T}^{(i)}$ associated to $\zeta$ has type $j$, and it is known from \cite{Steph15} that, when $n$ is large, $T_n^{(i)}$ is locally close in distribution to this Kesten tree. Let us recall the construction of $\widehat{T}^{(i)}$: it possesses an infinite line of descent called a \emph{spine}, along which the offspring distributions are the size-biased version of $\zeta$, defined by 
\begin{equation}
\label{eq:biasspine}
\widehat{\zeta}^{(j)}(\mathbf{z}):=\frac{b_{\mathbf{z}}}{b_j}\zeta^{(j)}(\mathbf{z})
\end{equation}
for $j\in[{\kappa}]$ and $\mathbf{z}\in(\Z_+)^{\kappa}.$
Outside of the spine, the genealogy uses the usual offspring distributions $\zeta$, and finally, given that the offspring of one element of the spine is $\mathbf{z}$, its successor in the spine has type $j$ with probability proportional to $z_jb_j$. This description will be useful in the proof.

\begin{proof} Fix the types $i,j$, and then fix $a\in\N$ and $\bar{\lambda}'=\left((\lambda_2,i_2),(\lambda_3,i_3),\ldots,(\lambda_p,i_p)\right)$ such that $\sum_{m=2}^p \lambda_m= a-\mathbf{1}_{\{i=1\}}$, and consider $\bar{\lambda}=\left((n-a,j),\bar{\lambda'}\right) \in \overline{\p}^0_n.$ Call $\mathbf{z}=\mathbf{z}(\bar{\lambda}'),$ such that $\mathbf{z}(\bar{\lambda})=\mathbf{z}+\mathbf{e}_{j}.$ Let us show that $q_n^{(i)}(\bar{\lambda})$ converges as $n$ tends to infinity, and that the limit is a probability distribution (formally a distribution on the set $[{\kappa}]\times \cup_{n=1}^{\infty}\overline{\p}^0_n$). First, noting that, for $n$ large, $m_{(n-a,j)}(\bar{\lambda})=1,$ rewrite $q_n^{(i)}(\bar{\lambda})$ as

\[
q_n^{(i)}(\bar{\lambda})=\zeta^{(i)}(\mathbf{z}(\bar{\lambda}))\frac{\prod_{k\in[{\kappa}]} z_k(\bar{\lambda})!}{\displaystyle\prod_{{(\ell,k)}\in\Z_+\times[{\kappa}]}m_{(\ell,k)}(\bar{\lambda}')!}\frac{\pr\left(\#_1 T^{(j)}=n-a\right)}{\pr\big(\#_1 T^{(i)}=n\big)}\prod_{k=2}^{p}\pr\big(\#_1 T^{(i_k)}=\lambda_k\big). \]

By (\ref{byword}), the second fraction here converges to $\frac{b_j}{b_i}.$ We then have 
\begin{align*}
q_n^{(i)}\big((n-a,j),\bar{\lambda}'\big)&\underset{n\to\infty}\longrightarrow
 \frac{b_j}{b_i}\zeta^{(i)}(\mathbf{z}+\mathbf{e}_j)(z_j+1)
\frac{\displaystyle\prod_{k\in[{\kappa}]} z_k!}{\displaystyle\prod_{{(\ell,k)}\in\Z_+\times[{\kappa}]}m_{(l,k)}(\bar{\lambda}')!}\prod_{k=2}^{p}\pr\big(\#_1 T^{(i_k)}=\lambda_k\big).
\end{align*}
However, this limit is exactly the probability that, in $\widehat{T}^{(i)}$, the second element of the spine has type $j$, while the types of the rest of the offspring of the root, and the sizes of the corresponding subtrees are given by $\bar{\lambda}'$. Since these add up to $1$, we have shown that, if $\bar{\Lambda}$ has law $q_n^{(i)}$, then $\big(i_1,\big((\lambda_2,i_2),\ldots,(\lambda_{p(\bar{\lambda})},i_{p(\bar{\lambda})})\big)\big)$ converges in distribution (in $[{\kappa}]\times \cup_{n=1}^{\infty}\overline{\p}^0_n$). In particular, $i_1$ by itself converges in distribution, to the type of the second element of the spine.
\end{proof}

We let $Q$ be the $\mathsf Q$-matrix such that $$Q_{i,j}:=\frac{1}{b_i} \sum_{\mathbf{z}\in (\Z_+)^{\kappa}}\zeta^{(i)}(\mathbf{z})z_j b_j =\frac{b_j m_{i,j}}{b_i}\quad for \quad i\neq j.$$ It is then easily checked that $Q$ is irreducible, and that its invariant distribution $\chi$ satisfies 
\[\chi_i=a_ib_i \qquad \text{ for all } i\in[{\kappa}].\]

\subsection{Convergence of the splitting distributions}

For $n\in\N$ and $i\in[{\kappa}]$, let $\nu_n^{(i)}$ be the distribution of $\frac{1}{n}\Lambda,$ without the types, where $\bar{\Lambda}$ has distribution $q_n^{(i)}$ of Proposition \ref{prop:GWisMB}. Our main objective here is to prove assumption (\ref{hypocvmixing}) in this form:
\begin{prop}\label{prop:cvbrownien}
On $\mathcal S^{\downarrow}$, for all $i \in [{\kappa}]$,
\[\sqrt{n}(1-s_1)\nu_n^{(i)} \ \underset{n\to\infty}{\overset{weakly}{\longrightarrow}} \ \frac{\sum_{1 \leq j,k \leq {\kappa}} b_jb_k Q^{(i)}_{j,k}}{2b_1b_i\sigma_1}(1-s_1)\nu_{\mathsf{Br}}.\]
\end{prop}
This does fit with Theorem \ref{thGW} as, taking the average with respect to the distribution $\chi$ found in the previous section, we have
\begin{equation}\label{eq:melangebrownien}
\frac{1}{2b_1\sigma_1}\sum_{1\leq i,j,k \leq {\kappa}} \frac{\chi_ib_jb_kQ^{(i)}_{j,k}}{b_i}=\frac{\sqrt{a_1}}{2\sigma}\sum_{1 \leq i,j,k \leq {\kappa}} a_ib_jb_kQ^{(i)}_{j,k}=\frac{\sigma\sqrt{a_1}}{2},
\end{equation}
and it is readily checked that, for any multi-type fragmentation tree $\T_{\gamma,\bnu}$ and any $c>0$, $\T_{\gamma,c\bnu}\overset{(d)}=\frac{1}{c}\T_{\gamma,\bnu}.$
After proving Proposition \ref{prop:cvbrownien}, the only part missing in the proof of Theorem \ref{thGW} will then be the problem of subtrees with size $0$, which is treated in Section \ref{sec:GWzero}.

\subsubsection{Precise estimate for the number of vertices of type $1$}
The proof of Proposition \ref{prop:cvbrownien} hinges on the following improvement of Equation (\ref{byword}), which is where the more interesting aspects of the multi-type structure appear:
\begin{prop}\label{prop:betterbyword} 
For $\mathbf{z}\in (\Z_+)^{\kappa}$,  we have
\[\pr^{(\mathbf{z})}\big(\#_1 F=n\big) = \frac{b_{\mathbf{z}}}{b_1}\frac{1}{\sqrt{2\pi\sigma_1^2n^3}}\big(1-g(\mathbf{z},n)+o(1)\big)\]
where:
\begin{itemize}
\item $o(1)$ is uniform in $\mathbf{z}:$ it is a function $h$ on $(\Z_+)^{\kappa}\times \N$ such that, for all $\veps>0$, there exists $N\in\N$ such that for $n>N$, $|h(\mathbf{z},n)|<\veps$ independently of $\mathbf{z}\in(\Z_+)^{\kappa}.$
\item $g$ is a  function on $(\Z_+)^{\kappa}\times \N$ with values in $[0,1]$ which is $o(1)$ as $|\mathbf{z}|n^{-1/2}$ tends to $0$: for all $\eta>0,$ there exists $\veps>0$ , for $n\in\N$ and $\mathbf{z}\in(\Z_+)^{\kappa}$ with $|\mathbf{z}|\leq \veps\sqrt{n}$, $g(\mathbf{z},n)\leq \eta.$
\end{itemize}
\end{prop}

\begin{proof} Note that if $\mathbf{z}=(0,\ldots,0)$ then there is nothing to say, so we assume $\mathbf{z}\neq(0,\ldots,0).$ We start with the case where $\mathbf{z}=(p,0,0\ldots,0),$ with $p\in\N.$ In this case we can directly work on the monotype reduced forest $\Pi^{(1)}(F)$, and it is well known from the Otter-Dwass formula that
\[\pr^{(p,0,\ldots)}(\#_1 F=n)=\frac{p}{n}\pr(S_n=-p)\]
where $(S_r,r\geq 0)$ is a random walk with step distribution $(\bar{\zeta}_{1,1}(k+1),k\geq -1).$ Moreover the local limit theorem, in the non-lattice, finite variance case also tells us that
\[\pr(S_n=-p)=\frac{1}{\sqrt{2\pi n\sigma_1^2}}e^{-p^2/2n\sigma_1^2}+o\left(\frac{1}{\sqrt{n}}\right),\]
where $o\big(\frac{1}{\sqrt{n}}\big)$ is uniform in $p$, and setting $g((p,0,\ldots),n)=1-e^{-p^2/2n\sigma_1^2}$ ends our first case.

Now take general $\mathbf{z}$. To study the number of vertices of type $1$ in the forest, we first go to its first generation of type $1$, its size having distribution $\zeta_{\mathbf{z},1}$ and expectation $b_{\mathbf{z}}/b_1.$ We have	
\begin{align*}
\pr^{(\mathbf{z})}(\#_1 F=n)&=\sum_{p=0}^{\infty} \zeta_{\mathbf{z},1}(p)\pr^{(p,0,\ldots)}(\#_1 F=n)\\
	&=\sum_{p=0}^{\infty} \zeta_{\mathbf{z},1}(p)\frac{p}{\sqrt{2\pi\sigma_1^2n^3}}(e^{-p^2/2n\sigma_1^2}+o(1))\\
	&=\frac{1}{\sqrt{2\pi\sigma_1^2n^3}}\left(\sum_{p=0}^{\infty} \zeta_{\mathbf{z},1}(p)p(1+o(1))-\sum_{p=0}^{\infty} \zeta_{\mathbf{z},1}(p)p(1-e^{-p^2/2n\sigma_1^2})\right)\\
	&=\frac{1}{\sqrt{2\pi\sigma_1^2n^3}}\left( \frac{b_{\mathbf{z}}}{b_1}(1+o(1))-\sum_{p=0}^{\infty} \zeta_{\mathbf{z},1}(p)p(1-e^{-p^2/2n\sigma_1^2})\right).
\end{align*}
Note that the last step is justified by the uniformity in $p$ of $o(1).$ Let us therefore set
\[g(\mathbf{z},n)=\frac{b_1}{b_{\mathbf{z}}}\sum_{p=0}^{\infty} \zeta_{\mathbf{z},1}(p)p(1-e^{-p^2/2n\sigma_1^2})=\frac{b_1}{b_{\mathbf{z}}}\E\Big[P(1-e^{-P^2/2n\sigma_1^2})\Big]\]
where $P$ is a variable with distribution $\zeta_{\mathbf{z},1}$, and we now want to check that $g(\mathbf{z},n)$ has limit $0$ when $|\mathbf{z}|n^{-1/2}$ tends to $0$. We know that there is $c>0$ such that $\E[P]\leq c|\mathbf{z}|$ and $\Var(P)\leq c|\mathbf{z}|$ for all $\mathbf{z}.$ Thus, for $\delta>c$, $\pr(P\geq \delta |\mathbf{z}|\,)\leq \frac{c}{(\delta-c)^2|\mathbf{z}|}.$ Up to taking larger $c$ such that $b_1\leq c b_i$ for all $i$ and $c\geq 1,$ we then have
\begin{align*}
g(\mathbf{z},n)&\leq \frac{c}{|\mathbf{z}|}\E\Big[P(1-e^{-P^2/2n\sigma_1^2})\Big] \\
               &\leq \frac{c}{|\mathbf{z}|}\left(\E\left[P(1-e^{-P^2/2n\sigma_1^2})\mathbf{1}_{\{P< \delta |\mathbf{z}|\}}\right]+\E\left[P(1-e^{-P^2/2n\sigma_1^2})\mathbf{1}_{\{P\geq\delta |\mathbf{z}|\}}\right]\right) \\
               &\leq \frac{c}{|\mathbf{z}|}\left(\delta |\mathbf{z}|(1-e^{-\delta^2 |\mathbf{z}|^2/2n\sigma_1^2})+\E\left[P\mathbf{1}_{\{P\geq\delta |\mathbf{z}|\}}\right]\right)\\
               &\leq c\delta(1-e^{-\delta^2 |\mathbf{z}|^2/2n\sigma_1^2})+\frac{c}{|\mathbf{z}|}\sqrt{\E[P^2]\pr(P\geq \delta |\mathbf{z}|)}\\
               &\leq  c\delta(1-e^{-\delta^2 |\mathbf{z}|^2/2n\sigma_1^2}) + \frac{c\sqrt{c(c^2|\mathbf{z}|^2+c|\mathbf{z}|)}}{(\delta-c)|\mathbf{z}|^{3/2}}\\
               &\leq  c\delta(1-e^{-\delta^2 |\mathbf{z}|^2/2n\sigma_1^2}) + \frac{c^{5/2}\sqrt{2}}{\delta-c}.
\end{align*}
The last inequality comes from the fact that $c\geq1$ and $|\mathbf{z}|\geq 1$ (the latter being a nonzero integer). Now let $\eta>0$, and take $\delta$ large enough that $\frac{c^{5/2}}{\delta-c}\leq \frac{\eta}{2}.$ Now take $\veps$ small enough such that $c\delta(1-e^{-\delta^2\veps^2/2\sigma_1^2})\leq \frac{\eta}{2},$ then we do have $g(\mathbf{z},n)\leq \eta$ whenever $|\mathbf{z}|\leq \veps\sqrt{n}.$
\end{proof}

\subsubsection{Proof of Proposition \ref{prop:cvbrownien}}

Once we have Proposition \ref{prop:betterbyword} in hand, we can prove Proposition \ref{prop:cvbrownien}. A monotype version of this result is proved in \cite[Proposition 39]{HM12}. We will use the same structure as the monotype proof, relying on size-biased reorderings and appropriate truncations. 

We start with a lemma on the number of children of the root in $T_n^{(i)}$. Recall that for a partition $\bar{\lambda},$ $p(\bar{\lambda})=|\mathbf{z}(\bar{\lambda})|$ denotes its number of parts.
\begin{lem}\label{lem:partsnumber} 
For any $\veps>0$,
\[\sqrt{n} q_n^{(i)}(p(\bar{\lambda})>\veps\sqrt{n})\underset{n\to\infty} \longrightarrow 0\]
\end{lem}

\begin{proof} Since the function $g$ appearing in Proposition \ref{prop:betterbyword} is nonnegative, we have, for large enough $n$,
\[q_n^{(i)}(\mathbf{z}(\bar \lambda)=\mathbf{z})=\zeta^{(i)}(\mathbf{z})\frac{\pr^{(\mathbf{z})}(\#_1 F=n-\mathbf{1}_{\{i=1\}})}{\pr^{(i)}(\#_1 T=n)}\leq 2\widehat{\zeta}^{(i)}(\mathbf{z}), \quad \forall \mathbf z \in (\mathbb Z_+)^{\kappa}\]
with $\widehat{\zeta}^{(i)}(\mathbf{z})$ defined by (\ref{eq:biasspine}). 
However since under $\widehat{\zeta}^{(i)}$, $|\mathbf{z}|$ has finite mean (since ${\zeta}^{(i)}$ has finite second moments), we have $\widehat{\zeta}^{(i)}(|\mathbf{z}|>p)=o(p^{-1}),$ and the result follows.
\end{proof}
\noindent \textbf{Size-biased ordering.} Given a finite or infinite indexed set $(x_r,r\in\mathcal{R})$ of nonnegative numbers with finite sum, we can define its size-biased ordering $(x_1^*,x_2^*,\ldots)$ the following way: let $r^*$ be a random index such that 
\[ \pr(r^*=r)=\frac{x_r}{\sum_{n \in \mathbb N} x_n},\] 
let $x_1^*=x_{r^*}$, remove $r^*$ from $\mathcal{R}$ and proceed inductively.

This procedure can also be applied to measures on the set of sequences \[\s=\bigg\{\mathbf{x}=(x_n)_{n\in\N}\in [0,1]^{\N}: \sum_{n\in \mathbb N} x_n \leq 1\bigg\}\]
(which is a compact metric space when endowed with the metric $d$ defined by $d(\mathbf{x},\mathbf{y}):=\linebreak \sum_n 2^{-n} |x_n-y_n|$). If $\mu$ is a measure on $\s^{\downarrow}$, then define $\mu^*$ by $\mu^*(f)=\int_{\s}\mathrm d\mu (\mathbf{s}) f(\mathbf s^*)$ where $\mathbf s^*$ is the random sequence $(s_1^*,s_2^*,\ldots).$ 
Note that $((1-s_1)\mu)^*=(1-\max \mathbf{x})\mu$ (in particular this holds for $\nu_{\mathrm{Br}}$) and that
\[\nu_{\mathsf{Br}}^*(f)=\frac{\sqrt{2}}{\sqrt{\pi}}\int_0^1 \frac{f(x,1-x,0,\ldots)}{x^{1/2}(1-x)^{3/2}}\mathrm d x.
\]
Following \cite[Lemma 38]{HM12}, Proposition \ref{prop:cvbrownien} can be proved by instead showing
\begin{equation}
\label{cv:biased}
\sqrt{n}\big((1-s_1)\nu_n^{(i)}\big)^* \ \underset{n\to\infty}{\overset{weakly}{\longrightarrow}} \ \frac{\sum_{j,k} b_jb_k Q^{(i)}_{j,k}}{2b_1b_i\sigma_1}\big((1-s_1)\nu_{\mathsf{Br}}\big)^*.
\end{equation}
The following lemma contains the tools necessary to prove this. It relies strongly on Proposition \ref{prop:betterbyword}, and its consequence Lemma \ref{lem:partsnumber}. 

\begin{lem}\label{biglemmabrownien} We have the following limiting properties of $(\nu_n^{(i)})^*$:
\begin{itemize}
\item[\emph{(i)}] $\underset{\eta\to 0}\lim\,\underset{n\to\infty}\limsup \;\sqrt{n}(\nu_n^{(i)})^*\big((1-x_1)\mathbf{1}_{\{x_1>1-\eta\}}\big)=0.$
\item[\emph{(ii)}] $\underset{n\to\infty}\lim \sqrt{n}(\nu_n^{(i)})^*\big(\mathbf{1}_{\{x_1<n^{-7/8}\}}\big)=0.$
\item[\emph{(iii)}] For any $\eta>0$, $\underset{n\to\infty}\lim \sqrt{n}(\nu_n^{(i)})^*(\mathbf{1}_{\{x_1+x_2<1-\eta\}})=0$
\item[\emph{(iv)}] There exists a function $\beta_{\eta}$ which is $o(\eta)$ as $\eta \to 0$, such that, for any function $f$ on $\s$ which can be written as $f(\mathbf{x})=(1- \max \mathbf{x})h(\mathbf{x})$ with $h$ continuous, we have
\begin{align*}
\underset{\eta \to 0}\lim\underset{n\to\infty}\liminf\; \sqrt{n}(\nu_n^{(i)})^*(f\mathbf{1}_{\{x_1<1-\eta,x_1+x_2>1-\beta_{\eta}\}})
&=\underset{\eta \to 0}\lim\underset{n\to\infty}\limsup\; \sqrt{n}(\nu_n^{(i)})^*(f\mathbf{1}_{\{x_1<1-\eta,x_1+x_2>1-\beta_{\eta}\}})\\
&=  \frac{\sum_{j,k}b_jb_kQ_{j,k}^{(i)}}{b_1b_i\sigma_1\sqrt{2\pi}}\int_0^1 \frac{f(x,1-x,0,\ldots)}{x^{1/2}(1-x)^{3/2}}\mathrm d x.  
\end{align*}
\end{itemize}
\end{lem}

\begin{proof} The proofs of each item all use the same tools, namely Proposition \ref{prop:betterbyword} and some elementary properties of the size-biased ordering. Thus we will only focus on (ii) and (iv), and let the reader fill the rest in.

For any $\mathbf{z}\in (\Z_+)^{\kappa}$, we let $(X^{(j)}_m,j\in[{\kappa}],1\leq m \leq z_j)$ be independent variables such that $X^{(j)}_m$ has the distribution of $\#_1 T^{(j)}$. We call $(X^*_m,1\leq m\leq |\mathbf{z}|)$ their size-biased ordering and also let $(i^*_m,1\leq m\leq |\mathbf{z}|)$ be the matching ordering of their types. To be specific, this means that the size-biased order is obtained by taking $\mathcal{R}=\{(m,j)\in \N\times [{\kappa}]:1\leq m \leq z_k\}$, then $(X^*_1,i^*_1)=(X_{m^*}^{(i^*)},i^*)$ where $r^*=(m^*,i^*),$ and we proceed inductively. We also call $S_{\mathbf{z}}=\sum_{m=1}^{|\mathbf{z}|} X^*_m.$ Note that $S_{\mathbf{z}}$ has the same distribution as $\#_1 F$ under $\pr^{(\mathbf{z})},$ and note the relation 
\[\pr(X_1^*=m,i_1^*=j,S_{\mathbf{z}}=n)=\frac{z_jm}{n}\pr(\#_1 T^{(j)}=m)\pr(S_{\mathbf{z}-\mathbf{e}_j}=n-m)\]
for $n\in\N,$ $m\leq n$ and $j\in[{\kappa}],$ which follows from the definition of the size-biased order. We obtain $(ii)$ by writing, for $n'=n-\mathbf{1}_{\{i=1\}}$ and $\mathbf{z}'=\mathbf{z}-\mathbf{e}_j,$
\begin{align*}
\sqrt{n}(\nu_n^{(i)})^*&(\mathbf{1}_{\{x_1<n^{-7/8}\}})= \sqrt{n}\sum_{\mathbf{z}\in(\Z_+)^{\kappa}}q_n^{(i)}(\mathbf{z}(\bar{\Lambda})=\mathbf{z})\sum_{m=1}^{n^{1/8}}\sum_{j\in[{\kappa}]}\frac{\pr(X_1^*=m,i_1^*=j,S_{\mathbf{z}}=n')}{\pr(S_{\mathbf{z}}=n') } \\
         &= \sqrt{n}\sum_{\mathbf{z}\in(\Z_+)^{\kappa}}\zeta^{(i)}(\mathbf{z})\frac{\pr^{(\mathbf{z})}(\#_1 F=n')}{\pr^{(i)}(\#_1 T=n)}\sum_{m=1}^{n^{1/8}}\sum_{j\in[{\kappa}]}\frac{z_jm}{n}\frac{\pr(\#_1 T^{(j)}=m)\pr^{(\mathbf{z}')}(\#_1 F=n'-m)}{\pr^{(\mathbf{z})}(\#_1 F=n')}.
\end{align*}

Since the function $g$ in Proposition \ref{prop:betterbyword} is nonnegative, we have the existence of a constant $C$ such that $\pr^{(\mathbf{z}')}(\#_1 F=n'-m) \leq C|\mathbf{z}|\pr^{(i)}(\#_1 T=n)$ uniformly in $i,j,n,m$ (with $m\leq n^{1/8}$) and $\mathbf{z}.$ Bounding moreover $\pr(\#_1 T^{(j)}=m)$ by $1$ yields
\[
\sqrt{n}(\nu_n^{(i)})^*\big(\mathbf{1}_{\{x_1<n^{-7/8}\}}\big)\leq Cn^{-3/8} \sum_{\mathbf{z}\in(\Z_+)^{\kappa}} \zeta^{(i)}(\mathbf{z}) |\mathbf{z}|^2,\]
which ends the proof of (ii) because the sum is finite.

As for (iv), we will focus on treating its $\liminf$ part. Let $\eta>0$ be fixed, and let $\veps>0$ and $0<\eta'<\eta$, to be specified later. By Propositions \ref{prop:GWisMB} and \ref{prop:betterbyword}, we have, keeping the same notation as earlier,
\begin{align} 
\label{eq:step1}
\sqrt{n}&(\nu_n^{(i)})^*(f)=\\
\nonumber
&\sqrt{n}\sum_{\mathbf{z}\in (\Z_+)^{\kappa}}\zeta^{(i)}(\mathbf{z})\,\frac{b_{\mathbf{z}}}{b_i}\left(\frac{n}{n'}\right)^{3/2}\frac{1-g(\mathbf{z},n')+o(1)}{1-g(\mathbf e_{i},n)+o(1)}\E\left[f\left(\frac{X_1^*}{n},\ldots,\frac{X_{|\mathbf{z}|}^*}{n},0,\ldots\right)\mid S_{\mathbf{z}}=n'\right].
\end{align}
Setting 
$$
\E_{\mathbf z,n,n'}^{\eta,\eta'}[f]:=\E\left[f\left(\frac{X_1^*}{n},\ldots,\frac{X_{|\mathbf{z}|}^*}{n},0,\ldots)\right)\mathbf{1}_{\{ n^{1/8}\leq X_1^*< (1-\eta)n,X_1^*+X_2^*>(1-\eta')n\}}\mid S_{\mathbf{z}}=n'\right],
$$
replacing $f$ by $f\mathbf{1}_{\{x_1<1-\eta,x_1+x_2>1-\eta'\}}$ in (\ref{eq:step1}) and using Lemma \ref{lem:partsnumber} and point (ii), we then have
\begin{align*}
\sqrt{n}&(\nu_n^{(i)})^*(f\mathbf{1}_{\{x_1<1-\eta,x_1+x_2> 1-\eta'\}})\\
&=\sqrt{n}\left(\frac{n}{n'}\right)^{3/2}\sum_{\mathbf{z}:2\leq |\mathbf{z}|\leq \veps n^{1/2}} \zeta^{(i)}(\mathbf{z})\frac{b_{\mathbf{z}}}{b_i}\frac{1-g(\mathbf{z},n')+o(1)}{1-g(\mathbf e_{i},n)+o(1)} \E_{\mathbf z,n,n'}^{\eta,\eta'}[f]+o(1). 
\end{align*}
By the properties of $g$ given in Proposition \ref{prop:betterbyword}, taking $\veps$ small enough, we have $\frac{1-g(\mathbf{z},n')+o(1)}{1-g(\mathbf e_{i},n)+o(1)}\geq 1-\eta$ for all $\mathbf z$ in the sum, for all $n$ large enough. Thus
\begin{equation*} 
\underset{n\to\infty}\liminf \sqrt{n}(\nu_n^{(i)})^*(f\mathbf{1}_{\{x_1<1-\eta,x_1+x_2>1-\eta'\}})\geq (1-\eta)\underset{n\to\infty}\liminf\;\sqrt{n}\sum_{\mathbf{z}:2\leq |\mathbf{z}|\leq \veps n^{1/2}} \zeta^{(i)}(\mathbf{z})\frac{b_{\mathbf{z}}}{b_i} 
\E_{\mathbf z,n,n'}^{\eta,\eta'}[f].
\end{equation*}
For each $\mathbf{z},$ rewrite $\E_{\mathbf z,n,n'}^{\eta,\eta'}[f]$ as
\begin{align*}
&\underset{\begin{subarray}{c}
  n^{1/8}\leq m_1<(1-\eta)n \\
  (1-\eta')n<m_1+m_2\leq n'
  \end{subarray}}\sum \sum_{j,k\in[{\kappa}]} \E\left[f\left(\frac{m_1}{n},\frac{m_2}{n},\frac{X_3^*}{n},\ldots \right) \mid X_1^*=m_1, X_2^*=m_2,i_1^*=j,i_2^*=k,S_{\mathbf{z}}=n'\right] \\
 & \hspace{4cm} \times\pr\left(X_1^*=m_1, X_2^*=m_2,i_1^*=j,i_2^*=k \mid S_{\mathbf{z}}=n'\right).
\end{align*}
Since $f$ is uniformly continuous on the subset of $\s$ where $x_1+x_2>3/4$, we have for $\eta'$ small enough,
\[\left|f\left(\frac{m_1}{n},\frac{m_2}{n},\frac{m_3}{n},\ldots\right)-f\left(\frac{m_1}{n},\frac{n-m_1}{n},0,\ldots\right) \right|\leq \eta\]
for any $(m_1,m_2,m_3,\ldots)$ with total sum $n'$ and $m_1+m_2\geq (1-\eta')n$. Elementary manipulations also show that
\begin{align*}
\pr(X_1^*=m_1, X_2^*=m_2,i_1^*&=j,i_2^*=k \mid S_{\mathbf{z}}=n')= \\
&\frac{z_j m_1}{n'}\pr(\#_1 T^{(j)}=m_1)\frac{z'_km_2}{n'-m_1}\pr(\#_1 T^{(k)}=m_2)\frac{\pr^{(\mathbf{z}'')}(\#_1 F=n'-m_1-m_2)}{\pr^{(\mathbf{z})}(\#_1 F=n')}
\end{align*}
where $\mathbf{z}'=\mathbf{z}-\mathbf{e}_j$ and $\mathbf{z}''=\mathbf{z}-\mathbf{e}_j-\mathbf{e}_k.$ Now,
notice that all $m_1$ and $m_2$ involved in our sum have lower bounds which tend to infinity. We can then apply Proposition \ref{prop:betterbyword} to obtain for large $n$
\begin{equation}\label{eq:bornefraction}
\frac{\pr(\#_1 T^{(j)}=m_1)\pr(\#_1 T^{(k)}=m_2)}{\pr^{(\mathbf{z})}(\#_1 F=n')}\geq (1-\eta)\frac{b_j b_k}{b_1 b_\mathbf{z}}\frac{1}{\sqrt{2\pi \sigma_1^2}}\left(\frac{n}{m_1m_2}\right)^{3/2},
\end{equation}  
uniformly in $j,k \in [\kappa]$ and $\mathbf z: |\mathbf{z}| \geq 2$.
We can then write for $n$ large enough, uniformly on $\mathbf{z}$ with $ |\mathbf{z}|\geq 2$,
\begin{align*}
\pr(&X_1^*=m_1, X_2^*=m_2,i_1^*=j,i_2^*=k \mid S_{\mathbf{z}}=n') \\
&\geq (1-\eta)z_jz'_k\frac{b_j b_k}{b_1 b_\mathbf{z}}\frac{1}{\sqrt{2\pi \sigma_1^2}}\left(\frac{n}{m_1m_2}\right)^{1/2}\frac{1}{n-m_1}\pr^{(\mathbf{z}'')}(\#_1 F=n'-m_1-m_2) \\
&\geq (1-\eta)z_jz'_k\frac{b_j b_k}{b_1 b_\mathbf{z}}\frac{1}{\sqrt{2\pi \sigma_1^2}} \frac{1}{n^{3/2}}\frac{1}{(m_1/n)^{1/2}(1-m_1/n)^{3/2}}\pr^{(\mathbf{z}'')}(\#_1 F=n'-m_1-m_2).\stepcounter{equation}\tag{\theequation}\label{myeq1}
\end{align*}
Putting everything together, we have
\begin{align*} 
& \underset{n\to\infty}\liminf  \sqrt{n}(\nu_n^{(i)})^*(f\mathbf{1}_{\{x_1<1-\eta,x_1+x_2>1-\eta'\}})\geq  \\
&(1-\eta)^2\sum_{j,k\in[{\kappa}]}\frac{b_j b_k}{b_ib_1\sqrt{2\pi\sigma_1^2}} \\ 
&\times \underset{n\to\infty}\liminf \hspace{-0.4cm} \sum_{\mathbf{z}:2\leq |\mathbf{z}|\leq \veps n^{1/2}} \hspace{-0.4cm} \zeta^{(i)}(\mathbf{z})z_jz'_k \hspace{-0.4cm} \sum_{n^{1/8}\leq m_1\leq (1-\eta)n} \hspace{-0.4cm} \left(f\left(\frac{m_1}{n},1-\frac{m_1}{n},0,\ldots\right)-\eta\right)\frac{1}{n}\frac{1}{(m_1/n)^{1/2}(1-m_1/n)^{3/2}} \\
&\times \sum_{(1-\eta')n'-m_1 \leq m_2 \leq n'-m_1}\pr^{(\mathbf{z}'')}(\#_1 F =n'-m_1-m_2).
\end{align*}
The last sum is equal to $\pr^{(\mathbf{z}'')}(\#_1 F\leq \eta'n')$, and can be made bigger than $1-\eta$ uniformly in our choice of $\mathbf{z}$ for large $n$, up to taking $\veps$ yet smaller. Indeed, letting $\mathbf{1}=(1,1,\ldots,1)\in(\Z_+)^{\kappa}$, we have $\pr^{(\mathbf{z}'')}(\#_1 F> \eta'n')\leq \pr^{(\lfloor \veps n^{1/2}\rfloor\mathbf{1})}(\#_1 F> \eta'n').$ By Proposition \ref{prop:betterbyword}, the latter term is, for large $n$, smaller than $C\veps\sqrt{n}\sum_{k\geq \eta' n'} k^{-3/2}\sim C'\veps,$ for two constants $C$ and $C'$, and thus choosing $\veps<\eta/C'$ fits.
Thus we can now write
\begin{align*} 
\underset{n\to\infty}\liminf\sqrt{n}&(\nu_n^{(i)})^*(f\mathbf{1}_{\{x_1<1-\eta,x_1+x_2>1-\eta'\}})\geq  \\
&(1-\eta)^3\,\underset{n\to\infty} \liminf \sum_{j,k\in[{\kappa}]}\frac{b_j b_k}{b_ib_1\sqrt{2\pi\sigma_1^2}} \sum_{\mathbf{z}:2\leq |\mathbf{z}|\leq \veps n^{1/2}}\zeta^{(i)}(\mathbf{z})z_jz'_k \\ 
&\times \sum_{n^{1/8}\leq m_1\leq (1-\eta)n} \left(f\left(\frac{m_1}{n},1-\frac{m_1}{n},0,\ldots\right)-\eta\right)\frac{1}{n}\frac{1}{(m_1/n)^{1/2}(1-m_1/n)^{3/2}} .
\end{align*}
The first sum converges to $\sum_{j,k\in[{\kappa}]}\frac{b_jb_kQ^{(i)}_{j,k}}{b_ib_1\sigma_1\sqrt{2\pi}},$ while the second one is a Riemann sum 
and thus converges to $\int_{0}^{1-\eta} x^{-1/2}(1-x)^{-3/2}(f(x,1-x,0,\ldots)-\eta) \mathrm dx.$ Letting $\eta$ tend to zero then ends the proof of the $\liminf$.

The proof of the $\limsup$ functions the same way. Proposition \ref{prop:betterbyword} yields upper bounds involving $1+\eta$ as it yielded lower bounds involving $1-\eta.$ The main differences are that the upper bound version of Equation (\ref{eq:bornefraction}) requires $|\mathbf{z}|\leq \veps\sqrt{n},$ and that the upper bound version of (\ref{myeq1}) uses $m_2^{-1/2}\leq (1-\eta'/\eta)^{-1/2}(n'-m_1)^{-1/2},$ thus involving the term $1-\eta'/\eta,$ hence we need to take $\eta'=o(\eta).$
\end{proof}

\noindent\textit{End of the proof of Proposition \ref{prop:cvbrownien}.} Recall that we just need to prove the size-biased convergence (\ref{cv:biased}). Let  $f(\mathbf{x})=(1- \max \mathbf{x})h(\mathbf{x})$ with $h$ continuous on $\s.$ For $\eta>0$ and $\eta'>0$, write
\[\Big|(\nu_n^{(i)})^*(f)-(\nu_n^{(i)})^*(f\mathbf{1}_{\{x_1<1-\eta,x_1+x_2>1-\eta'\}})\Big|\leq (\nu_n^{(i)})^*(|f|\mathbf{1}_{\{x_1\geq 1-\eta\}})+(\nu_n^{(i)})^*(|f|\mathbf{1}_{\{x_1+x_2\leq 1-\eta'\}}).\]
Let $\veps>0$. By Lemma \ref{biglemmabrownien}, we can choose $\eta$ and $\eta'$ such that 
$(\nu_n^{(i)})^*(|f|\mathbf{1}_{\{x_1\geq 1-\eta\}})\leq \veps$ and 
\[\Big|(\nu_n^{(i)})^*(f\mathbf{1}_{\{x_1<1-\eta,x_1+x_2>1-\eta'\}})-\frac{\sum_{j,k}b_jb_kQ_{j,k}^{(i)}}{b_1b_i\sigma_1\sqrt{2\pi}}\int_0^1 \frac{\mathrm d x}{x^{1/2}(1-x)^{3/2}}f(x,1-x,0,\ldots)\Big| \leq \veps \]
for large enough $n$. We then get
\[\Big|(\nu_n^{(i)})^*(f)-\frac{\sum_{j,k}b_jb_kQ_{j,k}^{(i)}}{b_1b_i\sigma_1\sqrt{2\pi}}\int_0^1 \frac{\mathrm d x}{x^{1/2}(1-x)^{3/2}}f(x,1-x,0,\ldots)\Big|\leq 2\veps + \sqrt{n}(\nu_n^{(i)})^*(|f|\mathbf{1}_{\{x_1+x_2\leq 1-\eta'\}}),\]
and since the last term tends to $0$, again by Lemma \ref{biglemmabrownien}, we finally obtain Proposition \ref{prop:cvbrownien}. \qed

\subsection{Zero mass subtrees are small}\label{sec:GWzero}

The MB approach to the study of Galton-Watson trees fails to capture the behaviour of subtrees which contain no vertices of type $1$, since Theorem \ref{theo:melange} does not normally allow for zero-mass vertices in the tree. This section is dedicated to showing that the subtrees with no vertices of type $1$ are small enough to disappear in the scaling limit, thus justifying the use of Theorem \ref{theo:melange}. Specifically, we prove the following:

\begin{prop}\label{prop:zero}
Let $\mathring{T}_n^{(i)}$ be the same tree as $T^{(i)}_n$, except that we have removed all vertices with size $0$, and endow both trees with the uniform measure on their vertices of type 1. Then there exists $C>0$ such that
\[\pr\big(d_{\mathrm{GHP}}(T^{(i)}_n,\mathring{T}^{(i)}_n )\leq C\log n\big)\underset{n\to\infty} \longrightarrow 1,\]
where $d_{\mathrm{GHP}}$ denotes the Gromov-Hausdorff-Prokhorov distance.
\end{prop}

This will be proved by showing that the subtrees appended to $\mathring{T}^{(i)}_n$ to obtain $T^{(i)}_n$ are all small in distribution, and there are not too many of them. Note that the Prokhorov part of this convergence is in fact immediate, since these subtrees have no mass.

We assume that there exists at least one $i\in[{\kappa}]$ such that $\pr(\#_1 T^{(i)}=0)>0$, otherwise $T^{(i)}_n=\mathring{T}^{(i)}_n$ for all $n$ and there is nothing to do. We recall that given their number and types, the subtrees removed from $T^{(i)}_n$ to get $\mathring{T}_n^{(i)}$ are independent Galton-Watson trees conditioned on not having vertices of type 1.

\noindent\textbf{Trees conditioned to be missing a type are subcritical.} Up to reordering the types, we can assume that $\pr(\#_1 T^{(i)}=0)>0$ if and only if $i\geq L$ for some $L\in\{2,\ldots,{\kappa}\}$. We let, for $i$ in $\{L,\ldots,{\kappa}\}$, $T^{(i)}_{\cross}$ be a ${\kappa}$-type Galton-Watson tree with offspring distributions $\zeta$ and root of type $i$, conditioned on not having any vertices of type $1$. By default, $T^{(i)}_{\cross}$ then only has vertices of types $L$ to ${\kappa}$.

\begin{prop} 
\label{prop:subcritical}
The tree $T^{(i)}_{\cross}$ is a subcritical ${\kappa}-L+1$-type Galton-Watson tree.
\end{prop}

We recall that irreducible Galton-Watson trees are called subcritical if the Perron eigenvalue of their mean matrix is strictly less than $1.$ However, the mean matrix of $T^{(i)}_{\cross},$ which we call $M^{\cross},$ is not necessarily irreducible, and so we need a more general notion of subcriticality. We sort the elements of $\{L,\ldots,\kappa\}$ into \emph{irreducible components} by saying that two types $i$ and $j$ are in the same component if there exists integers $n$ and $m$ such that $(M^{\cross})^n(i,j)>0$ and $(M^{\cross})^m(j,i)>0.$ In this case, up to reordering the set $\{L,\ldots,\kappa\}$, we can assume
\[M^{\cross}=\left({\begin{smallmatrix}B_{1}&*&*&\cdots &*\\0&B_{2}&*&\cdots &*\\\vdots &\vdots &\vdots &&\vdots \\0&0&0&\cdots &*\\0&0&0&\cdots &B_{h}\end{smallmatrix}}\right),\]
where the blocks correspond to each irreducible component. The eigenvalues of $M^{\cross}$ are then those of the $(B_k),$ and we say that the tree is \emph{subcritical} if the Perron eigenvalues of each are all strictly smaller than $1.$

\begin{proof} It is straightforward to see that the conditioning does not lose the branching property, and thus $T^{(i)}_{\cross}$ is a Galton-Watson tree. We call $\zeta_{\cross}^{(i)}$, for $i\in\{L,\ldots,\kappa\}$ the corresponding offspring distributions. Calling $p_i=\pr(\#_1 T^{(i)}=0)$, we have by definition the relation
\begin{equation}\label{eq:relationzeta}
\zeta_{\cross}^{(i)}(z_L,\ldots,z_{\kappa})=\frac{1}{p_i}\zeta^{(i)}(0,\ldots,0,z_L,\ldots,z_{\kappa})\prod_{j=L}^{\kappa}p_j^{z_j}
\end{equation}
for $\mathbf{z}=(z_L,\ldots,z_{\kappa})\in \mathbf{z}\in(\Z_+)^{\{L,\ldots,\kappa\}}.$ 
Define the generating functions $\mathbf{f}=(f_1,\ldots,f_{\kappa})$ and $\mathbf{f}^{\cross}=(f^{\cross}_L,\ldots,f^{\cross}_{\kappa})$ by, for $\mathbf{x}\in (\R_+)^{\kappa}$ and $i\in[\kappa],$
\[f_i(\mathbf{x})=\sum_{\mathbf{z}\in(\Z_+)^{\kappa}} \zeta^{(i)}(\mathbf{z})\prod_{j=1}^{\kappa} x_j^{z_j}\]
and, for $\mathbf{x}\in(\R_+)^{\{L,\ldots,\kappa\}}$ and $i\geq L,$
\[f^{\cross}_i(\mathbf{x})=\sum_{\mathbf{z}\in(\Z_+)^{\{L,\ldots,\kappa\}}} \zeta_{\cross}^{(i)}(\mathbf{z})\prod_{j=L}^{\kappa} x_j^{z_j}.\]
Note that we have $p_i=\sum_{\mathbf{z}\in(\Z_+)^{\{L,\ldots,\kappa\}}}\zeta^{(i)}(0,\ldots,0,z_L,\ldots,z_{\kappa})\prod_{k=L}^{\kappa}p_j^{z_j}$ for all $i$, implying the fixed point equation $\mathbf{f}(0,\ldots,p_L,\ldots,p_{\kappa})=(0,\ldots,p_L,\ldots,p_{\kappa}).$ Also, equation (\ref{eq:relationzeta}) implies
\[f^{\cross}_i(\mathbf{x})=\frac{1}{p_i}f_i(0,0,\ldots,p_Lx_L,\ldots,p_{\kappa}x_{\kappa}).\]
The matrix $M^{\cross}$ is the differential of $\mathbf{f}^{\cross}$ at $\mathbf{1}=(1,\ldots,1).$ Differentiating the above we have the equality of $\{L,\ldots,\kappa\}\times\{L,\ldots,\kappa\}$-indexed matrices:
\[M^{\cross}=P^{-1}N P\]
where $N$ is the matrix with entries $N_{i,j}=\partial_j f_i (0,\ldots,0,p_L,\ldots,p_K),$ and $P$ is the diagonal matrix with entries $(p_L,\ldots,p_{\kappa}).$ As a consequence, $M^{\cross}$ has the same eigenvalues as $N$, and we will focus on showing that the eigenvalues of $N$ are strictly less than one. 

Assume by contradiction that this is not the case, then by applying the Perron-Frobenius theorem in the irreducible component of $N$ whose largest eigenvalue is the highest (note that these components same as those of $M^{\cross}$), there exists a vector $\mathbf{x}=(x_L,\ldots,x_{\kappa})$ such that $N\mathbf{x}\geq \mathbf{x}$ (where the relation $\geq$ between vectors means that the comparison holds for each entry). Moreover $\mathbf{x}$ has nonnegative entries, and its nonzero entries correspond to a single irreducible component of $M^{\cross}.$

Now consider for $t\geq 0$ the vector $\mathbf{x}(t)=(0,\ldots,0,p_L+tx_L,\ldots,p_{\kappa}+tx_{\kappa}),$ and notice that $\mathbf{x}(t)\leq \mathbf{1}$ if $t\leq t_k:=\frac{1-p_k}{x_k}$ for all $k$ (and $t_k:=\infty$ if $x_k=0$ or $x_k$ is undefined). Choose $j$ for which $t_j$ is minimal, and consider the function with one variable: $g_j: t \mapsto f_j(\mathbf{x}(t)).$ As a polynomial in $t$ with nonnegative coefficients, $g_j$ is convex, and so $g_j(0)=p_j$ and $g_j'(0)\geq x_j$ imply together that $g_j(t_j)=f_j(\mathbf{x}(t_j))\geq 1.$ Now notice that $f_j$ is nondecreasing in each variable, and is strictly increasing in the $k$-th variable if the probability of an individual of type $j$ to have at least one child of type $k$ is nonzero. Thus, if there exists such a $k$ which also satisfies $t_k>t_j,$ then $p_k+t_jx_k<1$, and by (strict) monotonicity $f_j(\mathbf{1})>1,$ our wanted contradiction. If we cannot find a $k$ satisfying this, it means that for all types $j'$ of possible children of an individual of type $j$, we have $t_{j'}=t_j.$ Now we can apply the previous reasoning to the functions $g_{j'}$, and repeat inductively, stopping when we find a $j$ and $k$ such that $t_j$ is minimal, $t_k$ is not minimal, and the probability that an individual of type $j$ has a child of type $k$ is nonzero. To conclude, we need to check that this procedure does end - if it doesn't, then $t_j$ will be the same value for all $j$ in one same irreducible component of $M^{\cross},$ and none of these types can give birth (under $\zeta^{(j)}$) to children outside of that component, a contradiction.
\end{proof}

\noindent\textbf{The height of subcritical trees has exponential moments.} Let $(\xi^{(i)},i\in[{\kappa}])$ be a set of subcritical offspring distributions (i.e.\ the corresponding Galton-Watson tree is subcritical, as defined above), without any assumption of irreducibility. We consider Galton-Watson trees $(S^{(i)},i\in[{\kappa}])$ with offspring distributions $(\xi^{(i)},i\in[{\kappa}]),$ such that $S^{(i)}$ has root of type $i$, for all $i\in[{\kappa}].$

\begin{prop} 
\label{prop:heightsubcritical}
There exists $\lambda>0$ such that, for all $i\in[{\kappa}]$, 
\[\E\big[e^{\lambda\mathrm{ht}(S^{(i)})}\big]<\infty.\]
\end{prop}

\begin{proof}
Let, for $n\in\Z_+$ and $i\in[{\kappa}]$, $x^i_n=\pr (\mathrm{ht}(S^{(i)})\leq n)$ and then $\mathbf{x}_n=(x^1_n,\ldots,x^{\kappa}_n).$ By the Galton-Watson property, we have for all $n$ 
\[\mathbf{x}_{n+1}=\mathbf{f}(\mathbf{x}_n)
\]
where $\mathbf{f}=(f_1,\ldots,f_{\kappa})$ is the moment generating function defined by $f_i(\mathbf{x})=\sum_{\mathbf{z}\in(\Z_+)^{\kappa}} \xi^{(i)}(\mathbf{z})\prod_{j=1}^{\kappa} x_j^{z_j}.$ Noticing that $\mathbf{x}_n$ tends to $\mathbf{1}:=(1,\ldots,1)$ as $n$ tends to infinity, we let $\mathbf{y}_n=\mathbf{1}-\mathbf{x}_n$, and we then have the first order expansion as $n\rightarrow \infty$
\begin{align*}
\mathbf{y}_{n+1}&=N\mathbf{y}_n + o(\mathbf{y}_n),
\end{align*}
where $N$ is the mean matrix of $(\xi^{(i)},i\in[{\kappa}])$. By subcriticality, the largest eigenvalue $\rho$ of $N$ satisfies $\rho<1$, and taking $a<1$ with $a>\rho$, the series $\sum_n a^{-n}\mathbf{y}_n$ converges. Letting $\lambda=-\log a>0$, we then have $\E[e^{\lambda\mathrm{ht}(S^{(i)})}]<\infty$ for all $i\in[{\kappa}].$
\end{proof}

\noindent\textbf{There are not too many vertices.} Now knowing that the subtrees without any vertices of type $1$ have relatively small heights, we want to check that they are not too numerous. The following rough estimate will be enough for our purposes. We use the notation $\#_j$ for the function which maps a $\kappa$-type tree to its number of vertices with type $j$, for $j\in[\kappa].$
\begin{lem}\label{lem:crudebound} Let $i,j$ be two types, with $j\neq 1$. We have
\[\pr\big( \#_j T^{(i)}>n^3\mid \#_1 T^{(i)}=n\big)\underset{n\to\infty}\longrightarrow 0.\]
\end{lem}
A rough but useful consequence is that the probability that the number of subtrees removed from $T^{(i)}_n$ to get $\mathring{T}_n^{(i)}$ is larger than $n^3$, given $\#_1 T^{(i)}=n$, converges to 0 as $n \rightarrow \infty$.

\begin{proof} We use the same idea as the proof of \cite[Lemma 6.7]{Steph15} where a much stronger result, but assuming  exponential moments, is obtained. We start with the case where $i=1$. Consider a sequence of independent trees $(T_k,k\in\N),$ all distributed as $T^{(1)},$  and list their vertices of type $1$ in an arbitrary order, but such that all the vertices of type $1$ of $T_k$ are listed before those of $T_{k+1}$ for all $k$. Let then $A_k$ be number of vertices of type $j$ whose highest ancestor of type $1$ is the $k-$th element of this list. The $(A_k)$ are then all i.i.d., and by \cite[Proposition 4]{M08}, have expectation $\frac{a_j}{a_1}.$ The event $\{\#_j T_1>n^3\}\cap\{ \#_1 T_1=n\}$ is moreover included in the event $\{A_1+\ldots+A_n> n^3\}.$ However by Markov's inequality,
\[\pr (A_1+\ldots+A_n> n^3)\leq \frac{1}{n^2}\frac{a_j}{a_1},\]
and recalling that, from Proposition \ref{prop:betterbyword}, $\pr(\#_1 T^{(1)}=n)\sim \frac{c}{n^{3/2}}$ for some $c>0,$ we get
\[\pr(\#_j T_1>n^3\mid \#_1 T_1=n)\leq \frac{a_j}{ca_1 \sqrt{n}},\]
ending the proof of this case.

If $i\neq 1$, then we can write
\[\pr( \#_j T^{(i)}>n^3\cap \#_1 T^{(i)}=n)\leq \pr (X+A_1+\ldots+A_n> n^3),\]
where $X$ has the distribution of the number of vertices with type $j$ of $T^{(i)}$ standing between its root and first generation of type $1$, and is independent from the $(A_k,k\in\N)$. By \cite[Proposition 4]{M08}, $X$ also has a finite expectation, hence we can end the proof the same way.
\end{proof}

\noindent\textbf{Proof of Proposition \ref{prop:zero}.} Combining the previous results with the following lemma will give the proof of the proposition.

\begin{lem}\label{lem:expo} For $n\in\N$, let $(X_n(k),k\in\N)$ and $N_n$ be $\N$-valued random variables on a certain probability space and $\mathcal{F}_n$ a sub-$\sigma$-algebra such that $N_n$ is $\mathcal{F}_n$-measurable, and we have the two following conditions:
\begin{itemize}
\item[\emph{(i)}] There exists $b\in\N$ such that $\pr(N_n\leq n^b)\to 1$ as $n\to\infty$.
\item[\emph{(ii)}] Conditionally on $\mathcal{F}_n$, the $(X_n(k),k\leq N_n)$ are independent and there exist $A>0$ and $\lambda>0$ such that $\E[e^{\lambda X_n(k)}|\mathcal{F}_n]\leq A$ a.s.\ for all $k\leq N_n$.
\end{itemize}
Then there exists $C>0$ such that
\[\pr \big(\max\{X_n(1),\ldots,X_n(N_n)\}\leq C\log n \big)\to 1.\]
\end{lem}

\begin{proof}
Take any $C>0$ and $n$ large enough such that $A <n^{\lambda C}$. Conditioning on $\mathcal{F}_n,$ we have, using (ii) and the Markov inequality:
\begin{align*}
\pr \big(\max\{X_n(1),\ldots,X_n(N_n)\}\leq C\log n\mid \mathcal{F}_n\big)&= \prod_{k=1}^{N_n}\big(1-  \pr[X_n(k)> C\log n\mid \mathcal{F}_n]\big) \\
   &\geq \prod_{k=1}^{N_n}\big(1-e^{-\lambda C \log n} A\big)\\
   &\geq \left(1-\frac{A}{n^{\lambda C}}\right)^{N_n}\\
   &\geq \mathbf{1}_{\{N_n\leq n^b\}} \left(1-\frac{A}{n^{\lambda C}}\right)^{n^b}.
\end{align*}
Removing the conditioning, we obtain
\[\pr \big(\max\{X_n(1),\ldots,X_n(N_n)\}\leq C\log n\big)\geq \pr \big(N_n\leq n^b \big)\left(1-\frac{A}{n^{\lambda C}}\right)^{n^b},\]
and this tends to $1$ if we choose $C>b/\lambda.$
\end{proof}

\noindent\textit{Proof of Proposition \ref{prop:zero}.} Fix $i\in[{\kappa}]$. For $n\in\N$ and $j\in[{\kappa}]\setminus\{1\}$, let $Z_n^{(i)}(j)$ be the number of vertices of $T_n^{(i)}$ of type $j$ which have no vertices of type $1$ as descendants, but such that their parent does have at least one descendant with type $1$. Let then $\mathcal{F}_n^{(i)}=\sigma\big(Z_n^{(i)}(2),\ldots,Z_n^{(i)}({\kappa})\big)$ and $N_n^{(i)}=Z_n^{(i)}(2)+\ldots+Z_n^{(i)}({\kappa}).$ Conditionally on $\mathcal{F}_n^{(i)},$ $T_n^{(i)}$ can then be obtained from $\mathring{T}_n^{(i)}$ by grafting at its leaves $N_n^{(i)}$ independent trees, the heights of which all have a uniformly bounded exponential moment, by Proposition \ref{prop:subcritical} and Proposition \ref{prop:heightsubcritical}. Simply bounding $N_n^{(i)}$ by $\# T_n^{(i)}$, Lemma \ref{lem:crudebound} and Lemma \ref{lem:expo} conclude the proof.  \qed

\subsection{End of the proof of Theorem \ref{thGW}}

By construction, the sequence of reduced trees $(\mathring{T}_n^{(i)})$, introduced in Proposition \ref{prop:zero}, is MB, and its splitting distributions $(\mathring{q}_n^{(i)})$ are the push-forwards of $(q_n^{(i)})$ by the operation of removing the parts with zero size. The results of Proposition \ref{prop:cvtypesGW} and Proposition \ref{prop:cvbrownien} then pass on to $\mathring{q}_n^{(i)}$, and thus we have, by Theorem \ref{theo:melange} and the calculation in (\ref{eq:melangebrownien}):
\[\left(\frac{\mathring T_n^{(i)}}{\sqrt{n}},\mu_n^{(i)}\right) \ \underset{n \rightarrow \infty}{\overset{(d)}\longrightarrow} \ \frac{2}{\sigma\sqrt{a_1}}(\T_{\mathrm{Br}},\mu_{\mathrm{Br}}).\]
By Proposition \ref{prop:zero}, we will get the same limit if we replace $\mathring T_n^{(i)}$ by $T_n^{(i)}$, ending the proof. \qed

\section{Scaling limits of multi-type MB trees: preliminary work}
\label{sec:proofs1}

The aim of this section is to set up some definitions and results that prepare for the proofs of Theorems \ref{theo:critique} and \ref{theo:melange} in the next section. Several lines of our proofs will not differ too much from those of the proof of Theorem 5 of \cite{HM12} in the monotype setting, so we will mostly only give the reader details when the multi-type structure comes into play. Throughout the section $(q_n^{(i)})$ is a sequence of splitting distributions that satisfies (\ref{hyp:principal}) and either the hypotheses of Theorem \ref{theo:critique} or Theorem \ref{theo:melange} (this will be specified each time), and $(T_n^{(i)})$ is an associated sequence of MB trees. Below we start by showing in Section \ref{sec:assumptions} that we can restrict ourselves to conservative cases where particles of size $1$ die without reproducing. We then introduce in Section \ref{sec:prem} the partition-valued multi-type fragmentation processes associated to our trees and their continuous counterpart, and set up preliminary results. In Section \ref{sec:heightmoments} we show some necessary fine bounds on the moments of the heights of the trees, that will be needed to obtain a tightness criterion for the rescaled trees. This ends up being more difficult to prove than it the monotype setting, in particular in the mixing regime when some of the limiting measures $(\nu^{(i)},i\in[\kappa])$ are null. These bounds will also allow us to improve some results of \cite{HS18} on the asymptotic description of bivariate Markov chains, that will be needed in Section \ref{sec:proofs2} to evaluate the scaling limits of typical paths of our multi-type MB trees (the results of \cite{HS18} are not strong enough in the mixing case if some of the limiting measures are null).

\subsection{Simplifying the assumptions}
\label{sec:assumptions}

We use two simple couplings to show that we can limit ourselves to offspring distribution sequences $(q_n^{(i)})$ such that:
\begin{equation}
\label{eq:simple}
q_1^{(i)}(\emptyset)=1 \text{ for all } i\in[\kappa] \quad \text{and} \quad q_n^{(i)} \text{ is conservative for all }i\in[\kappa] \text{ and }n\geq 2.
\end{equation}

\noindent\textbf{Particles of size $1$ die without reproducing.} Starting from a sequence $(q_n^{(i)})$ satisfying (\ref{hyp:principal}), we let for $n\in\N$ and $i\in[\kappa]$, ${\accentset{\bullet}{q}}_n^{(i)}$ be the splitting distribution such that:
\begin{equation*}
\accentset{\bullet}{q}_1^{(i)}(\emptyset)=1 \quad \text{and} \quad {\accentset{\bullet}{q}}_n^{(i)}(\bar{\lambda})={q}_n^{(i)}(\bar{\lambda}) \text{ for }\bar{\lambda}\in \overline\p_n \text{ if } n\geq 2 .
\end{equation*}
Let $(\accentset{\bullet}{T}_n^{(i)})$ be a MB tree sequence with this splitting distribution sequence. Then there is a natural coupling of $\accentset{\bullet}{T}_n^{(i)}$ and $T_n^{(i)}$ such that $T_n^{(i)}$ is obtained by grafting independent copies of the $(T_1^{(j)},j\in[\kappa])$ onto the leaves of size $1$ of $\accentset{\bullet}{T}_n^{(i)},$ and there are at most $n$ such leaves.
Notice however that, for any $j\in[\kappa]$, $T_1^{(j)}$ is essentially an a.s.\ killed Markov chain on $[\kappa]$. It is then well-known that its death time (i.e.\ the height of $T_1^{(j)}$) has exponential moments. Thus the maximum of $n$ independent such variables is at most of order $\log n$, i.e.\ there exists $C>0$ such that 
\[\pr\Big(d_{\mathrm{GHP}}\big(T_n^{(i)},{\accentset{\bullet}{T}}_n^{(i)}\big)\leq C\log n\Big)\underset{n\to\infty}\longrightarrow 1,\]
where we recall that $d_{\mathrm{GHP}}$ denotes the Gromov-Hausdorff-Prokhorov distance. 
As a consequence, proving Theorems \ref{theo:critique} and \ref{theo:melange} for $(\accentset{\bullet}{q}_n^{(i)})$ will also prove them for $({q}_n^{(i)})$.

\bigskip

\noindent\textbf{Conservation of mass.} Now we give a simple coupling between non-conservative MB trees and conservative ones which shows that the non-conservative cases of Theorem \ref{theo:critique} and Theorem \ref{theo:melange} are consequences of the conservative cases.
Take a sequence of offspring distributions $(q_n^{(i)})$ satisfying (\ref{hyp:principal}) and $q_1^{(i)}(\emptyset)=1$ for all  $i\in[\kappa]$, that are not necessarily conservative. For $n \geq 2$ and $\bar{\lambda}\in\overline \p_n$ such that $\lambda_0=n-\sum_{m=1}^{p(\bar{\lambda})}\lambda_m \neq 0$, let 
\[\mathring{\bar{\lambda}}:=\left((\lambda_1,i_1),\ldots,(\lambda_{p(\bar{\lambda})},i_{p(\bar{\lambda})}),(1,1),\ldots,(1,1)\right),\]
 where $(1,1)$ has been repeated $\lambda_0$ times, while if $\lambda_0=0$ let $\mathring{\bar{\lambda}}={\bar{\lambda}}$. Let then $\mathring{q}_n^{(i)}$ be the image measure of $q_n^{(i)}$ by the mapping $\bar{\lambda}\to \mathring{\bar{\lambda}}$, which is conservative. We have by construction $d(n^{-1} \cdot \bar{\lambda}, n^{-1} \cdot \mathring{\bar{\lambda}})\leq n^{-1}$, where $d$ is the metric (\ref{def:metric}). Thus if $(q_n^{(i)})$ satisfies the limiting assumption of Theorem \ref{theo:critique} or Theorem \ref{theo:melange}, and if it is conservative for $n$ large or $\gamma<1$, then $(\mathring{q}_n^{(i)})$ also does.
 
Moreover, calling $(\mathring{T}_n^{(i)})$ a MB tree sequence with offspring distribution sequence $(\mathring{q}_n^{(i)})$, there is a natural coupling between $\mathring{T}_n^{(i)}$ and $T_n^{(i)}$ such that $d_{\mathrm{GHP}}(\mathring{T}_n^{(i)},T_n^{(i)})\leq 1$ a.s. Thus proving Theorems \ref{theo:critique} and \ref{theo:melange} for $(\mathring{q}_n^{(i)})$ will also prove it for $({q}_n^{(i)})$. 

\bigskip 

Consequently, \textbf{in the following we can restrict ourselves to sequences $(q_n^{(i)})$ satisfying (\ref{eq:simple})}.
 
\bigskip

\subsection{Fragmentation processes and trees}
\label{sec:prem}

In this section we start by introducing the discrete partition-valued fragmentation processes associated to our trees $(T_n^{(i)})$, and then turn to the main definitions of the theory of continuous-time multi-type fragmentation processes and trees. We also set up some results that will be needed later, in particular concerning the so-called tagged fragment process. Throughout, the starting type $i\in[\kappa]$ will be fixed.

\subsubsection{Details about $\kappa$-type partitions}
\label{sec:partitions}


We extend the concept of $\kappa$-type partitions to subsets of $\mathbb N$, and define their paintbox measures. 

\medskip

\textbf{Typed partitions of subsets of $\mathbb N$.} Let $B$ be a subset of $\N$. We call $\overline{\p}_{B}$ the set of $\kappa$-type partitions of $B,$ which are objects of the form $\bar{\pi}=(\pi,\mathbf{i})=(\pi_n,i_n)_{n\in\N}$, where $\pi$ is a classical partition of $B$, its blocks $\pi_1,\pi_2,\ldots$ being listed in increasing order of their least element, and $i_n\in\{0,\ldots,\kappa\}$ is the type of the $n$-th block for all $n\in\N$, with $i_n=0$ if and only if $\pi_n$ is empty. For $n\in B$, we also use the notation $\bar{\pi}_{(n)}=(\pi_{(n)},i_{(n)})$ for the block of $\bar{\pi}$ containing $n$ and its type. Finally, if it exists, we write $|B|$ for the asymptotic frequency of $B$: the limit of $n^{-1}\#(B\cap[n])$ as $n$ tends to infinity.

\bigskip

\textbf{Paintbox measures.} We will need to use ``paintbox'' random partitions: for $\bar{\mathbf{s}}=(\mathbf{s},\mathbf{i})\in \overline{\s}^{\downarrow}$ such that $\sum s_n=1$, let $(U_n,n\in\N)$ be an i.i.d sequence of uniform random variables on $[0,1]$, and define a random partition $\Pi_{\bar{\mathbf{s}}} \in \overline{\p}_{\mathbb N}$ by declaring two integers $n$ and $m$ to be in the same block if there exists $k\in\N$ such that $U_n$ and $U_m$ are both in the interval $\big[\sum_{\ell=1}^{k-1}s_\ell, \sum_{\ell=1}^k s_{\ell}\big)$, and the type of this block is then $i_k$. We let $\kappa_{\bar{\mathbf{s}}}$ be the distribution of $\Pi_{\bar{\mathbf{s}}}$ and call it the paintbox measure associated to $\bar{\mathbf{s}}.$

\subsubsection{Finite-dimensional marginals of $T_n^{(i)}$ and discrete fragmentation processes}

Under our extra assumptions (\ref{eq:simple}), $T_n^{(i)}$ has exactly $n$ leaves. We label them $L_1,\ldots,L_n$ with a uniform random order.

\medskip

\textbf{Associated discrete fragmentation processes.} This labelling of leaves lets us define a $\overline{\p}_{[n]}$-valued process $\overline{\Pi}_n=\left(\overline{\Pi}_n(\ell),\ell\geq 0\right)$ by saying that two integers $p\leq n$ and $q\leq n$ are in the same block of $\overline{\Pi}_n(\ell)$ if and only if the highest common ancestor of $L_p$ and $L_q$ in $T_n^{(i)}$ has height at least $\ell$, and the type of this block is then the type of their common ancestor with height $\ell.$ In particular, $p$ is in a singleton if and only if $\ell$ is at least equal to the height of $L_p,$ and we  still need to give it a type in this case: we choose the type of $L_p$ for all $\ell$ larger or equal to the height of $L_p$.  This defines an exchangeable Markov chain on $\overline{\p}_{[n]}$ with a natural branching property. We let $p_n^{(i)}$ be the distribution of $\overline{\Pi}_n(1).$ 

For $m\in[n]$, we let $D_{\{m\}}^{(n)}$ be the height of $L_m$, and for a subset $B\subset [n]$ with cardinality at least $2$, we let $D_B^{(n)}=\inf\left\{\ell\in \N: B\cap\Pi_n(\ell)\neq B\right\}$ the time at which $\overline{\Pi}_n$ splits $B$. In the case where $B=[k]$ for some $2\leq k\leq n$, we let $D_k^{(n)}:=D_{[k]}^{(n)}$, and  $D_1^{(n)}:=D_{\{1\}}^{(n)}$.

\bigskip

\textbf{Finite-dimensional marginals.} For $B\subset [n]$, we let $$\mathcal{R}(T_n^{(i)},B) \text{ be the  \emph{$B$-marginal} of } T_n^{(i)},$$
which is its smallest subtree to contain the root and each $L_m$ for $m\in B$. 

\bigskip

\textbf{The tagged fragment chain.} We call thus the Markov chain $(X_n,J_n)$ on $\Z_+\times [\kappa]$ of the typed block containing the integer 1 which is such that, for $\ell\leq D_{1}^{(n)},$ 
\begin{equation}
\label{def:discretetag}
\left(X_n,J_n\right)(\ell):=\left(\#(\Pi_n(\ell))_{1},(\mathbf i_n(\ell))_1\right)
\end{equation} 
(where $\overline \Pi_n(\ell)=( \Pi_n(\ell),\mathbf i_n(\ell))$) and which is stationary starting from $D_{1}^{(n)}$.  Its transition probabilities, which we call $(p_{m,j}(\ell,k))$ (probability of going from $(m,j)$ to $(\ell,k)$), do not depend on $n$ or $i$ and are given for $m\geq 2$, $1 \leq \ell \leq m$  and  $j,k\in [\kappa]$ by
\begin{equation}
\label{transitions}
p_{m,j}(\ell,k)=\sum_{\bar{\lambda}\in\overline{\mathcal{P}}_m}q_m^{(j)}(\bar{\lambda})\frac{\ell}{m}m_{(\ell,k)}(\bar{\lambda}),
\end{equation}
where $m_{(\ell,k)}(\bar{\lambda})$ is the number of occurrences of $(\ell,k)$ in the sequence $\bar{\lambda}$, and $p_{1,j}(1,j)=1$ for all $j \in [\kappa]$. (In all other cases these probabilities are null.)

The following lemma gives us some information about the distribution of this chain up to $D_k^{(n)}, 2 \leq k \leq n$ and what happens at that time. This result echoes a similar result in the monotype setting, \cite[Lemma 27]{HM12}. The same proof will work, with a light modification to take the types into account. We will not reproduce it here, and simply point out that it follows easily from noticing that
\[\pr\big(D_k^{(n)}\geq r \mid X_n(r'),0\leq r'\leq r-1\big)=\frac{(X_n(r-1)-1)_{k-1}}{(n-1)_{k-1}},\]
which is a variant of Lemma 13 in \cite{HM12}. We recall the factorial notation $(x)_r=x(x-1)\ldots(x-r+1)$ for $x\in\R$, $r\in\N$.

\begin{lem}\label{lem:separationdiscret} Let $2 \leq k \leq n$ and $\bar{\pi}'\in\overline{\mathcal{P}}_{[k]}$ with $b\geq 2$ blocks. For measurable nonnegative functions $f,g,h$ on the appropriate spaces, we have
\begin{align*}
&\E\Big[f\big(D_k^{(n)}\big)g\big((X_n,J_n)(\cdot\wedge D_k^{(n)}-1)\big)h\big((\#\Pi_n(D_k^{(n)}))_m,1\leq m\leq b\big)\mathbf{1}_{\{[k]\cap\overline{\Pi}_n(D^{(n)}_k))=\bar{\pi}'\}}\Big] \\
&=\sum_{r \in \mathbb N}f(r)\E\left[\frac{(X_n(r-1)-1)_{k-1}}{(n-1)_{k-1}}
         g\big((X_n,J_n)(\cdot\wedge(r-1))\big)p_{X_n(r-1)}^{(J_n(r-1))}\left(h(\#\pi_m,1\leq m\leq b)\mathbf{1}_{\{[k]\cap\bar{\pi}=\bar{\pi}'\}}\right)\right].
\end{align*}
\end{lem}

\subsubsection{Construction and properties of multi-type fragmentation processes and trees}
\label{sec:construcmultMB}

We formally build the multi-type fragmentation processes and trees mentioned in Section \ref{sec:multi-type frag}. This part is quite brief, we refer to \cite{Steph17} for more details. We also give definitions and results that will be needed later, and that are similar to the ones given above in the discrete setting. In the following, $\gamma>0$, $\bnu=(\bar{\nu}^{(j)},j \in [\kappa])$ is a vector of dislocation measures and $(\T^{(i)}_{\gamma,\bnu}, \mu^{(i)}_{\gamma,\bnu})$ a multi-type fragmentation trees with characteristics $(\gamma,\bnu)$. The unbold notation $(\T_{\gamma,\nu}, \mu_{\gamma,\nu})$ denotes the monotype case where there is a unique dislocation measure $\nu$. 

\bigskip

\textbf{Multi-type fragmentation processes and trees.} The tree $\T^{(i)}_{\gamma,\bnu}$ is typically built out of a so-called $(-\gamma)$-self-similar fragmentation process $\overline{\Pi}=(\overline{\Pi}(t),t\geq 0),$  with dislocation measures $\bnu$ and initial type $i$,
which is a $\overline{\p}_{\N}$-valued process of which we will now give the construction. It is a Lamperti transform of a homogeneous fragmentation process $\overline{\Pi}^{(\mathrm{hom})}$ which we construct first. 

For $\bar{\mathbf{s}}\in\overline{\s}^{\downarrow}$ with total sum 1, recall that $\kappa_{\bar{\mathbf{s}}}$ is the paintbox measure on $\overline{\p}_{\N}$ associated to $\bar{\mathbf{s}},$ and for $j\in[\kappa]$, let $$\kappa_{\bar{\nu}^{(j)}}:=\int_{\overline{\s}^{\downarrow}}\kappa_{\bar{\mathbf{s}}}\bar{\nu}^{(j)}(\mathrm d\bar{\mathbf{s}}).$$ 
The process $\overline{\Pi}^{(\mathrm{hom})}$ describes the evolution of typed blocks such that a block $B$ of type $j$ splits into typed blocks $((B\cap \Delta_m, \delta_m),m\geq 1),$ where $\overline \Delta=(\Delta_n,\delta_m)_m \in \overline{\p}_{\N},$ at rate $\kappa_{\bar{\nu}^{(j)}}(\mathrm d \overline \Delta)$. Formally it can be constructed from Poisson point processes as follows.
For $n\in\N$ and $j\in [\kappa]$, we let $(\overline{\Delta}^{(n,j)}(t),t\geq0)=\big((\Delta^{(n,j)}(t),\delta^{(n,j)}(t)),t\geq0 \big)$ be a Poisson point process with intensity $\kappa_{\bar{\nu}^{(j)}}$, which we all take independent. 
The process $\overline{\Pi}^{(\mathrm{hom})}$ is then built thus:
\begin{itemize}
\item Start with $\overline{\Pi}^{(\mathrm{hom})}(0)=(\N,i),$ the partition with only one block, which has type $i$.
\item For $t\geq0$ such that there is an atom $\overline{\Delta}^{(n,j)}(t)$, if $i_n(t-)=j$: replace the typed block $\overline{\Pi}^{(\mathrm{hom})}_n(t-)$ by the typed blocks $\big(\Pi^{(\mathrm{hom})}_n(t-)\cap \Delta^{(n,j)}_m(t),\delta^{(n,j)}_m(t)\big)$, $m\in\N$ and reorder them appropriately; if $i_n(t-)\neq j$, $\overline{\Pi}^{(\mathrm{hom})}(t)=\overline{\Pi}^{(\mathrm{hom})}(t-)$.
\end{itemize}
Note that this is well defined despite the arrival times of the involved Poisson point processes possibly having accumulation points, as shown in \cite{Ber08}. The self-similar fragmentation process $\overline{\Pi}$ is then obtained from $\overline{\Pi}^{(\mathrm{hom})}$ by using the Lamperti time-change: for $n\in\N$ and $t\geq0,$ let
\[\tau_n(t) = \inf \left\{u, \int_0^u \big|\Pi^{(\mathrm{hom})}_{(n)}(r)\big|^{\gamma} \mathrm{d}r >t\right\}
\] 
and then define the block of $\overline{\Pi}$ containing $n$ as
\[\overline{\Pi}_{(n)}(t):=\overline{\Pi}^{(\mathrm{hom})}_{(n)}(\tau_n(t))\]
with the convention $\overline{\Pi}^{(\mathrm{hom})}_{(n)}(\tau_n(t))=\{\{n\},0\}$ when $\tau_n(t)=+\infty$.
The process $\overline{\Pi}$ is \emph{self-similar} in the following sense: for $t\geq 0$, conditionally on $(\overline{\Pi}(s),s\leq t)$, with $\overline{\Pi}(t)=\bar{\pi}=(\pi,\mathbf{i})$, the processes $(\overline{\Pi}(t+\cdot)\cap \pi_m, m\in\N)$ are independent, and each one has the distribution of $\overline{\Pi}^{(i_m)}(|\pi_m|^{\gamma}\cdot)\cap \pi_m$, where $\overline{\Pi}^{(i_m)}$ is a copy of $\overline{\Pi}$ starting with type $i_m.$

We define analogous notations to the discrete case: for $n\in\N$, $D_{\{n\}}$ is the time at which $n$ is sent into a singleton, and for $B\subset \N$ with $\#B\geq 2$, $D_B$ is the time at which $B$ is split by $\overline{\Pi}$. Also, $D_{k}:=D_{[k]}$ for $k\geq 2$ and $D_1:=D_{\{1\}}$. We have that $D_k<D_1$ a.s.\ for $k\geq 2$.

The fragmentation tree $\T^{(i)}_{\gamma,\bnu}$ is then, as explained in \cite{Steph17}, a compact real tree which is the family tree of $\overline{\Pi}$ in the following sense: it has distinguished leaves $(L_n,n\in\N)$ such that the height of $L_n$ is $D_{\{n\}}$ for all $n\in\N$ and, for $m\neq n,$ the paths from the root to $L_n$ and $L_m$ split at height $D_{\{n,m\}}.$ The measure $\mu^{(i)}_{\gamma,\bnu}$ is then characterized by the following property: for all integers $n$ and all positive real numbers $t$ strictly smaller than the height of $L_n$,  if $L_n(t)$ denotes the unique ancestor of $L_n$ at height $t$ and $\mathcal T_{L_n(t)}$ the subtree of descendants of $L_n(t)$, one has $$\mu_{\gamma,\bnu}(\mathcal T_{L_n(t)})=|\Pi_{(n)}(t-)|>0.$$

\bigskip

\textbf{Finite dimensional marginals.} We let for $B \subset \mathbb N$
$$\mathcal{R}(\T^{(i)}_{\gamma,\bnu},B) \text{ be the  \emph{$B$-marginal} of } \T^{(i)}_{\gamma,\bnu},$$ which is its smallest subtree to contain the root and each $L_n$ for $n\in B$.

\bigskip

\textbf{The tagged fragment process as a Lamperti transform of a Markov additive process.} Still as in the discrete cases, we are interested in the evolution of a typical block: the typed block containing the integer 1. We let for $t \geq 0$ 
$$(X(t),J(t)) \text{ be the size and type of the block containing } 1 \text{ at time }t,$$ with in particular that $(X(t),J(t))=(0,0)$ for $t\geq D_1$. It is known from \cite{Ber08} and \cite{Steph17} that $(X,J)$ is the $\gamma$-Lamperti transform of a \emph{Markov additive process}, i.e.\ $(X,J)=(\exp(-\xi_\rho),K_\rho)$ 
where 
\[\rho(t) := \inf \left\{u, \int_0^u (\exp(-\gamma \xi_r)) \mathrm{d}r >t\right\},
\] 
and $\left((\xi_t,K_t),t \geq 0 \right)$ is a  Markov process on $\mathbb R_+ \times \{1,\ldots,\kappa\}\cup\{(+\infty,0)\}$ such that, if $\mathbb P_{(x,i)}$ denotes its distribution when starting at a point $(x,i)$, for all $t \in \mathbb R_+$ and all $(x,i) \in \mathbb R_+ \times \{1,\ldots,\kappa\}$,
\[
\left(\left(\xi_{t+s}-\xi_t,K_{t+s}),s \geq 0\right) \ | \ (\xi_u,K_u),u\leq t , \xi_t<\infty \right) \text{ under } \mathbb P_{(x,i)} \text{ has distribution } \mathbb P_{(0,K_t)},
\]
and $(+\infty,0)$ is an absorbing state. Note that the process $\xi$ is nonincreasing, and, when $\kappa=1$, it is simply a subordinator. For general $\kappa$, it is known (we refer to Asmussen \cite[Chapter XI]{asmussen} for background on  Markov additive processes) that the distribution of the process $(\xi,K)$ is characterized by three families of parameters which here are:
\begin{itemize}
\item For all $i\in[\kappa]$
\[ \psi_i(q)=\int_{\overline{\s}^{\downarrow}} \bigg(\sum_{n=1}^{\infty}\big(s_n-s_n^{1+q}\big)\mathbf{1}_{\{i_n=i\}}\bigg)\bar{\nu}^{(i)}(\mathrm d \bar{\mathbf{s}}),\qquad q\geq0,\]
which is the Laplace exponent of the subordinator which governs the local dynamics of the component $\xi$ when the type is $i$. 
\item For all $i\neq j \in[\kappa],$ the transition rate $\lambda_{i,j}$ of types from $i$ to $j$ and the distribution $B_{i,j}$ of the jump (that may be null) of the component $\xi$ when the type changes from $i$ to $j$, which are defined by 
\[ \lambda_{i,j}\int_{0}^{\infty}e^{-qx}B_{i,j}(\mathrm dx):=\int_{\overline{\s}^{\downarrow}}\bigg(\sum_{n=1}^{\infty}s_n^{1+q}\mathbf{1}_{\{i_n=j\}}\bigg)\bar{\nu}^{(i)}(\mathrm d \bar{\mathbf{s}}),\qquad q\geq0.\]
\end{itemize}
Since these parameters are all constructed from the vector $\bnu$, the following definition is justified.
\begin{defn}
We say that the distribution of the tagged fragment process $(X,J)$ is that of a \emph{$\gamma$-Lamperti transform of a Markov additive process with characteristics $\bnu$} (and characteristic $\nu$ if there is a unique type and the dislocation measure is $\nu$).  
\end{defn}

We now give a version of Lemma \ref{lem:separationdiscret} for this setting. For $k\geq 2$ we let $A_k=\{\bar{\pi}\in\overline{\mathcal P}_{\mathbb N},[k]\cap \pi\neq [k]\}$ be the set of partitions which split the first $k$ integers. 
\begin{lem}\label{lem:fragsplit} Let $k\geq 2$. For nonnegative measurable functions $f,g,h$ on the appropriate spaces, we have
\begin{align*}
&\E\left[f(D_k)g\big(|\overline{\Pi}_{(1)}(\cdot\wedge D_k-)|\big)h\big(\overline{\Delta}^{(1,i_1(D_k-))}(D_k)\big)\right] \\
&=\E\left[\int_0^{\infty}\mathrm du\ f(u)|\Pi_{(1)}(u)|^{k-1-\gamma}\mathbf 1_{\{|\Pi_{(1)}(u)|>0\}}g(|\overline{\Pi}_{(1)}(\cdot\wedge u)|)\kappa_{\nu^{(i_1(u))}}(h\mathbf{1}_{A_k})\right].
\end{align*}
\end{lem}
\begin{proof} This a multi-type version of the monotype result \cite[Proposition 18]{HM12}. We only detail the main differences of the proofs. As in the monotype setting, moving from the $\gamma=0$ case to general $\gamma$ is easily done by using the Lamperti time change. We can therefore focus on the $\gamma=0$ case. For all $i\in[\kappa],$ let $\kappa^b_{\bar{\nu}^{(i)}}(\mathrm d\bar{\pi}):=\kappa_{\bar{\nu}^{(i)}}(\mathrm d\bar{\pi})\mathbf{1}_{\{\bar{\pi}\in A_k\}}$ and $\kappa^a_{\bar{\nu}^{(i)}}:=\kappa_{\bar{\nu}^{(i)}}-\kappa^b_{\bar{\nu}^{(i)}},$ the former having finite total mass. By standard properties of Poisson point processes, we can realise the pair $\big(\overline{\Pi}(\cdot\wedge D_k-),\overline{\Delta}^{(1,i_1(D_k-))}(D_k)\big)$ the following way:
\begin{itemize}
\item Let $\overline{\Pi}^a=(\Pi^a,\mathbf{i}^a)$ be a homogeneous fragmentation process constructed the same way as $\overline{\Pi}$, except that we use Poisson point processes with intensity $(\kappa^a_{\bar{\nu}^{(i)}})$ instead of $(\kappa_{\bar{\nu}^{(i)}}).$
\item Conditionally on $\overline{\Pi}^a,$ let $T$ be a positive r.v. such that, for $t\geq 0$, $\pr(T\geq t)= e^{-\int_0^t \kappa_{\bar{\nu}^{(i_1^a(u))}}(A_k) \mathrm d u}$. Equivalently, $T$ is the first arrival time of a Poisson process with inhomogeneous rate $\kappa_{\bar{\nu}^{(i_1^a(t))}}(A_k)$ at time $t.$
\item Conditionally on $\overline{\Pi}^a$ and $T$, let $\bar{\pi}$ be a random partition with distribution $\kappa_{\bar{\nu}^{(i_1(T))}}(\cdot\mid A_k):=\kappa_{\bar{\nu}^{(i_1(T))}}(\cdot\,\cap A_k)/\kappa_{\bar{\nu}^{(i_1(T))}}(A_k)$.
\end{itemize}
The triplet $\big(T,\overline{\Pi}^a_{(1)}(\cdot\wedge T),\bar{\pi}\big)$ then has the same distribution as $\big(D_k,\overline{\Pi}_{(1)}(\cdot\wedge D_k-),\overline{\Delta}^{(1,i_1(D_k-))}(D_k)\big)$. Thus we have, using Fubini's theorem and the fact that the conditional density of $T$ is \linebreak $\kappa_{\bar{\nu}^{(i_1(u))}}(A_k) e^{-\int_0^u \kappa_{\bar{\nu}^{(i_1^a(s))}}(A_k)\mathrm d s}$:
\begin{align*}
\E\Big[f(D_k)g\big(|\overline{\Pi}_{(1)}&(\cdot\wedge D_k-)|\big)h(\overline{\Delta}^{(1,i_1(D_k-))}(D_k))\Big]  \\
&=\E\left[\int_0^{\infty}\mathrm d u \;f(u)\kappa_{\bar{\nu}^{(i_1^a(u))}}(A_k) e^{-\int_0^u \kappa_{\bar{\nu}^{(i_1(s))}}(A_k)\mathrm d s}\kappa_{\bar{\nu}^{(i_1^a(u))}}(h\mid A_k)g\big(|\overline{\Pi}^a_{(1)}(\cdot\wedge u)|\big)\right] \\
&=\E\left[\int_0^{\infty}\mathrm d u \;f(u)\kappa_{\bar{\nu}^{(i_1^a(u))}}(h\mathbf{1}_{A_k}) e^{-\int_0^u \kappa_{\bar{\nu}^{(i_1^a(s))}}(A_k)\mathrm d s}g\big(|\overline{\Pi}^a_{(1)}(\cdot\wedge u)|\big)\right] \\
&=\int_0^{\infty}\mathrm d u\;f(u)\E\left[ \kappa_{\bar{\nu}^{(i_1^a(u))}}(h\mathbf{1}_{A_k}) \pr(T>u\mid \overline{\Pi}^a_{(1)})\;g\big(|\overline{\Pi}^a_{(1)}(\cdot\wedge u)|\big)\right] \\
&=\int_0^{\infty}\mathrm d u\;f(u)\E\left[ \kappa_{\bar{\nu}^{(i_1^a(u))}}(h\mathbf{1}_{A_k}) \mathbf{1}_{\{T> u\}}g\big(|\overline{\Pi}^a_{(1)}(\cdot\wedge u)|\big)\right] \\
&=\int_0^{\infty}\mathrm d u\;f(u)\E\left[ \kappa_{\bar{\nu}^{(i_1(u))}}(h\mathbf{1}_{A_k}) \mathbf{1}_{\{D_k> u\}}g\big(|\overline{\Pi}_{(1)}(\cdot\wedge u)|\big)\right].
\end{align*}
The proof is then ended by noting that, by exchangeability, $\pr\left(D_k>u \mid \, |\overline{\Pi}_{(1)}(\cdot\wedge u)|\right)=|\Pi_{(1)}(u)|^{k-1}.$ For the details about this quite intuitive fact, we refer to the proof of \cite[Proposition 18]{HM12}, which is quite similar.
\end{proof}

\textbf{The multi-type fragmentation tree is the limit of its marginals.} By construction, the tree \ $\T^{(i)}_{\gamma,\bnu}$ \ is the limit of its marginals $\mathcal{R}(\T^{(i)}_{\gamma,\bnu},[k])$ in the Gromov-Hausdorff sense. We will need to incorporate the measures:
\begin{lem}\label{lem:marginalstowholetree} Let, for $k\in\N,$ $\eta_k$ be the uniform measure on the leaves $L_1,\ldots,L_k.$ We then have the following a.s.\ convergence for the GHP topology:
\[\big(\mathcal{R}(\T^{(i)}_{\gamma,\bnu},[k]),\eta_k\big){\underset{k\rightarrow \infty}{\longrightarrow}} \big(\T^{(i)}_{\gamma,\bnu},\mu^{(i)}_{\gamma,\bnu}\big).\]
\end{lem}
This is a straightforward consequence of Lemma 3.11 in \cite{Steph13}, which treats this type of convergence in a general case and says, with our notation, that the conclusion holds  as soon as the tree is compact and the blocks of the partitions $\overline \Pi(t), \overline \Pi(t-)$ all have asymptotic frequencies for all $t\geq 0$ and the process $t\mapsto |\overline \Pi_{(i)}(t-)|$ is left continuous for all $i$. The compactness of the tree is proved in \cite{Steph17} and the other assumptions can all be checked using the Poissonian construction.

\subsection{On the moments of the height of $T_n^{(i)}$}\label{sec:heightmoments}

We now turn to a crucial result, which will be needed both to prove the tightness of the sequences $(T_n^{(i)}/n^{\gamma})_n$ and to improve some results of \cite{HS18} on the scaling limits of bivariate Markov chains, which are two essential points of the proofs of our theorems undertaken in Section \ref{sec:proofs2}. We recall that we are working under Assumption (\ref{eq:simple}). 

\begin{prop}
\label{Hmoments}
Assume that either  the hypotheses of Theorem \ref{theo:critique} or those of Theorem \ref{theo:melange} are satisfied, and let $H_n^{(i)}$ denote the height of $T_n^{(i)}$, $n\geq 1$, $i \in [\kappa]$. Then for all types $i \in [\kappa]$ and all $p>0$,
$$
\sup_{n\geq 1}\frac{\mathbb E\big[\big(H_n^{(i)}\big)^p\big]}{n^{p\gamma}}<\infty.
$$ 
\end{prop}

A similar result was proved in the monotype setting, first in  \cite{HM12} by evaluating the tails of the variables $H_n^{(i)}$ and then more elegantly in \cite{Dadoun17+} by using martingale arguments. These approaches could be adapted here to prove the above proposition in the critical or solo regimes, but cannot be used directly in the mixing regime when at least one of the measures $\nu^{(i)}$ arising in (\ref{hypocvmixing}) is the zero measure. The issue is that a type $i$ corresponding to a zero measure $\nu^{(i)}$ then tends to slow down the process, which means that the trees tend to be larger. We give below a detailed proof of this most delicate mixing regime and then explain quickly how to adapt it to the critical and solo regimes.

So, throughout this section, except the last subsection, we are in the mixing regime and assume (\ref{hypocvmixing}) and (\ref{hypocvmixing2}), with at least the measure $\nu^{(1)}$ a true (i.e.\ nonzero) dislocation measure.
Note that by Jensen's inequality, it is sufficient to prove Proposition \ref{Hmoments} for $p$ large enough. In the following $p$ is fixed such that 
\begin{equation*}
\label{hypp}
p>1 \quad \text{and } \quad p\gamma+\beta-\gamma>1,
\end{equation*}
which in particular implies that 
\begin{equation}
\label{majo}
\sum_{\ell \geq 1} \lambda_\ell^{p\gamma+\beta-\gamma} \leq n^{p\gamma+\beta-\gamma} \text{ as soon as } \sum_{\ell \geq 1} \lambda_\ell \leq n,
\end{equation}
for any sequence $(\lambda_{\{\ell\}})_{\ell \geq 1}$ of non-negative terms, an argument we will use at several places. In this mixing regime, we will have to take into account that on average an individual of size $n$ needs a time of order $n^{\beta}$ to change type which leads us to prefer the tails control approach. Our strategy is to prove by induction on $n$ ($p$ being fixed) that 
\begin{equation}
\label{HypH_n}
\mathbb P\big( H_n^{(i)} \geq x n^{\gamma}\big) \leq \frac{D_p + D_p^{1-\frac{1}{2p}}\cdot C_i \cdot n^{\beta-\gamma}}{x^p}  \quad\text{ for all } x>0 \text{ and all types } i \in [\kappa], \tag{$\mathsf H_n$}
\end{equation}
where $D_p$ and $C_i, i \in [\kappa]$ are all finite and independent of $n$, and must be chosen subtly for the induction to work (see Section \ref{sec:prelim-constants} below). Clearly if such inequalities are valid for all $p$ large enough, Proposition \ref{Hmoments} holds since $\beta<\gamma$. 

In order to get (\ref{HypH_n}) for all $n$, we start below by setting some preliminary results, in particular we will choose the constants appearing in (\ref{HypH_n}) and set up some useful bounds. We then proceed to the proof of (\ref{HypH_n}) by induction, and finally explain how to adapt our proof to the critical and solo regimes.

\subsubsection{Preliminary work and notation}
\label{sec:prelim-constants}

\textbf{Finiteness of $\mathbb E[(H_n^{(i)})^p]$ for all $n\geq 1$.} We first note $H_1^{(i)}=0$ for all $i$, by (\ref{eq:simple}), and that for all $n\geq 1$ the height of the first branch-point above the root of $T_n^{(i)}$ (i.e.\ the number of steps needed to split $n$ starting from $(n,i)$) has an exponential moment,  by (\ref{hyp:principal}). The MB property and a straightforward induction then imply the finiteness of $\mathbb E[(H_n^{(i)})^p]$.

\bigskip 

\textbf{Choice of the constants $C_i,i \in [\kappa]$.} Let $Q$ be the irreducible $\mathrm Q$-matrix of dimension $\kappa \times \kappa$ arising in (\ref{hypocvmixing2}). 

\begin{lem}
\label{lem:PF}
There exist some strictly positive terms $(C_i,i\in[\kappa]\backslash\{1\})$ such that for all $i\in[\kappa]\backslash\{1\}$,
\begin{equation}\label{eq:Ci}
-Q(i,i) C_{i} > \sum_{j \in [\kappa]\backslash \{i,1\} } C_{j}Q(i,j).
\end{equation}
\end{lem} 

\begin{proof}
Since both sides of (\ref{eq:Ci}) are linear in $Q,$ we can, up to dividing it by a large constant, assume without loss of generality that $\max_{i\in[\kappa]} |Q(i,i)|<1.$ Consider then the probability transitions $P(i,j):=Q(i,j)$, $i\neq j$, and $P(i,i):=1+Q(i,i)$, and $\overline P$ the restriction of the matrix $P$ to $[\kappa]\backslash\{1\} \times [\kappa]\backslash\{1\}$. For each $i\in [\kappa]\backslash\{1\}$, we have $\sum_{j=2}^{\kappa} \overline P(i,j)\leq 1$. We let $I\subset [\kappa]\backslash\{1\}$ be the subset of indices $i$ such that this sum is strictly smaller than 1. It is nonempty by irreducibility of $P$. 

Let then $\varepsilon>0$. From $\overline P$ and $\varepsilon$ we construct another matrix $\overline P^{\varepsilon}$ with nonnegative terms, by setting $\overline P^{\varepsilon}(i,j):=\varepsilon$ for all $(i,j)$ such that $i \in I$, $j\neq i$ and $\overline P(i,j)=0$, and $ \overline P^{\varepsilon}(i,j):=\overline P(i,j)$ otherwise. Note that $\varepsilon$ may be chosen small enough so that $\sum_{j=2}^{\kappa} \overline P^{\varepsilon}(i,j)\leq 1$ for all $i \in [\kappa]\backslash\{1\}$ and $\sum_{j=2}^{\kappa} \overline P^{\varepsilon}(i,j)<1$ for at least one $i$, which we assume from now on.  Note also that $\overline P^{\varepsilon}$ is irreducible by irreducibility of $P$. We can therefore apply the Perron-Frobenius theorem to $\overline P^{\varepsilon}$, from which we get the existence of $0<\lambda<1$ and a vector $\textbf C=(C_i,i\in[\kappa]\backslash\{1\})$ of strictly positive terms such that  $\overline P^{\varepsilon} \mathbf C=\lambda \mathbf C$. So we have for all $i \in [\kappa]\backslash\{1\}$
\begin{eqnarray*}
\sum_{j \in  [\kappa]\backslash\{1\}} C_{j}Q(i,j)&=&\sum_{j \in  [\kappa]\backslash\{1\}} C_{j}P(i,j)-C_i \\
&\leq & \sum_{j \in  [\kappa]\backslash\{1\}} C_{j}\overline P^{\varepsilon}(i,j)-C_i \\
&=& \lambda C_i-C_i<0,
\end{eqnarray*}
as expected.
\end{proof}

From now on, $C_i,i\in [\kappa]\backslash\{1\}$ is a fixed sequence satisfying the inequalities of Lemma \ref{lem:PF} and we set $C_1:=0$. We also fix $\varepsilon>0$ small enough such that for all $i \in [\kappa] \backslash \{1\}$,
\begin{equation}
\label{C+eps}
\bigg(\sum_{j \in [\kappa] \backslash\{i\}}Q(i,j)-\varepsilon\bigg)C_{i}>\sum_{j \in [\kappa]\backslash\{i,1\}}(Q(i,j)+\varepsilon)C_j.
\end{equation}

\bigskip

\textbf{Some bounds.} We derive now some bounds from Hypotheses (\ref{hypocvmixing}) and (\ref{hypocvmixing2}), using (\ref{majo}). First, there exist $d_1,d_2<\infty,d_3>0$ such that for $n$ large enough
\begin{equation}
\label{inegH1}
\sum_{\bar \lambda \in \overline{\mathcal P}_n} q_n^{(1)}(\bar \lambda) \frac{\lambda_1^{p\gamma+\beta-\gamma}}{n^{p\gamma}}\mathbbm 1_{\{i_1\neq 1\}} \leq n^{\beta-\gamma} \cdot q_n^{(1)}(i_1 \neq 1) \leq n^{\beta-\gamma}  \cdot \frac{d_1}{n^{\beta}} =\frac{d_1}{n^{\gamma}},
\end{equation}
and
\begin{equation}
\label{inegH2}
\sup_{i\in [\kappa]}\sum_{\bar \lambda \in \overline{\mathcal P}_n} q_n^{(i)}(\bar \lambda) \sum_{\ell \geq 2}\frac{\lambda^{p\gamma+\beta-\gamma}_\ell}{n^{p\gamma}} \leq n^{\beta-\gamma} \cdot \frac{d_2}{n^{\gamma}},
\end{equation}
(here we use (\ref{hypocvmixing}) and the fact that the function $\mathbf s \mapsto \sum_{\ell \geq 2} s_{\ell}^{p\gamma+\beta-\gamma}(1-s_1)^{-1}\mathbbm 1_{\{s_1<1\}}$ is continuous on $\mathcal S^{\downarrow}$ since $p\gamma+\beta-\gamma>1$)
and
\begin{equation}
\label{inegH3}
\sum_{\bar \lambda \in \overline{\mathcal P}_n} q_n^{(1)}(\bar \lambda) \Bigg(1-\sum_{\ell \geq 1} \frac{\lambda^{p\gamma}_\ell}{n^{p\gamma}}\Bigg) \geq \frac{d_3}{n^{\gamma}}
\end{equation}
(again we use (\ref{hypocvmixing}) with $\nu^{(1)}$ nontrivial, and the fact that the function $\mathbf s \mapsto$ $(1-\sum_{\ell \geq 1} s_{\ell}^{p\gamma})(1-s_1)^{-1}\mathbbm 1_{\{s_1<1\}}$ is continuous on $\mathcal S^{\downarrow}\backslash\{(1,0,\ldots)\}$ -- note that $p\gamma>1$ here).
Second, using (\ref{hypocvmixing2}), we get that for $i\neq 1,j\neq 1, j \neq i$ and $\varepsilon>0$ chosen so that (\ref{C+eps}) holds, we have for $n$ large enough
\begin{equation}
\label{inegH4}
\sum_{\bar \lambda \in \overline{\mathcal P}_n} q_n^{(i)}(\bar \lambda)   \frac{\lambda^{p\gamma+\beta-\gamma}_1}{n^{p\gamma}}\mathbbm 1_{\{i_1= j\}} \leq n^{-\gamma} \cdot (Q(i,j)+\varepsilon)
\end{equation}
and for $i\neq 1$
\begin{equation}
\label{inegH5}
\sum_{\bar \lambda \in \overline{\mathcal P}_n} q_n^{(i)}(\bar \lambda) \frac{\lambda^{p\gamma+\beta-\gamma}_1}{n^{p\gamma}}\mathbbm 1_{\{i_1=i\}} \leq n^{\beta-\gamma} - n^{-\gamma}\cdot \bigg(\sum_{j\in [\kappa] \backslash \{i\}}Q(i,j)-\varepsilon \bigg).
\end{equation}
\bigskip

\textbf{Choice of an intermediate fixed constant: $n_0$.} From now on $n_0$ denotes an integer large enough so that the inequalities (\ref{inegH1}),  (\ref{inegH2}),  (\ref{inegH3}),  (\ref{inegH4}),  (\ref{inegH5}) hold for $n \geq n_0$ and
$$
\bigg(\sum_{j \in [\kappa] \backslash\{i\}}Q(i,j)-\varepsilon\bigg)C_{i}-\sum_{j \in [\kappa]\backslash\{i,1\}}(Q(i,j)+\varepsilon)C_j-\frac{\max_{i \in [\kappa] \backslash \{1\}}C_i d_2}{n_0^{\gamma-\beta}}>0
$$ 
(recall that $\beta<\gamma$, hence $n_0$ indeed exists).

\bigskip

\textbf{Choice of $D_p$.} The  $C_i,i\in [\kappa]\backslash\{1\}$, $\varepsilon>0$, $n_0 \in \mathbb N$ and $d_1,d_2\geq 0,d_3>0$ being fixed, we now choose a positive real number $D_p\geq 1$ large enough such that  the four following inequalities hold:
\begin{equation}
\label{inegC1}
d_3 \geq \frac{2p}{D_p^{1/p}}+\frac{\max_{i\in [\kappa]\backslash\{1\}} C_i}{D_p^{1/2p}} \cdot (d_1+d_2)\left(1+2p\right),\
\end{equation}

\smallskip

\begin{equation}
\label{inegC3}
\sup_{1\leq n \leq n_0} \frac{\mathbb E[(H_{n}^{i})^p]}{n^{\gamma p}} \leq D_p,
\end{equation}

\smallskip

\begin{equation}
\label{inegC4}
\frac{1}{(x-1)^p} \leq \frac{1}{x^p}+ \frac{2p}{x^{p+1}}, \quad \forall x \geq D^{1/p}_p,
\end{equation}
and
\begin{eqnarray}
\label{inegC2}
\left(\sum_{j \in [\kappa] \backslash\{i\}}Q(i,j)-\varepsilon\right)C_{i}-\sum_{j \in [\kappa]\backslash\{i,1\}}(Q(i,j)+\varepsilon)C_j-\frac{\max_{i \in [\kappa] \backslash \{1\}}C_i d_2}{n_0^{\gamma-\beta}} \\
\geq 2p \left(\frac{1}{{D_p}^{1/2p}}+\frac{\max_{i\in[\kappa] \backslash \{1\}}C_i}{{D_p}^{1/p}n_0^{\gamma-\beta}}\right).
\end{eqnarray}

\subsubsection{Proof of $(\mathsf H_n)$ by induction}

We will in fact prove by induction on $n \geq 1$ and $k\geq 1$ that 
\begin{equation}
\label{hyp:rechauteur}
\mathbb P\big( H_n^{(i)} \geq x n^{\gamma}\big) \leq \frac{D_p + D_p^{1-\frac{1}{2p}} \cdot C_i \cdot n^{\beta-\gamma}}{x^p}, \quad \forall x\in (0,k/n^{\gamma}), \quad \forall i \in [\kappa] 
\tag{$\mathsf H_{n,k}$}.
\end{equation}

\bigskip

\textbf{Initialisation.} By (\ref{inegC3}) and Markov's inequality, $(\mathsf H_n)$ holds for all $n \leq n_0$. Also, for all $n\geq 1$, $(\mathsf H_{n,k})$ holds for all $k\leq D_p^{1/p} n^{\gamma}$ since $\mathbb P\big( H_n^{(i)} \geq x n^{\gamma}\big) \leq 1 \leq D_p/x^p$ for all $x\leq D_p^{1/p}$. In particular for all $n\geq 1$, $(\mathsf H_{n,1})$ holds, since we have chosen $D_p \geq 1$.

\bigskip

\textbf{Induction.} Let $n \geq  n_0+1$. We assume that $(\mathsf H_m)$ holds for all $m\leq n-1$ and that $(\mathsf H_{n,k})$ holds for some $k\geq 1$. We want to prove $(\mathsf H_{n,k+1})$, which is sufficient to conclude. In that aim, we use the following consequence of the MB property: for all types $i \in [\kappa]$
$$
H_n^{(i)}\overset{\mathrm{law}}=1+\max_{\ell \geq 1} H_{\Lambda_\ell}^{(i_\ell)}
$$
where $(\Lambda,\mathbf i) \sim q_n^{(i)}$ and the heights in the right-hand side are independent given $(\Lambda,\mathbf i)$. 

Now fix $x \in \left[\frac{k}{n^{\gamma}}, \frac{k+1}{n^{\gamma}}\right)$. We want to prove that $\mathbb P\big( H_n^{(i)} \geq x n^{\gamma}\big) \leq \frac{D_p + D_p^{1-1/2p} \cdot C_i \cdot n^{\beta-\gamma}}{x^p}$ for all types $i$. If $x\leq D_p^{1/p}$ this is obvious as noticed above. We can therefore assume that $x>D_p^{1/p}$and therefore apply (\ref{inegC4}) when needed.
Then for any type $i$, by the induction hypotheses (since $n^{\gamma}x-1< k)$, we have that
\begin{eqnarray*}
\mathbb P(H_n^{(i)} < n^{\gamma} x) &=& \sum_{\bar \lambda \in \overline{\mathcal P}_n} q_n^{(i)}(\bar \lambda) \prod_{\ell \geq 1} \mathbb P(H_{\lambda_{\ell}}^{(i_{\ell})} < n^{\gamma}x-1) \\
&\geq & \sum_{\bar \lambda \in \overline{\mathcal P}_n} q_n^{(i)}(\bar \lambda) \prod_{\ell \geq 1} \left(1- \frac{(D_p+D_p^{1-1/2p} C_{i_{\ell}} \cdot \lambda_{\ell}^{\beta-\gamma})\lambda_{\ell}^{\gamma p}}{(n^{\gamma}x-1)^p}\right)_+ \\
&\geq & \sum_{\bar \lambda \in \overline{\mathcal P}_n} q_n^{(i)}(\bar \lambda) \left( 1-\sum_{\ell \geq 1} \frac{(D_p+D_p^{1-1/2p} C_{i_{\ell}}\cdot \lambda_{\ell}^{\beta-\gamma})\lambda_{\ell}^{\gamma p}}{(n^{\gamma}x-1)^p}\right)\\
&\underset{(\text{by } (\ref{inegC4}))}{\geq} & \sum_{\bar \lambda \in \overline{\mathcal P}_n} q_n^{(i)}(\bar \lambda)  \left( 1-\frac{1}{(xn^{\gamma})^p}\left(1+\frac{2p}{xn^{\gamma}}\right)
\sum_{\ell \geq 1}  (D_p+D_p^{1-1/2p} C_{i_{\ell}} \cdot \lambda_{\ell}^{\beta-\gamma})\lambda_{\ell}^{\gamma p}\right) \\
&\geq & 1-\frac{D_p}{x^p}+\frac{D_p}{x^p} \sum_{\bar \lambda \in \overline{\mathcal P}_n} q_n^{(i)}(\bar \lambda) \left(1-\sum_{\ell \geq 1} \frac{\lambda^{p\gamma}_\ell}{n^{p\gamma}}\right)-\frac{2pD_p}{x^{p+1}n^{\gamma}}\sum_{\bar \lambda \in \overline{\mathcal P}_n} q_n^{(i)}(\bar \lambda)\sum_{\ell \geq 1} \frac{\lambda^{p\gamma}_\ell}{n^{p\gamma}} \\
&-&\frac{D_p^{1-1/2p}}{x^p}\left(1+\frac{2p}{xn^{\gamma}}\right) \sum_{j\in [\kappa]\backslash \{1\}} C_{j}  \sum_{\bar \lambda \in \overline{\mathcal P}_n} q_n^{(i)}(\bar \lambda)  \frac{\lambda^{p\gamma+\beta-\gamma}_1}{n^{p\gamma}}\mathbbm 1_{\{i_1 = j\}} \\
&-&\frac{D_p^{1-1/2p}}{x^p}\left(1+\frac{2p}{xn^{\gamma}}\right) \max_{j\in [\kappa]\backslash \{1\}} C_{j}  \sum_{\bar \lambda \in \overline{\mathcal P}_n} q_n^{(i)}(\bar \lambda) \sum_{\ell \geq 2} \frac{\lambda^{p\gamma+\beta-\gamma}_\ell}{n^{p\gamma}}
\end{eqnarray*}
(in the last inequality we have used that $C_1=0$). Note that $\sum_{\ell \geq 1} \frac{\lambda^{p\gamma}_\ell}{n^{p\gamma}} \leq 1$ since $p\gamma \geq 1$. 

$\bullet$ When $i=1$, using (\ref{inegH1}), (\ref{inegH2}) and (\ref{inegH3}), this leads to
\begin{eqnarray*}
\mathbb P(H_n^{(1)} < n^{\gamma} x) &\geq& 1-\frac{D_p}{x^p}+\frac{1}{x^p n^{\gamma}}\left(D_p d_3-\frac{2pD_p}{x} -D_p^{1-1/2p}\max_{j\in [\kappa]\backslash \{1\}} C_{j} (d_1+d_2)\left(1+\frac{2p}{x n^{\gamma}}\right)\right) \\
&\geq &  1-\frac{D_p}{x^p}+\frac{D_p}{x^p n^{\gamma}}\left(d_3-\frac{2p}{D_p^{1/p}} -\frac{\max_{j\in [\kappa]\backslash \{1\}} C_{j}}{D_p^{1/2p}} (d_1+d_2)\left(1+2p\right)\right),
\end{eqnarray*}
since $x\geq D_p^{1/p}$ and $n\geq n_0 \geq 1$. The term in the large brackets on the right-hand side of the second inequality is nonnegative by (\ref{inegC1}), hence the expected lower bound for $\mathbb P(H_n^{(1)} < n^{\gamma} x)$.

$\bullet$ When $i \neq 1$, by (\ref{inegH5}), (\ref{inegH4}) and (\ref{inegH2}), we get that 
\begin{eqnarray*}
&& \mathbb P(H_n^{(i)} < n^{\gamma} x) \\
&\geq& 1-\frac{D_p}{x^p}-\frac{1}{x^p n^{\gamma}}\frac{2pD_p}{x}-\frac{D_p^{1-1/2p}}{x^p}\left(1+\frac{2p}{xn^{\gamma}}\right) \cdot C_i \cdot n^{\beta-\gamma} \\
&& \hspace{-0.5cm} + \frac{D_p^{1-1/2p}}{x^p n^{\gamma}} \left(1+\frac{2p}{xn^{\gamma}}\right)\Bigg(\bigg(\sum_{j \in [\kappa] \backslash\{i\}}Q(i,j)-\varepsilon\bigg)C_{i} -\sum_{j \in [\kappa]\backslash\{i,1\}}(Q(i,j)+\varepsilon) C_j-\frac{\max_{j\in [\kappa]\backslash \{1\}} C_{j} d_2}{n^{\gamma-\beta}}\Bigg)\\
&\geq & 1-\frac{D_p}{x^p}-\frac{D_p^{1-1/2p}}{x^p} \cdot C_i \cdot n^{\beta-\gamma} -\frac{D_p^{1-1/2p}}{x^p n^{\gamma}} 2p \left(\frac{D_p^{1/2p}}{x}+\frac{C_i}{xn^{\gamma-\beta}}\right)\\
&&  \hspace{-0.5cm} + \frac{D_p^{1-1/2p}}{x^p n^{\gamma}} \left(1+\frac{2p}{xn^{\gamma}}\right)  \Bigg(\bigg(\sum_{j \in [\kappa] \backslash\{i\}}Q(i,j)-\varepsilon\bigg)C_{i} -\sum_{j \in [\kappa]\backslash\{i,1\}}(Q(i,j)+\varepsilon) C_j -\frac{\max_{j\in [\kappa]\backslash \{1\}} C_{j} d_2}{n^{\gamma-\beta}}\Bigg) 
\end{eqnarray*}
which is larger than $1-\frac{D_p}{x^p}-\frac{D_p^{1-1/2p} \cdot C_{i} \cdot n^{\beta-\gamma}}{x^p} $ (as expected) by (\ref{inegC2}) and since $x \geq D_p^{1/p}$ and $n\geq n_0$.

\subsubsection{Critical and solo regimes}

Under Hypotheses (\ref{hypocvcritique}) and (\ref{hypocvcritique2}), or (\ref{hypocvcritique}) and (\ref{hypocvcritique3}), the proof of Proposition \ref{Hmoments} follows the same lines as in the mixing regime but is much easier, essentially because we do not need to introduce the constants $C_i,i \in [\kappa]$. In fact, the proof here consists simply in showing by induction on $n$ that
$$
\mathbb P\big( H_n^{(i)} \geq x n^{\gamma}\big) \leq \frac{D_p}{x^p}  \quad\text{ for all } x>0 \text{ and all types } i \in [\kappa], 
$$ 
for all $p$ such that $p\gamma>1$. To prove this, we do not need as many inequalities as before, the only point is to note the following version of (\ref{inegH3}): there exists a constant $d>0$ such that \emph{for all} $i\in [\kappa]$
$$
\sum_{\bar \lambda \in \overline{\mathcal P}_n} q_n^{(i)}(\bar \lambda) \Bigg(1-\sum_{\ell \geq 1} \frac{\lambda^{p\gamma}_\ell}{n^{p\gamma}}\Bigg) \geq \frac{d}{n^{\gamma}}, \quad \text{for all } n \text{ large enough.}
$$
This is true since each $\bar \nu^{(i)}$ in the convergence (\ref{hypocvcritique}) is assumed to be a dislocation measure, and in particular  $\bar \nu^{(i)}(s_1<1)>0$. Then, as before, we need to choose an integer $n_0$ large enough so that these inequalities hold for all $n \geq n_0$, and a constant $D_p$ large enough such that (\ref{inegC1}) (with each $C_i$ replaced by 0 and $d_3$ by $d$), (\ref{inegC3}) and (\ref{inegC4}) hold. The induction then follows easily. 

\subsection{Scaling limits of bivariate Markov chains in the mixing regime}

As already mentioned, a key point in the proofs of our theorems is to describe the scaling limits of typical paths of multi-type MB trees. This will rely on previous results obtained in \cite{HS18} on the scaling limits of  bivariate Markov chains, which are informally a $1$-dimensional version of this work. The results of \cite{HS18} are however not sufficient to treat the mixing regime situation where some of the measures $\nu^{(i)}$ in the convergence (\ref{hypocvmixing}) are null.  One consequence of Proposition \ref{Hmoments} is the improvement of the results of \cite{HS18} to this case. We first quote the $1$-dimensional convergences as they appear in \cite{HS18}, and then their extension to the general situation that we need here.

\begin{thm}[\cite{HS18},Theorem 4.1 and Theorem 4.2]\label{th:old1dmixing} Let $(p_{n,i}(m,j))$ be the transition rates of a bivariate Markov chain $(X,J)$ on $\Z_+\times [\kappa]$ with $X$ nonincreasing. Let also, for $n,m\in \Z_+$ and $i\in[\kappa],$ $p_n^{(i)}(m)=\sum_{j\in[\kappa]}p_{n,i}(m,j).$ Assume that there exists $0\leq\beta<\gamma$ such that:
\begin{enumerate}
\item[\emph{(i)}] There exist finite measures $(\mu^{(i)},i \in [\kappa])$ on $(0,1),$ at least one of which is nontrivial, such that, for all continuous functions $f:[0,1]\rightarrow \mathbb R$,
\[
n^{\gamma} \sum_{m=0}^n f\left(\frac{m}{n}\right) \left(1-\frac{m}{n}\right)p_{n}^{(i)}(m) \underset{n\rightarrow \infty}\longrightarrow \int_{[0,1]} f(x)\mu^{(i)}(\mathrm dx). 
\]
\item[\emph{(ii)}]  Moreover, there exists an irreducible $\mathsf Q$--matrix $Q=(q_{i,j})_{i,j\in[\kappa]}$ such that
\[
n^{\beta}(P_n-I) \underset{n\rightarrow \infty}\longrightarrow Q.
\]
\end{enumerate}
Let, for $i\in[\kappa]$ and $\lambda\geq 0$, $\psi_i(\lambda)=\int_0^1(1-x^{\lambda})\frac{\mu^{(i)}(\mathrm d x)}{1-x},$ and then define the mixed Laplace exponent $\psi=\sum_{i\in[\kappa]}\chi_i\psi_i$ where $\chi$ is the invariant measure of $Q.$
Then, calling $(X_n^{(i)},J_n^{(i)})$ a version of the Markov chain starting from $(n,i)\in\Z_+\times [\kappa],$ we have the following convergence in distribution for the Skorokhod topology:
\[
\left(\frac{X^{(i)}_n(\lfloor n^{\gamma}t \rfloor)}{n}, t\geq 0\right) \ \overset{\mathrm{(d)}}{\underset{n \rightarrow \infty} \longrightarrow} (X(t),t\geq 0),
\]
where $(X(t),t\geq 0$ is the $\gamma$-Lamperti transform of a subordinator with Laplace exponent $\psi.$

Moreover, \emph{if all the $\mu^{(i)}$ are nontrivial}, then calling $A_n^{(i)}$ the absorption time of $X_n^{(i)}$ and $A$ that of $X$, we have, jointly with the above,
\[ \frac{A_n^{(i)}}{n^{\gamma}}\overset{\mathrm{(d)}}{\underset{n \rightarrow \infty} \longrightarrow}A,\]
and for all $p \geq 0$,
\[\mathbb E\left[\left(\frac{A_n^{(i)}}{n^{\gamma}}\right)^p \right] {\underset{n \rightarrow \infty} \longrightarrow} \mathbb E \left[ A^p\right].\]
\end{thm}

The extra assumption in the second part, that all the measures are nontrivial, is in fact unnecessary:

\begin{cor}\label{cor:tempsabsorption} The full conclusions of Theorem \ref{th:old1dmixing} stay true if we only assume that at least one of the measures $\mu^{(i)}, i \in [\kappa]$ is nontrivial.
\end{cor}
\begin{proof} The key point to adapt the proofs of Theorem 4.1 and Theorem 4.2 of \cite{HS18}  to our more general situation is to show that, for all $p\geq 0$ and all $i\in [\kappa]$, $n^{-p\gamma}\mathbb E[(A_n^{(i)})^{p}]$ is uniformly bounded in $n$. This is a straightforward consequence of Proposition \ref{Hmoments}, since  $A_n^{(i)}$ is in fact the height of a linear MB tree satisfying the assumptions of Theorem \ref{theo:melange}. In the following we let
$$
\mathrm c_A:=\sup_{i \in [\kappa], n\geq 1} \frac{\mathbb E[A_n^{(i)}]}{n^{\gamma}} <\infty.
$$
Introduce then for all $n\geq 1,i \in [\kappa]$ and $t\geq 0$, 
$$\tau_n^{(i)}(t):=\inf \left\{u\geq 0 : \int_0^u \left(\frac{X^{(i)}_n(\lfloor n^{\gamma} r\rfloor)}{n}\right)^{-\gamma}\mathrm dr>t \right\}, \qquad  Z_n^{(i)}(t):=\frac{X^{(i)}_n(\lfloor n^{\gamma} \tau^{(i)}_n(t)\rfloor)}{n}$$
and note that 
\begin{equation*}
\tau^{(i)}_n(t)=\int_0^t (Z^{(i)}_n(r))^{\gamma}\mathrm dr, \quad \forall t \geq 0,  \qquad \text{and} \qquad \frac{A_n^{(i)}}{n^{\gamma}}=\int_0^{\infty}(Z^{(i)}_n)^{\gamma}(r) \mathrm dr.
\end{equation*}
Theorem 4.1 of \cite{HS18} in fact asserts that, as soon as at least one of the measures $\mu^{(i)}, i \in [\kappa]$ is nontrivial,
\begin{equation}
\label{cvjointHS18}
\left(\frac{X^{(i)}_n(\lfloor n^{\gamma} \cdot \rfloor)}{n},Z^{(i)}_n \right) \underset{n \rightarrow \infty}{\overset{\mathrm(d)}\longrightarrow} (X,Z), \quad \text{ for all types }i \in [\kappa]
\end{equation}
where $X$ is defined in Theorem \ref{th:old1dmixing}, $Z:=X \circ \tau$ and $\tau(t):=\inf\{u \geq 0:\int_0^u (X(r))^{-\gamma} \mathrm dr>t\}$. We recall that $A$ denotes the absorption time $X$. Similarly to the discrete case, we have that $\tau(t)=\int_0^{t} (Z(r))^{\gamma}\mathrm dr, \forall t \geq 0$ and $A=\int_0^{\infty}(Z(r))^{\gamma} \mathrm dr$. Our goal is to deduce from this and the finiteness of $\mathrm c_A$ the joint convergence of 
$$
\left(\frac{X^{(i)}_n(\lfloor n^{\gamma} \cdot \rfloor)}{n},\frac{A^{(i)}_n}{n^{\gamma}} \right) \quad \text{to} \quad (X,A)
$$
(the convergence of the moments $\mathbb E[n^{-\gamma p}(A_n^{(i)})^p]$ will then follow immediately from the boundedness of these moments for all $p\geq 1$). By the Skorokhod representation theorem we may assume that (\ref{cvjointHS18}) holds almost surely, which we do from now on. Note in particular that a.s. for almost every $r \geq 0$, $Z_n^{(i)}(r) \rightarrow Z(r)$, and that all these quantities are between 0 and 1.

From (\ref{cvjointHS18}) and  the finiteness of $\mathrm c_A$, we see that the sequence $(n^{-1}X^{(i)}_n(\lfloor n^{\gamma}\cdot \rfloor), Z_n^{(i)}, n^{-\gamma}A^{(i)}_n)_n$ is tight. Let $(\mathcal X,\mathcal Z,\mathcal A)$ denote a possible limit of a subsequence, say along the subsequence $(\sigma(n))_n$. If we prove that $(\mathcal X,\mathcal Z,\mathcal A)$ is distributed as $(X,Z, A)$, the proof of the corollary will be finished. By (\ref{cvjointHS18}), $(\mathcal X, \mathcal Z)$ is distributed as $(X,Z)$ and so our goal is to prove that $\mathcal A = \int_{0}^{\infty} (\mathcal Z(r))^{\gamma} \mathrm dr$ a.s. Note that, by (\ref{cvjointHS18}) again and Fatou's lemma, $\int_{0}^{\infty} (\mathcal Z(r)^{\gamma} \mathrm dr \leq \mathcal A$ a.s. It is therefore sufficient to prove that $\mathbb E[\mathcal A] \leq \mathbb E[ \int_{0}^{\infty} (\mathcal Z(r))^{\gamma} \mathrm dr]$.
In that aim, note that for all $t_0 \geq 0$, by dominated convergence
$$
\mathbb E\big[\tau_{n}^{(i)}(t_0)\big]=\mathbb E\bigg[\int_0^{t_0}(Z_{n}^{(i)}(r))^{\gamma}\mathrm dr\bigg] \underset{n \rightarrow \infty}\longrightarrow \mathbb E\bigg[\int_0^{t_0}(Z(r))^{\gamma}\mathrm dr\bigg]=\mathbb E[\tau(t_0)]. 
$$
Then, a key observation is that for all $t_0 \geq 0$,  $\lfloor n^{\gamma}\tau^{(i)}_n(t_0) \rfloor$ is a stopping time with respect to the filtration generated by $(X_n^{(i)},J_n^{(i)})$ and thus:
$$
A^{(i)}_n=\lfloor n^{\gamma}\tau^{(i)}_n(t_0) \rfloor+\tilde A^{(J_n^{(i)}(\lfloor n^{\gamma}\tau^{(i)}_n(t_0)\rfloor ))}_{X_n^{(i)}(\lfloor n^{\gamma}\tau^{(i)}_n(t_0)\rfloor)}
$$
where $\tilde A$ is distributed as $A$ and independent of $(X_n^{(i)},J_n^{(i)})(\lfloor n^{\gamma}\tau^{(i)}_n(t_0)\rfloor)$. Consequently,
$$
\mathbb E\Bigg[\frac{A^{(i)}_n}{n^{\gamma}} \Bigg] \leq \mathbb E\big[\tau^{(i)}_n(t_0)\big]+ \mathrm c_A \mathbb E\big[\big(Z^{(i)}_n(t_0)\big)^{\gamma}\big].
$$
For all $\varepsilon>0$, we have that $\mathrm c_A \mathbb E[(Z(t_0))^{\gamma}] \leq \varepsilon$ for $t_0$ large enough (since $Z$ is non-increasing and converges to 0) and therefore for a.e. $t_0$ large enough,
$$
\limsup_{n\rightarrow \infty} \mathbb E\Bigg[\frac{A^{(i)}_n}{n^{\gamma}}\Bigg] \leq \mathbb E[\tau(t_0)]+\varepsilon \leq  \mathbb E\Bigg[\int_0^{\infty} (Z(r))^{\gamma} \mathrm dr\Bigg] + \varepsilon.
$$
Besides, $\sigma(n)^{-\gamma} A^{(i)}_{\sigma(n)}$ converges  in distribution to $\mathcal A$ and the moments $n^{-p\gamma}\mathbb E[(A_n^{(i)})^{p}]$ are all bounded, so $\mathbb E[\sigma(n)^{-\gamma} A^{(i)}_{\sigma(n)}]$ converges to $\mathbb E[\mathcal A]$. 

Finally we have shown that $\mathbb E[\mathcal A] \leq \mathbb E[\int_0^{\infty} (Z(r))^{\gamma} \mathrm dr] + \varepsilon$ \ for all $\varepsilon>0$, hence  $\mathbb E[\mathcal A] \leq \mathbb E[\int_0^{\infty} (Z(r))^{\gamma} \mathrm dr]$ as required.
\end{proof}

\section{Scaling limits of MB trees: proofs of Theorems \ref{theo:critique} and \ref{theo:melange}}
\label{sec:proofs2}

For the main part of the proofs, we follow the standard structure where one proves the convergence of the finite-dimensional marginals (in the sense of Aldous: the $k$-dimensional marginal is the subtree spanned by $k$ exchangeable leaves) and then establishes a tightness property to show that the whole tree is close to the marginals with large dimension. As  already mentioned, some lines of our proofs will not differ too much from those in the monotype setting, and we mainly detail the difference arising in the multi-type framework. The proof of the convergence of the finite-dimensional marginals is done by induction on $k$. We will start for $k=1$ with the asymptotic description of a bivariate Markov chain that describes the sizes and types of the embedded subtrees containing a typical leaf, relying on results of \cite{HS18} (though those results are not strong enough in the mixing case if some of the limiting measures are null, thus we use the improvement from Corollary \ref{cor:tempsabsorption}). The induction process is then based on the MB property. For this, the multi-type structure is more involved in the critical case than in the mixing case, and so we will at times focus specifically on it. The proof of tightness mostly differs from the monotype one because Proposition \ref{Hmoments} is more difficult to establish in the mixing case, but proceeds similarly to \cite{HM12} once we have Proposition \ref{Hmoments}. 

Throughout the section, $(q_n(i))$ is  a sequence of splitting distributions that satisfies (\ref{hyp:principal}), (\ref{eq:simple}) and either the hypotheses of Theorem \ref{theo:critique} or Theorem \ref{theo:melange}, and $(T_n^{(i)})$ is an associated sequence of MB trees. We use the notations introduced in Section \ref{sec:proofs1}, without redefining them.
Section \ref{sec:cvfd} is devoted to the finite-dimensional convergences and Section~\ref{sec:tightness} to the tightness, which will lead to the Gromov-Hausdorff convergence of the rescaled trees. Finally, we add the measures in Section \ref{sec:measure}.

\subsection{Convergence of finite-dimensional marginals}
\label{sec:cvfd}

We now have the tools to prove the following:

\begin{prop}
\label{prop:cvfd}
Assume that the splitting distributions $(q_n^{(i)})$ satisfy the hypotheses of Theorem \ref{theo:critique}, with in the limit a vector $\bnu$ of dislocation measures. Then, jointly for all finite $B \subset \mathbb N$,
$$
n^{-\gamma} \cdot \mathcal R(T_n^{(i)},B) \overset{(d)}{\underset{n \rightarrow \infty} \longrightarrow} \mathcal R(\mathcal T_{\gamma,\bnu},B)
$$
for the Gromov-Hausdorff topology.

Similarly, if $(q_n^{(i)})$ satisfies the hypotheses of Theorem \ref{theo:melange}, then jointly for all finite $B \subset \mathbb N$,
$$
n^{-\gamma} \cdot \mathcal R(T_n^{(i)},B) \overset{(d)}{\underset{n \rightarrow \infty} \longrightarrow} \mathcal R(\mathcal T_{\gamma,\nu},B)
$$
for the Gromov-Hausdorff topology, where $\nu=\sum_{i=1}^{\kappa}\chi_i\nu^{(i)}$, with the notation  $(\chi_i,\nu^{(i)})$ introduced in Theorem  \ref{theo:melange}.
\end{prop}

The joint convergence follows straightforwardly from the individual convergences since for $B' \subset B$, the $B'$-marginal is embedded in the $B$-marginal (by using the same leaf labels). We therefore just have to prove the convergence for each $B$. By exchangeability, we can moreover restrict ourselves to $B=[k]$, for $k\geq 1$. We will proceed by induction on $k$.

\subsubsection{The case of the 1-dimensional marginals: $k=1$}

The case $k=1$ amounts to describing the asymptotic behavior of $D_1^{(n)}$, which can be interpreted as the absorption time of the tagged fragment chain $(X_n,J_n)$ defined in (\ref{def:discretetag}). We will show that divided by $n^{\gamma}$, this absorption time converges in distribution to the absorption time of a $\gamma$-Lamperti transformed Markov additive process with characteristic $\bnu$ in the critical or solo cases and $\nu$ in the mixing case, with $\bnu,\nu$ defined in Proposition \ref{prop:cvfd} above. To iterate and move to higher integers $k$,  we will in fact need more information. In that aim we will describe here the asymptotic of the whole process $X_n$ and, in the critical case, a stronger version modified to include the types if we only go up to a bounded number of type changes. In that aim we will use results of \cite{HS18} as well as the complementary Corollary \ref{cor:tempsabsorption}, that give sufficient conditions on the transition probabilities of bivariate Markov chains to ensure their convergence to some Lamperti-transformed Markov additive process after an appropriate rescaling. 
We recall that the transition probabilities of $(X_n,J_n)$ are given by $(\ref{transitions})$. 

All convergences of processes listed below hold with respect to the Skorokhod topology on appropriate spaces. Corollary \ref{cor:hypok1} and Corollary \ref{cor:hypok2} imply Proposition \ref{prop:cvfd} when $k=1$.

\bigskip

\textbf{Critical and solo cases.} We assume that $(q_n^{(i)})$ satisfy the hypotheses of Theorem \ref{theo:critique}.

\medskip

\begin{lem}
\label{lem:hypok1}
For all types $i$ and $j$ and any continuous function $f$ on $[0,1]$,
\begin{equation*}
\label{Hcr}
n^{\gamma}\sum_{m=0}^n \left(1-\mathbf{1}_{\{j=i\}}\frac{m}{n}\right)f\left(\frac{m}{n}\right)p_{n,i}(m,j) \underset{n\to\infty}\longrightarrow \int_{\overline{\s}^{\downarrow}} \sum_{\ell:i_{\ell}=j}s_{\ell}(1-s_{\ell}\mathbf{1}_{\{j=i\}})f(s_{\ell}) \bar{\nu}^{(i)}(\mathrm d\bar{\mathbf{s}}).
\end{equation*}
\end{lem}

\medskip

As a consequence, the transition probabilities of $(X_n,J_n)$ satisfy the assumptions of the ``critical regime" and ``solo regime" theorems of \cite{HS18}. Specifically, assuming (\ref{hypocvcritique2}) then we can use Theorem 3.1, 3.2 and Lemma 3.3 of \cite{HS18}, and if we assume (\ref{hypocvcritique3}) we can use Theorem 5.1. This leads to the following scaling limit in distribution:
\begin{cor}
\label{cor:hypok1}
Let $(X,J)$ denote a $\gamma$-Lamperti transform of a Markov additive process with characteristic $\bnu$, and $D_1$ the absorption time at 0 of $X$. Then
\begin{equation}\label{eq:1dimwithdeathtime}
\left(\frac{X_n(\lfloor n^{\gamma}\cdot \rfloor )}{n},\frac{D_1^{(n)}}{n^{\gamma}}\right) \overset{(d)}{\underset{n\to\infty}\longrightarrow} \big(X(\cdot),D_1\big).
\end{equation}
Moreover, \emph{in the critical case}, let $S(p)$ denote the time of $p$-th type change of $(X,J)$ and $S_n(p)$ the time of $p$-th type change of $(X_n,J_n)$, for all  $p\in\N$ (with the convention that these time changes are equal to the absorption time if there are fewer than $p$ time changes before the process is absorbed). We then have 
\[\left(\frac{X_n(\lfloor n^{\gamma}(\cdot \wedge S_n(p))\rfloor)}{n},J_n(\lfloor n^{\gamma}(\cdot \wedge S_n(p))\rfloor), \frac{S_n(p)}{n^{\gamma}}\right) \overset{(d)}{\underset{n\to\infty}\longrightarrow} \big(X(\cdot\wedge S(p)),J(\cdot\wedge S(p)),S(p)\big).\]
\end{cor}

\medskip

\textit{Proof of Lemma \ref{lem:hypok1}.}
Let the types $i,j$ and $f$, a continuous function on $[0,1]$, be fixed. In order to prove the convergence of the statement,
rewrite the left-hand side as 
\begin{equation*}
n^{\gamma}\sum_{\bar{\lambda}\in\overline{\mathcal{P}}_n}q_n^{(i)}(\bar{\lambda})\sum_{\ell:i_{\ell}=j}\left(1-\mathbf 1_{\{j=i\}}\frac{\lambda_\ell}{n} \right)\frac{\lambda_{\ell}}{n}f\left(\frac{\lambda_{\ell}}{n}\right)=n^{\gamma}\sum_{\bar{\lambda}\in\overline{\mathcal{P}}_n}q_n^{(i)}(\bar{\lambda})\left(1-\frac{\lambda_1}{n}\mathbf{1}_{\{i_1=i\}}\right)g\left(\frac{\bar{\lambda}}{n}\right)
\end{equation*}
where, for $\overline{\mathbf{s}}\neq ((1,i),(0,0),\ldots),$ \[g(\mathbf{\bar s})=\frac{1}{1-s_1\mathbf{1}_{\{i_1=i\}}}\sum_{\ell:i_{\ell}=j}(1-\mathbf 1_{\{j=i\}}s_\ell)s_{\ell} f(s_\ell),\] while $g((1,i),(0,0),\ldots)$ can be arbitrarily set to $0$. Note that $g$ is bounded on $\overline{\s}^{\downarrow}$ by $2\sup_{x\in[0,1]} |f(x)|$ and is continuous at each point $\overline{\mathbf s}$ such that $\sum_{\ell=1}^{\infty}s_{\ell}=1$, except possibly at the point $((1,i),(0,0),\ldots)$, which is not an atom of $\bar{\nu}^{(i)}$ by assumption. Thus, by (\ref{hypocvcritique}) we directly have the expected convergence.
$\hfill \square$

\bigskip

\textbf{Mixing case.} We assume that $(q_n^{(i)})$ satisfy the hypotheses of Theorem \ref{theo:melange}, with $\beta<\gamma$.
\begin{lem}
\label{lem:hypok2}
We then have for all types $i$ and all continuous functions $f$ on $[0,1]$
\begin{equation*}
\label{Hmix1}
n^{\gamma}\sum_{j\in[\kappa]}\sum_{m=0}^n \left(1-\frac{m}{n}\right)f\left(\frac{m}{n}\right)p_{n,i}(m,j) \underset{n\to\infty}\longrightarrow  \int_{{\s}^{\downarrow}}\sum_{\ell\in\N}s_{\ell}(1-s_{\ell})f(s_{\ell})\nu^{(i)}(\mathrm d{\mathbf{s}})
\end{equation*}
and for $j\neq i$
\begin{equation*}
\label{Hmix2}
n^{\beta}\sum_{m=0}^np_{n,i}(m,j) \underset{n\to\infty}\longrightarrow Q_{i,j}
\end{equation*}
where $Q_{i,j}$ is the $(i,j)$-th entry of the matrix $Q$ appearing in Theorem \ref{theo:melange}.
\end{lem} 

\medskip

As a consequence, the transition probabilities of $(X_n,J_n)$ satisfy the assumptions of \cite[Theorem 4.1]{HS18}, which is sufficient to describe the asymptotic behavior of the process $X_n$ but not of its absorption time. To complete, we use our Corollary \ref{cor:tempsabsorption}, which leads to the scaling limit in distribution:

\begin{cor}
\label{cor:hypok2}
\begin{equation}\label{eq:1dimwithdeathtimemix}\left(\frac{X_n(\lfloor n^{\gamma} \cdot \rfloor)}{n},\frac{D_1^{(n)}}{n^{\gamma}}\right) \overset{(d)}{\underset{n\to\infty}\longrightarrow} \big(X(\cdot),D_1\big).
\end{equation}
where $X$ is a $\gamma$-Lamperti transform of a monotype Markov additive process with characteristic $\nu:=\sum_{i=1}^{\kappa} \chi_i\nu^{(i)}$, and $D_1$ is its absorption time at 0.
\end{cor}

\medskip

\textit{Proof of Lemma \ref{lem:hypok2}.}
The proof of first convergence is essentially the same as the one of Lemma \ref{lem:hypok1}, and left to the reader.
It then remains to prove that for $j\neq i$,
\[n^{\beta}\sum_{m=0}^np_{n,i}(m,j) \underset{n\to\infty}\longrightarrow Q_{i,j}.\]
We rewrite the left-hand side as
\[n^{\beta}\sum_{\bar{\lambda}\in\overline{\mathcal{P}}_n}q_n^{(i)}(\bar{\lambda})\sum_{\ell:i_\ell=j}\frac{\lambda_\ell}{n}=n^{\beta}\sum_{\bar{\lambda}\in\overline{\mathcal{P}}_n}q_n^{(i)}(\bar{\lambda})\mathbf{1}_{\{i_1=j\}}+n^{\beta}\sum_{\bar{\lambda}\in\overline{\mathcal{P}}_n}q_n^{(i)}(\bar{\lambda})\bigg(\Big(\frac{\lambda_1}{n}-1\Big)\mathbf{1}_{\{i_1=j\}}+\sum_{\ell\geq 2:i_\ell=j}\frac{\lambda_\ell}{n}\bigg).\]
The first term converges to $Q_{i,j}$ by (\ref{hypocvmixing2}), and the second is of order $n^{\beta-\gamma}$ by (\ref{hypocvmixing}) and thus tends to $0$.
$\hfill \square$

\subsubsection{Convergence of $k$-dimensional marginals for $k\geq 2$}

Our goal is to prove Proposition \ref{prop:cvfd} by induction on $k$. For simplicity we focus on the {\bf critical case only} - the statements for the solo and mixing cases would require some modification (essentially removing information about the types of the partitions involved when $[k]$ is split) but the proofs would only get simpler, so we only provide the occasional comments and leave these details to the reader. From now we assume that  $(q_n^{(i)})$ satisfies hypotheses (\ref{hypocvcritique}) and (\ref{hypocvcritique2}), with in the limit a vector $\bnu$ of dislocation measures.  We start with two preliminary lemmas, recalling the notation and results of Section \ref{sec:prem}. The first lemma essentially relies on (\ref{hypocvcritique}) and is a straightforward adaptation of \cite[Lemma 26]{HM12} - we leave its proof to the reader. For all $n\geq 1$ and all $i \in [\kappa]$, $p_n^{(i)}$ denotes the distribution of $\overline \Pi_{n}(1)$.

\begin{lem}\label{lem:Hplus} Let $i\in [\kappa]$ and $\bar{\pi}'\in\overline{\mathcal{P}}_{[k]}$ with $b\geq 2$ blocks. Let $g:(0,\infty)^b\to \R$ be continuous with compact support, then
\[n^{\gamma}p^{(i)}_n\left(g\left(\frac{\#\pi_1}{n},\ldots,\frac{\#\pi_b}{n}\right)\mathbf{1}_{\{\bar{\pi}_{\mid [k]}=\bar{\pi}'\}}\right)\underset{n\to\infty}\longrightarrow \int_{\overline{\mathcal{P}}_{\N}}\kappa_{\bar{\nu}^{(i)}}\mathrm (d\bar{\pi})g\big(|\pi_1|,\ldots,|\pi_b|\big)\mathbf{1}_{\{\bar \pi_{\mid [k]}=\bar{\pi}'\}}.\]
\end{lem}

\noindent\text{Note:} this also holds in the solo case. In the mixing case, we would consider a partition $\pi'$ without types, and have the measure $\kappa_{\nu^{(i)}}$ in the limit.

In the following, we call $\Pi_{n,(m)}(t)$ the block of $\Pi_n(t)$ which contains the integer $m$ and $\Pi_{n,m}(t)$ the $m$-th block of $\Pi_n(t)$ when blocks are ordered by increasing least element.
\begin{lem}\label{lem:cvseparation} We have the following joint convergence in distribution:
\begin{eqnarray*}
&& \left(\frac{D_k^{(n)}}{n^{\gamma}},\overline{\Pi}_{n}(D_k^{(n)})\cap[k],\bigg(\frac{\#\Pi_{n,(m)}(D_k^{(n)})}{n},1\leq m\leq k\bigg)\right) \\
&\overset{(d)}{\underset{n\to\infty}\longrightarrow}&\Big(D_k,\overline{\Pi}(D_k)\cap [k],\left(|\Pi_{(m)}(D_k)|,1\leq m\leq k\right)\Big).
\end{eqnarray*}
\end{lem} 
\noindent\text{Note:} this also holds for the solo case. In the mixing case, we would not want to keep information about the types, so the second term on the left hand side would be $\Pi^{(n)}(D_k^{(n)})\cap[k]$, and the second term on the right-hand side $\Pi(D_k)\cap[k].$
\begin{proof}
Let $p\in\N$, we first prove a version of this stopped at the $p$-th type change of the block containing $1$. For this we use the notation $S_n(p),S(p)$ introduced in Corollary \ref{cor:hypok1}. Let $\bar \pi'\in\overline{\mathcal{P}}_{[k]}$ with $b\geq 2$ blocks and $f:(0,\infty)\to \R$, $g:(0,\infty)\to\R,$ $h:(0,\infty)^b\to\R$ be continuous functions with compact support. By Lemma \ref{lem:separationdiscret}, we have

\begin{align*}
\E&\Bigg[f\left(\frac{D_k^{(n)}}{n^{\gamma}}\right)g\left(\frac{X_n(D_k^{(n)}-1)}{n}\right) \\
&\hspace{3.3cm} h\left(\frac{\#\Pi_{n,m}(D_k^{(n)})}{X_n(D_k^{(n)}-1)},1\leq m\leq b\right),\overline{\Pi}^{(n)}(D_k^{(n)})\cap[k]=\bar{\pi}',D_k^{(n)}<S_{n}(p)\Bigg]\\
&=\sum_{r\in \mathbb N}f\left(\frac{r}{n^{\gamma}}\right)\E\Bigg[\frac{(X_n(r-1)-1)_{k-1}}{(n-1)_{k-1}}g\left(\frac{X_n(r-1)}{n}\right) \\
&\hspace{4cm} p_{X_n(r-1)}^{(J_n(r-1))}\left(h\left(\frac{\#\pi_{m}}{X_n(r-1)},1\leq m\leq b\right),\bar{\pi}\cap[k]=\bar{\pi}'\right),r<S_{n}(p)\Bigg] \\
&=\int_{n^{-\gamma}}^{\infty}\mathrm d u\, f\left(\frac{\lfloor n^{\gamma}u\rfloor}{n^{\gamma}}\right)\E\Bigg[\frac{(\mathring{X_n}(u)-1)_{k-1}}{(n-1)_{k-1}}\frac{n^{\gamma}}{(\mathring{X_n}(u))^{\gamma}}g\left(\frac{\mathring{X_n}(u)}{n}\right) \\
&\hspace{4cm} (\mathring{X_n}(u))^{\gamma}p_{\mathring{X_n}(u)}^{(\mathring{J_n}(u))}\left(h\left(\frac{\#\pi_{m}}{\mathring{X_n}(u)},1\leq m\leq b\right),\bar \pi\cap[k]=\bar \pi'\right),\lfloor n^{\gamma}u\rfloor <S_n(p)\Bigg]
\end{align*}
where $\mathring{X_n}(u):=X_n(\lfloor n^{\gamma}u\rfloor-1)$ and $\mathring{J_n}(u):=J_n(\lfloor n^{\gamma}u\rfloor-1).$ Using the convergences in distribution of Corollary \ref{cor:hypok1} and Lemma \ref{lem:Hplus}, we obtain by dominated convergence (since the functions $f,g,h$ have compact support) that this integral converges to
\[\int_0^{\infty}f(u)\mathrm d u\, \E\Bigg[\left(X(u)\right)^{k-1-\gamma} g\left(X(u)\right)\Big(\int_{\overline{\mathcal{P}}_{\N}} \kappa_{\bar{\nu}^{(J(u))}}(\mathrm d\bar{\pi})h(|\pi_m|,1\leq m \leq b)\mathbf{1}_{\{\bar \pi\cap[k]=\bar \pi'\}}\Big),u<S(p)\Bigg],\]
which by Lemma \ref{lem:fragsplit} is equal to
\[\E\left[f(D_k)g\left(|\overline{\Pi}_{(1)}(D_k -)|\right)h\left(\frac{|\Pi_m(D_k)|}{|\Pi_{(1)}(D_k -)|},1\leq m\leq b\right),\overline{\Pi}(D_k)\cap[k]=\bar \pi',D_k<S(p)\right].\]

Since $D_k$, $|\overline{\Pi}_{(1)}(D_k^-)|$ and $(|\Pi_m(D_k)|,1\leq m\leq b)$ are strictly positive on the event where $\overline{\Pi}(D_k)\cap [k]=\overline \pi',$ this classically extends to all continuous and bounded functions $f,$ $g,$ and $h$. In particular, we have $\pr(D_k^{(n)}<S_n(p))\to\pr(D_k< S(p)).$ Now let $\veps>0$. Since $S(p)\to D_1$ a.s.\ and $D_k<D_1$ a.s., there exists $p\in\N$ such that $\pr(D_k<S(p))>1-\veps,$ and $\pr(D_k^{(n)}<S_n(p))>1-\veps$ for $n$ large enough. Taking $n$ possibly larger, we then obtain
\begin{align*}
&\Bigg|\E\Bigg[f\left(\frac{D_k^{(n)}}{n^{\gamma}}\right)g\left(\frac{X_n(D_k^{(n)}-1)}{n}\right) h\left(\frac{\#\Pi^{(n)}_{m}(D_k^{(n)})}{X_n(D_k^{(n)}-1)},1\leq m\leq b\right),\overline{\Pi}^{(n)}(D_k^{(n)})\cap[k]=\bar{\pi}'\Bigg]\\
&-\E\left[f(D_k)g\left(|\overline{\Pi}_1(D_k-)|\right)h\left(\frac{|\Pi_i(D_k)|}{|\Pi_1(D_k-)|}\right),\overline{\Pi}(D_k)\cap [k]=\pi'\right]\Bigg|<2\veps,
\end{align*}
concluding the proof.
\end{proof}

With Lemma \ref{lem:cvseparation} in hand, we can complete the proof of the convergence of the finite-dimensional marginals.

\noindent\textit{End of the proof of the convergence of Proposition \ref{prop:cvfd}.} We want to prove the convergence in distribution of $n^{-\gamma} \cdot \mathcal{R}(T_n^{(i)},[k])$ to $\mathcal{R}(\T_{\gamma,{\bnu}}^{(i)},[k])$ by induction on $k$. The case where $k=1$ has already been treated in the previous section, so let us now assume $k\geq 2,$ and that the convergence of marginals with dimension at most $k-1$ has been proven. The tree $\mathcal{R}(T_n^{(i)},[k])$ can be described the following way: conditionally on $\overline{\Pi}_{n}(D_k^{(n)})\cap[k]=\bar{\pi}=\big((\pi_1,i_1),\ldots,(\pi_b,i_b)\big)$, $\mathcal{R}(T_n^{(i)},[k])$ consists in a segment with length $D_k^{(n)},$ at the end of which are grafted subtrees which, conditionally on $\big(\#\Pi_{n,m}(D_k^{(n)}),1\leq m\leq b\big),$ are independent and are copies of the trees $\mathcal{R}\big(T_{{\tiny{\#\Pi_{n,m}(D_k^{(n)})}}}^{(i_m)},[\#\pi_m]\big),$ $1\leq m\leq b$. Applying the induction hypothesis and Lemma \ref{lem:cvseparation} then shows that the initial segment and the attached subtrees, when rescaled by $n^{\gamma}$, jointly converge in distribution to a segment with length $D_k$ and trees which, conditionally on $\overline{\Pi}(D_k)\cap[k]=\bar{\pi}$ and $\left(|\Pi_{m}(D_k)|,1\leq m\leq b\right)$, are independent and are distributed as copies of the trees $|\Pi_m(D_k)|^{\gamma}\cdot \mathcal{R}\big(\T_{\gamma,{\bnu}}^{(i_m)},[\#\pi_m ]\big),1\leq m\leq b.$ However, when we remove the conditioning on $\overline{\Pi}(D_k)\cap[k],\left(|\Pi_{m}(D_k)|,1\leq m\leq b\right)$ the grafting of such trees has the same distribution as $\mathcal{R}(\T^{(i)}_{\gamma,{\bnu}},[k])$ by self-similarity, and the proof is thus ended.
\qed

\subsection{Tightness}
\label{sec:tightness}

In this section, we will always assume either the hypotheses of Theorem \ref{theo:critique} or those of Theorem \ref{theo:melange}.
The following lemma will lead to the Gromov-Hausdorff part of those theorems. 

\begin{lem}\label{lem:tightness}
For all $\eta>0$ and $i\in[\kappa],$
\[
\underset{k\to\infty}\lim\underset{n\to\infty}\limsup \,\pr\big(d_{\mathrm{GH}}(\mathcal{R}(T_n^{(i)},[k]),T_n^{(i)})\geq \eta \, n^{\gamma}\big)=0.
\]
\end{lem}

Indeed, with the notation of Theorem \ref{theo:critique} and using the classical \cite[Theorem 3.2]{Bill99}, Lemma \ref{lem:tightness} combined with the convergences of $n^{-\gamma}\mathcal{R}(T_n^{(i)},[k])$ to $\mathcal{R}(\T_{\gamma,\bnu}^{(i)},[k])$ for all $k$ as $n\rightarrow \infty$ (Proposition \ref{prop:cvfd}) and the convergence of $\mathcal{R}(\T^{(i)}_{\gamma,\bnu},[k])$ to $\T_{\gamma,\bnu}^{(i)}$ as $k\rightarrow \infty$, immediately gives the convergence in distribution of $n^{-\gamma} \cdot T_n^{(i)}$ to $\T_{\gamma,\bnu}^{(i)}$ \ in the critical and solo regimes. We conclude similarly in the mixing regime.

Lemma \ref{lem:tightness} itself hinges on Proposition \ref{Hmoments} and the following lemma, of which we will not give a proof. The reader can use the proof of its monotype analogue in \cite{HM12}  as a reference -- see in particular Lemma 32 there and the concluding lines on Page 2633.
\begin{lem}\label{lem:cvpartitionfrag}
Fix $i\in[\kappa].$ Let, for $k\in\N$ and $n\geq k+1,$ \[\Sigma_n(k+1)=\inf\left\{r\in\N:\,[k]\cap\Pi_{n,(k+1)}^{(i)}(r)=\emptyset\right\}\]
(recalling that $\Pi_{n,(k+1)}^{(i)}(r)$ denotes the block of $\Pi_{n}^{(i)}(r)$ which contains $k+1$),
and
\[\Sigma(k+1)=\inf\left\{t\geq 0:\,[k]\cap\Pi_{(k+1)}^{(i)}(t)=\emptyset\right\}.\]
Then
\[\underset{n\to\infty}\limsup\,\E\left[\frac{1}{n}\#\Pi_{n,(k+1)}^{(i)}(\Sigma_n(k+1))\right]\leq \E\left[\left|\Pi^{(i)}_{(k+1)}(\Sigma(k+1)-)\right|\right].\]
\end{lem}

\begin{proof}[Proof of Lemma \ref{lem:tightness}] Fix $i\in[\kappa],$ $k\in\N$ and $n\geq k$. Let $\bar{\pi}=(\pi,\mathbf{i})$ be the random exchangeable typed partition of $\{k+1,\ldots,n\}$ whose typed blocks are those of $\Pi^{(i)}_n(\cdot)$ when they split off of $[k].$ Specifically, its typed blocks are those of the form $\overline{\Pi}_{n,(m)}^{(i)}(l)$, $m\in\{k+1,\ldots,n\},l\in\N$ such that $\Pi_{n,(m)}^{(i)}(l)\cap [k]=\emptyset$ and $\Pi_{n,(m)}^{(i)}(l-1)\cap [k]\neq\emptyset,$ where $\overline{\Pi}_{n,(m)}^{(i)}(l)$ denotes the block of $\overline{\Pi}_{n}^{(i)}(l)$ which contains $m.$ It is then clear that $T_n^{(i)}$ can be obtained from $\mathcal{R}(T_n^{(i)},[k])$ by grafting to some vertices of the latter independent copies of $T^{(i_m)}_{\#\pi_m},$ of which we note the heights $H^{(i_m)}_{\#\pi_m},$  and thus
\begin{align*}
\pr\left(d_{\mathrm{GH}}(\mathcal{R}(T_n^{(i)},[k]),T_n^{(i)})\geq \eta \, n^{\gamma}\right)&\leq \E\left[\sum_{m\in\N}\pr\left(H_{\#\pi_m}^{(i_m)}\geq \eta\,n^{\gamma}\right)\right] \\
&\leq \E\left[\sum_{m\in\N} \E\left[\frac{\left(H_{\#\pi_m}^{(i_m)}\right)^{2/\gamma}}{\eta^{2/\gamma}n^2}\mid \bar{\pi}_m\right]\right]. 
\end{align*}
Applying Proposition \ref{Hmoments} (which holds both for the critical, solo and mixing regimes) with $p=\frac{2}{\gamma}$ yields
\[\pr\left(d_{\mathrm{GH}}\left(\mathcal{R}(T_n^{(i)},[k]),T_n^{(i)}\right)\geq \eta \, n^{\gamma}\right)\leq \frac{C}{\eta^{2/\gamma}}\E\left[\sum_{m\in\N}\frac{(\#\pi_m)^2}{n^2}\right],\]
for some finite constant $C$. Since $\pi$ is an exchangeable partition of a set with $n-k$ members, we have $\E[\# \pi_{(k+1)}]=\frac{1}{n-k}\E[\sum_{m\in\N}(\#\pi_m)^2],$ hence

\[\pr\left(d_{\mathrm{GH}}\left(\mathcal{R}(T_n^{(i)},[k]),T_n^{(i)}\right)\geq \eta \, n^{\gamma}\right)\leq \frac{C}{\eta^{2/\gamma}}\E\left[\#\frac{\pi_{(k+1)}}{n}\right].\]

Noticing that $\pi_{(k+1)}=\Pi_{n,(k+1)}^{(i)}(\Sigma_n(k+1)),$ Lemma \ref{lem:cvpartitionfrag} implies that
\[\underset{n\to\infty}\limsup\,\pr\left(d_{\mathrm{GH}}\left(\mathcal{R}(T_n^{(i)},[k]),T_n^{(i)}\right)\geq \eta \, n^{\gamma}\right)\leq \frac{C}{\eta^{2/\gamma}}\E\left[|\Pi^{(i)}_{(k+1)}(\Sigma(k+1)-)|\right].
\]
However, by exchangeability, $\Pi^{(i)}_{(k+1)}(\Sigma(k+1)-)$ has the same distribution as $\Pi^{(i)}_{(1)}(\Sigma'(k+1)-)$
where $\Sigma'(k+1)=\inf\big\{t\geq 0:\,\{2,3,\ldots,k+1\}\cap\Pi_{(1)}^{(i)}(t)=\emptyset\big\},$ and thus
\[\underset{n\to\infty}\limsup\, \pr\left(d_{\mathrm{GH}}\left(\mathcal{R}(T_n^{(i)},[k]),T_n^{(i)}\right)\geq \eta \, n^{\gamma}\right)\leq \frac{C}{\eta^{2/\gamma}}\E\left[|\Pi^{(i)}_{(1)}(\Sigma'(k+1)-)|\right].
\]
This expectation, however, tends to $0$ as $k$ tends to infinity, because $\Sigma'(k+1)$ converges a.s.\ to $D_1^{(i)}.$
\end{proof}

\subsection{Adding the measure}\label{sec:measure}

The final step in the proofs of Theorem \ref{theo:critique} and Theorem \ref{theo:melange} consists in, knowing that $n^{-\gamma} \cdot T_n^{(i)}$ converges in distribution for the Gromov-Hausdorff topology, adding the measure $\mu_n^{(i)}$ to it and proof the Gromov-Hausdorff-Prokhorov convergence. Again, the arguments are the same for both theorems but the notation is different, so we use the hypotheses and notation of Theorem \ref{theo:critique}. These arguments are in fact very similar to the corresponding monotype Section 4.4 in \cite{HM12} so we only sketch them briefly.

It follows from classical results~\cite[Lemma 2.3]{EW06} that $(n^{-\gamma}\cdot T_n^{(i)},\mu_n^{(i)})$ is tight for the GHP-topology since $n^{-\gamma} \cdot T_n^{(i)}$ converges in distribution for the GH-topology. Let us thus assume that some subsequence converges in distribution to $(\T',\mu').$ Our goal is to show that $(\T',\mu')$ is distributed as $\big(\T^{(i)}_{\gamma,\bnu}, \mu^{(i)}_{\gamma,\bnu}\big).$ Note that we already know that $\T'$ has the correct distribution.

Under our assumption (\ref{eq:simple}), $\mu_n^{(i)}$ is the uniform measure on the $n$ leaves of $T_n^{(i)}.$ Let $k\in\N$ and $L_1^n,\ldots,L_k^n$ be $k$ independent uniformly chosen leaves of $T_n^{(i)}.$ By \cite[Proposition 10]{M09} and \cite[Lemma 35]{HM12}, the subtree of $T_n^{(i)}$ spanned by the root and the leaves $L_1^n,\ldots,L_k^n,$ seen as a $k+1$-pointed metric space, has scaling limit $\T'_k$ where $\T'_k$ is the subtree of $\T'$ spanned by its root and leaves $L'_1,\ldots,L'_k$, chosen independently with distribution $\mu'.$ On the other hand, we know that, conditionally on $L_1^n,\ldots,L_k^n$ being different (an event which has probability tending to $1$), the subtree of $T_n^{(i)}$ spanned by the root and these $k$ leaves is distributed as $\mathcal{R}(T_n^{(i)},[k]),$ and so multiplied by $n^{-\gamma}$ it converges in distribution to $\mathcal{R}(\T_{\gamma,\bnu}^{(i)},[k])$ by Proposition \ref{prop:cvfd}. Hence $\T'_k$ and $\mathcal{R}(\T_{\gamma,\bnu}^{(i)},[k])$ have the same distribution for all $k$ and, clearly, this also holds jointly for all $k$.

Then, calling $\mu'_k$ the uniform distribution on $L'_1,\ldots,L'_k,$ the measured tree $(\T'_k,\mu'_k)$ converges a.s. as $k\to\infty$ to $(\T'',\mu')$ where $\T''$ is the closure in \ $\T'$ of
\ $\bigcup_{i=1}^{\infty}\llbracket\rho,L'_i\rrbracket.$
Note moreover that, by Lemma \ref{lem:tightness} and since $(n^{-\gamma} \cdot T_n^{(i)}, n^{-\gamma} \cdot\mathcal{R}(T_n^{(i)},[k]))$ converges in distribution to $(\mathcal T',\mathcal T'_k)$ along the considered subsequence, $\pr(d_{\mathrm{GH}}(\T'_k,\T')>\eta)\to 0$ as $k \rightarrow \infty$ for any $\eta>0,$ implying that $\T'=\T''$ a.s.
On the other hand, we know by Lemma \ref{lem:marginalstowholetree} that $\big(\mathcal{R}(\T_{\gamma,\bnu}^{(i)},[k]),\eta_k\big)$ converges a.s. to $(\T_{\gamma,\bnu}^{(i)},\mu^{(i)}_{\gamma,\bnu})$ as $k\to\infty,$ where we recall that $\eta_k$ is the uniform measure on the leaves of $\mathcal{R}(\T_{\gamma,\bnu}^{(i)},[k]),$ thus identifying the distribution of $(\T',\mu')$ and $(\T_{\gamma,\bnu}^{(i)},\nu^{(i)}_{\gamma,\bnu})$ and ending our proof.

\bibliographystyle{siam}
\bibliography{frag}
\end{document}